\documentclass[11pt]{article}

\usepackage{amsmath,amsthm,amssymb}
\usepackage[usenames,dvipsnames]{xcolor}
\usepackage{enumerate}
\usepackage{graphicx}
\usepackage{cite}
\usepackage{comment}
\usepackage{mathrsfs}
\usepackage{bbm}
\usepackage[margin=1in]{geometry}

\usepackage[colorlinks,
            linkcolor=blue,
            anchorcolor=green,
            citecolor=blue
            ]{hyperref}

\theoremstyle{plain}
\newtheorem{thm}{Theorem}[section]
\newtheorem{theorem}[thm]{Theorem}
\newtheorem{cor}[thm]{Corollary}
\newtheorem{lemma}[thm]{Lemma}
\newtheorem{prop}[thm]{Proposition}

\newtheorem{definition}[thm]{Definition}

\def\@rst #1 #2other{#1}
\newcommand\MR[1]{\relax\ifhmode\unskip\spacefactor3000 \space\fi
  \MRhref{\expandafter\@rst #1 other}{#1}}
\newcommand{\MRhref}[2]{\href{http://www.ams.org/mathscinet-getitem?mr=#1}{MR#2}}

\theoremstyle{definition}

\newtheorem{remark}[thm]{Remark}

\numberwithin{equation}{section} 

\newcommand{\dsb}{\begin{adjustwidth}{2.5em}{0pt}
\begin{footnotesize}}
\newcommand{\dse}{\end{footnotesize}
\end{adjustwidth}}

\newcommand{\ssb}{\begin{adjustwidth}{2.5em}{0pt}}
\newcommand{\sse}{\end{adjustwidth}}

\newcommand{\aryb}{\begin{eqnarray*}}
\newcommand{\arye}{\end{eqnarray*}}
\def\alb#1\ale{\begin{align*}#1\end{align*}}
\def\allb#1\alle{\begin{align}#1\end{align}}
\newcommand{\eqb}{\begin{equation}}
\newcommand{\eqe}{\end{equation}}
\newcommand{\eqbn}{\begin{equation*}}
\newcommand{\eqen}{\end{equation*}}

\newcommand{\frk}{\mathfrak}

\newcommand\p{\partial}
\newcommand\e{\varepsilon}

\newcommand\R{\mathbb{R}}

\newcommand\norm[1]{\left\Vert#1\right\Vert}

\let\epsilon\varepsilon
\newcommand{\Var}{{\rm Var}}
\newcommand{\len}{\operatorname{len}}

\let\originalleft\left
\let\originalright\right
\renewcommand{\left}{\mathopen{}\mathclose\bgroup\originalleft}
\renewcommand{\right}{\aftergroup\egroup\originalright}

\title{Weak exponential metrics for high-dimensional log-correlated Gaussian fields}
\author{Andres A. Contreras Hip\thanks{University of Chicago}  \qquad \qquad Zijie Zhuang\thanks{University of Pennsylvania}}
 
\begin{document}

\maketitle

\begin{abstract}
\noindent For log-correlated Gaussian fields on $\mathbb{R}^d$ with $d \geq 2$, Ding-Gwynne-Zhuang (2023) established the existence of subsequential limits of exponential metrics obtained from appropriate approximations. For $\gamma \in (0,\sqrt{2d})$, we define a \textit{weak $\gamma$-exponential metric} to be a map $h \mapsto D_h$ that assigns to a sample of a log-correlated Gaussian field $h$ a continuous metric on $\mathbb{R}^d$ satisfying a list of axioms. We prove that every subsequential limit of exponential metrics built from appropriate approximations of $h$ is a weak $\gamma$-exponential metric in this sense. Moreover, we establish general properties that hold for any weak exponential metric: (1).\ sharp moment bounds for several natural distances; (2).\ optimal H\"older exponents when comparing $D_h$ and the Euclidean metric; and (3).\ Hausdorff dimension and a KPZ relation. These results extend the two-dimensional Liouville quantum gravity metric theory to higher dimensions. Along the way we derive several useful properties for log-correlated Gaussian fields including the equivalence between white-noise decomposition and convolution, and a shell independence lemma.

\end{abstract}

\setcounter{tocdepth}{1}
\tableofcontents

\section{Introduction}

\subsection{High-dimensional Liouville first passage percolation}

\textit{Liouville first passage percolation} (LFPP) with parameter $\xi \ge 0$ is a family of continuous metrics on $\mathbb{C}$ obtained by integrating $e^{\xi h_\epsilon}$ along piecewise continuously differentiable paths, where $\{h_\epsilon\}_{\epsilon>0}$ is a continuous approximation of the \textit{Gaussian free field} (GFF) $h$. It was shown in~\cite{DDDF-tightness, gm-existence, dg-supercritical-lfpp, dg-uniqueness} that, after appropriate normalization and as $\epsilon \to 0$, LFPP converges to a limiting metric known as the Liouville quantum gravity (LQG) metric. This metric describes the conjectural scaling limits of random planar maps; see~\cite{sheffield-icm, DDG-ICM} for surveys. Recently, there has been progress toward extending LFPP theory to higher dimensions~\cite{dgz-exponential-metric}. One expects that, after appropriate normalization, the high-dimensional LFPP (defined below) also converges to a limiting metric on $\mathbb{R}^d$.

We now discuss the definition of high-dimensional LFPP. Let $d \ge 2$. A whole-space log-correlated Gaussian field (LGF) $\phi$ is a random generalized function on $\mathbb{R}^d$, defined modulo an additive constant, whose covariance is given by
\begin{equation}\label{eq:lgf-cov}
    \mathrm{Cov}\big[(\phi, f_1), (\phi, f_2)\big] = \int_{\mathbb{R}^d \times \mathbb{R}^d} f_1(x) f_2(y) \log \frac{1}{|x-y|} \, dx \,dy,
\end{equation}
for Schwartz functions $f_1,f_2$ with zero average. See~\cite{lgf-survey, fgf-survey} for background on LGFs. When $d=2$, the LGF coincides with the GFF, whereas for $d \ge 3$ they are distinct. With a slight abuse of notation, throughout this paper we say that a random generalized function $h$ on $\mathbb{R}^d$ is a \textit{whole-space LGF} if it has the same law as a whole-space LGF viewed modulo an additive constant. We specify the additive constant when needed.

Consider a radially symmetric real-valued function $\mathsf K$ on $\mathbb{R}^d$ with $\int_{\mathbb{R}^d} \mathsf K(x) \,dx = 1$ as defined in Remark~\ref{rmk:def-K}. For $\epsilon>0$, set $\mathsf K_\epsilon(x) = \epsilon^{-d}\mathsf K(\epsilon^{-1}x)$, and let $h$ be a whole-space LGF. We consider the mollification
\begin{equation}\label{eq:def-mollification}
    h_\epsilon^*(x) := h * \mathsf K_\epsilon(x) = \int_{\mathbb{R}^d} h(w) \mathsf K_\epsilon(x-w) \,dw, \qquad \epsilon>0,\ x \in \mathbb{R}^d,
\end{equation}
and the associated exponential metric
\begin{equation}\label{eq:def-exponential-metric}
    D_h^\epsilon(x,y) := \inf_{P:x \to y} \int_0^1 e^{\xi h_\epsilon^*(P(t))} |P'(t)| \, dt, \qquad x,y \in \mathbb{R}^d,
\end{equation}
where the infimum is over piecewise continuously differentiable paths $P:[0,1] \to \mathbb{R}^d$ joining $x$ and $y$. Define the normalizing constant
\begin{equation}\label{eq:def-a-epsilon}
    \mathsf a_\epsilon := \mathrm{median}(D_h^\epsilon(0,e_1)),
\end{equation}
where $e_1=(1,0,\ldots,0)$. It was shown in~\cite{dgz-exponential-metric} (combined with Proposition~\ref{kernelbaby}) that there exists $\xi_c \in (0,\infty)$ such that for all $\xi \in (0,\xi_c)$, as $\epsilon \to 0$, the family $\{\mathsf a_\epsilon^{-1} D_h^\epsilon\}_{\epsilon>0}$ is tight in the local uniform topology on $\mathbb{R}^d \times \mathbb{R}^d$, and each subsequential limit induces the Euclidean topology on $\mathbb{R}^d$. The goal of this paper is to formulate a list of axioms defining a weak exponential metric and to prove that every such subsequential limit satisfies these axioms. This is analogous to the role of~\cite{lqg-metric-estimates} in the construction of the two-dimensional LQG metric. We further expect uniqueness of the weak exponential metric up to a multiplicative constant, which would imply convergence in probability; cf.~\cite{gm-existence, dg-uniqueness} for the planar case.

\begin{remark}\label{rmk:def-K}
    Let $\mathfrak K: \mathbb{R}^d \to \mathbb{R}$ be any radially symmetric, compactly supported, smooth function satisfying $\int_{\mathbb{R}^d} \mathfrak K(x)^2 \, dx = 1$. Set
    \[
        \mathsf K=\mathcal F^{-1}\!\left(\,\zeta\mapsto \sqrt{\frac{2\pi^{d/2}}{\Gamma(d/2)}\int_{|\zeta|}^{\infty} t^{d-1}\,|\hat{\mathfrak K}(t)|^{2}\,dt}\right),
    \]
    where $\hat{\mathfrak K}$ denotes the Fourier transform of $\mathfrak K$. We refer to Lemma~\ref{lem:kernel} for properties of $\mathsf K$. In particular, $\mathsf K$ is a radially symmetric, real-valued function with $\int_{\mathbb{R}^d} \mathsf K(x) \,dx = 1$, and
    \[
         \sup_{x \in \mathbb{R}^d} \max \{ |\mathsf K(x)|,  |x|^{2d-1} |\mathsf K(x)|  \}<\infty.
    \]
    When $d$ is even, $\mathsf K$ is Schwartz. Note that $\mathsf K$ is not necessarily positive. As shown in Proposition~\ref{kernelbaby}, convolving $h$ with $\mathsf K_\epsilon$ yields a smooth field with the same law as the white-noise decomposition considered in~\cite{dgz-exponential-metric}. Therefore, applying their Theorem 1.2 yields the tightness of the metrics $\{\mathsf a_\epsilon^{-1} D_h^\epsilon\}_{\epsilon>0}$ as $\epsilon \to 0$. Indeed, we expect that the tightness result should hold for any $\mathsf K$ with sufficiently rapid decay at infinity.
\end{remark}

\subsection{Axioms for weak exponential metrics}\label{subsec:axiom}

We now define weak exponential metrics. The theory may be parameterized by $\xi \in (0,\xi_c)$, or equivalently by $\gamma \in (0,\sqrt{2d})$. The relation between $\gamma$ and $\xi$ is
\begin{equation}\label{eq:def-xi-q}
Q(\xi) = \frac{d}{\gamma} + \frac{\gamma}{2}, \quad \mbox{or equivalently} \quad  \xi = \frac{\gamma}{\mathsf d_\gamma},
\end{equation}
where $Q(\xi)$ is the distance exponent of high-dimensional LFPP as in~\cite[Proposition 1.1]{dgz-exponential-metric}\footnote{Proposition 1.1 there is stated for a fixed kernel $\mathfrak K$ in the white-noise decomposition; however, $Q(\xi)$ and $\xi_c$ are independent of that choice; see \cite[Remark 1.1]{CZ-bound}.} and $\mathsf d_\gamma$ is the Hausdorff dimension of the weak exponential metric (Corollary~\ref{cor:hausdorff}). Except in the case $d=2$, $\gamma=\sqrt{8/3}$, $\xi=1/\sqrt{6}$ (the Brownian map case~\cite{legall-uniqueness, miermont-brownian-map}), we do not know an explicit closed form relating $\gamma$ and $\xi$. We refer to~\cite{CZ-bound} for the fact that $\gamma$ and $\xi$ determine each other uniquely and for quantitative bounds between them. 

We recall some metric space notation.

\begin{definition}
Let $(X,D)$ be a metric space. For a curve $P: [a,b]\to X$, the $D$-length of $P$ is
\[
\mathrm{len}(P;D) := \sup_T \sum_{i=1}^n D\big(P(a_i), P(a_{i+1})\big),
\]
where the supremum is over all partitions $a=a_1<a_2<\cdots<a_{n+1}=b$. For $Y\subset X$, the \textit{internal metric} of $D$ on $Y$ is
\[
D(z,w;Y) := \inf_{P:z\to w} \mathrm{len}(P;D) \qquad \mbox{for all } z,w\in Y,
\]
where the infimum is over all paths $P$ in $Y$ joining $z$ and $w$. We say $(X,D)$ is a \textit{length space} if for any $z,w\in X$ and $\epsilon>0$ there exists a path from $z$ to $w$ whose $D$-length is at most $D(z,w)+\epsilon$. If a metric $D$ on an open set $U\subset\mathbb{R}^d$ induces the Euclidean topology, we say $D$ is a \textit{continuous metric}.
\end{definition}

We say a random generalized function $h$ on $\mathbb{R}^d$ is a \textit{whole-space LGF plus a continuous/bounded continuous function} if one can couple $h$ with a random continuous/bounded continuous function $f:\mathbb{R}^d\to\mathbb{R}$ such that the law of $h-f$ is that of a whole-space LGF. Let $\{h_r(z)\}_{r>0,z\in\mathbb{R}^d}$ denote the spherical average process of $h$, which admits a continuous modification by Kolmogorov’s continuity theorem (see, e.g., Section 11 of~\cite{fgf-survey}). In particular, a whole-space LGF $h$, with additive constant chosen so that $h_1(0)=0$, can be viewed as a whole-space LGF plus a bounded continuous function.

Let $\mathcal{D}'(\mathbb{R}^d)$ be the space of generalized functions on $\mathbb{R}^d$ endowed with the weak topology. For $\gamma \in (0, \sqrt{2d})$, a \textit{weak $\gamma$-exponential metric} is a measurable function $h \mapsto D_h$ from $\mathcal{D}'(\mathbb{R}^d)$ to the space of continuous metrics on $\mathbb{R}^d$ if the following holds whenever $h$ is a whole-space LGF plus a continuous function:

\begin{enumerate}[I]
\item \textbf{Length space.}\label{axiom-length} Almost surely, $(\mathbb{R}^d,D_h)$ is a length space.

\item \textbf{Locality.}\label{axiom-local} For every deterministic open set $U\subset\mathbb{R}^d$, the internal metric of $D_h$ on $U$ is measurable with respect to $h|_U$.\footnote{Suppose $h$ is a whole-space LGF plus a continuous function. For an open set $U$, the $\sigma$-algebra generated by $h|_U$ is the Borel $\sigma$-algebra generated by the random variables $(h,\phi)$ with $\phi$ smooth and compactly supported in $U$. For a closed set $K$, the $\sigma$-algebra generated by $h|_K$ is defined as $\bigcap_{\epsilon>0}\sigma(h|_{B_\epsilon(K)})$, where $B_\epsilon(K)$ is the $\epsilon$-neighborhood of $K$.}

\item \textbf{Weyl scaling.}\label{axiom-weyl} For each continuous $f\colon\mathbb{R}^d\to\mathbb{R}$, define
\begin{equation}\label{eq:weyl}
(e^{\xi f}\cdot D_h)(z,w) := \inf_{P:z\to w} \int_0^{\mathrm{len}(P;D_h)} e^{\xi f(P(t))} \,dt \qquad \mbox{for } z,w\in\mathbb{R}^d,
\end{equation}
where the infimum is over all paths from $z$ to $w$ parametrized by $D_h$-length. Then almost surely $D_{h+f} = e^{\xi f}\cdot D_h$ simultaneously for all continuous $f$.

\item \textbf{Translation invariance.}\label{axiom-translation} For each $z\in\mathbb{R}^d$ and all $x,y\in\mathbb{R}^d$, almost surely
\[
D_{h(\cdot+z)}(x,y) = D_h(x+z,y+z).
\]

\item \textbf{Tightness across scales.}\label{axiom-tight} There exist deterministic constants $\{\mathfrak c_r\}_{r>0}$ such that for any $R>0$,
the laws of the metrics
\[
\big\{\mathfrak c_r^{-1} e^{-\xi h_r(0)} D_h(r\cdot,r\cdot)\big\}_{0<r<R}
\]
are tight with respect to the local uniform topology on $\mathbb{R}^d \times \mathbb{R}^d$, and every subsequential limit is a continuous metric on $\mathbb{R}^d$. Moreover, when $h$ is a whole-space LGF plus a bounded continuous function, the same conclusion holds with $R = \infty$.
\end{enumerate}
We do not specify $\mathfrak c_r$ in Axiom~\ref{axiom-tight}, since Theorem~\ref{thm:sharp-c} shows one may take $\mathfrak c_r = r^{\xi Q}$.

Our first main result is that any subsequential limit of $\{\mathsf a_\epsilon^{-1} D_h^\epsilon\}_{\epsilon>0}$ as $\epsilon \to 0$ satisfies these axioms.

\begin{theorem}\label{thm:axiom}
Let $\gamma\in(0,\sqrt{2d})$. For every sequence $\epsilon_n \to 0$ there exists a subsequence $\mathcal E$ and a weak $\gamma$-exponential metric $D_h$ such that for any $h$ which is a whole-space log-correlated Gaussian field plus a bounded continuous function, the rescaled exponential metrics $\mathsf a_{\epsilon}^{-1} D_h^{\epsilon}$ converge in probability to $D_h$ along $\mathcal E$.
\end{theorem}

We expect that for $\gamma\in(0,\sqrt{2d})$ the weak exponential metric is unique up to a multiplicative constant, and that $\mathsf a_\epsilon^{-1} D_h^\epsilon$ converges in probability to a limiting metric as $\epsilon \to 0$. In two dimensions this was proved in~\cite{gm-existence}. We hope to adapt the arguments of~\cite{gm-existence, dg-uniqueness} to establish uniqueness and full convergence in higher dimensions in future work.

\subsection{Properties of weak exponential metrics}\label{subsec:properties}

We now derive properties of weak exponential metrics. Fix $\gamma \in (0,\sqrt{2d})$ and recall $\xi$, $Q$, and $\mathsf d_\gamma$ from~\eqref{eq:def-xi-q}. Let $D_h$ be a weak $\gamma$-exponential metric, and let $h$ be a whole-space LGF with the additive constant chosen so that $h_1(0) = 0$.

The next theorem records moment bounds for several types of distances, extending the two-dimensional results in~\cite[Theorems 1.8--1.11]{lqg-metric-estimates}.

\begin{theorem}\label{thm:moments}
Let $U\subset \mathbb{R}^d$ be connected and open. Then:
\begin{enumerate}
\item\label{moment-diam} \textbf{Diameters.} For any connected compact set $K \subset U$ that is not a singleton,
\[
\mathbb{E}\left[\left(\sup_{z,w \in K} D_h(z,w;U)\right)^p\right] < \infty \qquad \mbox{for all } p \in \left(-\infty,\frac{2 d \mathsf d_\gamma}{\gamma^2}\right).
\]

\item\label{moment-set-set} \textbf{Set-to-set distance.} For any two disjoint connected compact sets $K_1,K_2 \subset U$ that are not singletons,
\[
\mathbb{E}\big[(D_h(K_1,K_2;U))^p\big] < \infty \qquad \mbox{for all } p \in \mathbb{R}.
\]

\item\label{moment-point-set} \textbf{Point-to-set distance.} For any connected compact set $K \subset U$ that is not a singleton and any $x \in U \setminus K$,
\begin{equation}\label{eq:thm1.3-3}
\mathbb{E}\left[\big(D_h(x,K;U)\big)^p\right] < \infty \qquad \mbox{for all } p \in \left(-\infty,\frac{2 Q \mathsf d_\gamma}{\gamma}\right).
\end{equation}

\item\label{moment-point-point} \textbf{Point-to-point distance.} For any distinct $x,y \in U$,
\begin{equation}\label{eq:thm1.3-4}
\mathbb{E}\left[\big(D_h(x,y;U)\big)^p\right] < \infty \qquad \mbox{for all } p \in \left(-\infty,\frac{2 Q \mathsf d_\gamma}{\gamma}\right).
\end{equation}
\end{enumerate}
\end{theorem}

\begin{remark}
The exponents in \ref{moment-diam}, \ref{moment-point-set}, and \ref{moment-point-point} are expected to be sharp. For the set-to-set distance in \ref{moment-set-set}, we in fact obtain a superpolynomial concentration bound; see Proposition~\ref{prop:superpolynomial-cross}. Analogous moment estimates with modified exponents hold in the presence of marked-point singularities. For example, if we add an $\alpha$-singularity $\alpha \log \tfrac{1}{|\cdot-x|}$ at marked points $x$ with $\alpha<Q$, then \eqref{eq:thm1.3-3} and \eqref{eq:thm1.3-4} hold with $p < \tfrac{2 (Q-\alpha)\mathsf d_\gamma}{\gamma}$, which should be optimal.
\end{remark}

Recall $\mathfrak c_r$ from Axiom~\ref{axiom-tight}. We can in fact take $\mathfrak c_r = r^{\xi Q}$, extending the planar case~\cite[Theorem 1.9]{dg-constant}. If full scaling invariance
\[
D_h(rx,ry) = D_{h(r\cdot)+Q\log r}(x,y), \qquad r>0, \quad x,y \in \mathbb{R}^d,
\]
were known for weak exponential metrics, Theorem~\ref{thm:sharp-c} would be immediate. At present we only assume tightness across scales.

\begin{theorem}\label{thm:sharp-c}
The constants in Axiom~\ref{axiom-tight} may be chosen as $\mathfrak c_r = r^{\xi Q}$. Equivalently, for any weak exponential metric $D_h$ and any $h$ that is a whole-space LGF plus a bounded continuous function, the laws of
\[
\big\{ r^{-\xi Q} e^{-\xi h_r(0)} D_h(r\cdot,r\cdot) \big\}_{0 < r < \infty}
\]
are tight with respect to the local uniform topology on $\mathbb{R}^d \times \mathbb{R}^d$, and every subsequential limit is a continuous metric on $\mathbb{R}^d$.
\end{theorem}

We next identify the optimal H\"older exponents comparing $D_h$ with the Euclidean metric, in the spirit of~\cite[Theorem 1.7]{lqg-metric-estimates}.

\begin{theorem}\label{thm:holder-continuous}
Let $\gamma \in (0,\sqrt{2d})$ and let $D_h$ be a weak $\gamma$-exponential metric.
\begin{enumerate}
\item The identity map $(\mathbb{R}^d,|\cdot|)\to(\mathbb{R}^d,D_h)$ is almost surely locally H\"older continuous with any exponent smaller than $\xi(Q-\sqrt{2d})$, and almost surely not locally H\"older continuous with any exponent larger than $\xi(Q-\sqrt{2d})$. Note that $Q=\tfrac{d}{\gamma}+\tfrac{\gamma}{2}>\sqrt{2d}$.

\item The identity map $(\mathbb{R}^d,D_h)\to(\mathbb{R}^d,|\cdot|)$ is almost surely locally H\"older continuous with any exponent smaller than $\xi^{-1}(Q+\sqrt{2d})^{-1}$, and almost surely not locally H\"older continuous with any exponent larger than $\xi^{-1}(Q+\sqrt{2d})^{-1}$.
\end{enumerate}
\end{theorem}

Knizhnik-Polyakov-Zamolodchikov (KPZ) relation~\cite{kpz-scaling} gives a quadratic correspondence between Euclidean and quantum scaling dimensions. In two dimensions, a metric KPZ relation for the LQG metric was established in~\cite{GP-kpz}, relating the Euclidean Hausdorff dimension of a set to its Hausdorff dimension with respect to the LQG metric. We prove the metric KPZ relation for weak exponential metrics in higher dimensions. For $X \subset \mathbb{R}^d$, write ${\rm dim}_{\mathcal H}^0 X$ for its Euclidean Hausdorff dimension and ${\rm dim}_{\mathcal H}^\gamma X$ for its $D_h$-Hausdorff dimension.

\begin{theorem}[KPZ relation for weak exponential metrics]\label{thm:kpz}
Let $\gamma \in (0,\sqrt{2d})$ and let $D_h$ be a weak $\gamma$-exponential metric. If $X \subset \mathbb{R}^d$ is a deterministic Borel set or a random Borel set independent of $h$, then almost surely
\begin{align*}
&{\rm dim}_{\mathcal H}^0 X = \xi Q \,{\rm dim}_{\mathcal H}^\gamma X - \frac{\xi^2}{2}\big({\rm dim}_{\mathcal H}^\gamma X\big)^2,\\
&{\rm dim}_{\mathcal H}^\gamma X = \frac{1}{\xi}\bigl(Q - \sqrt{Q^2 - 2 {\rm dim}_{\mathcal H}^0 X}\bigr).
\end{align*}
\end{theorem}

Taking $X=\mathbb{R}^d$ in Theorem~\ref{thm:kpz} yields:

\begin{cor}\label{cor:hausdorff}
Let $\gamma \in (0,\sqrt{2d})$ and let $D_h$ be a weak $\gamma$-exponential metric. Then the Hausdorff dimension of $\mathbb{R}^d$ with respect to $D_h$ is almost surely $\mathsf d_\gamma$.
\end{cor}

\subsection{Useful results about log-correlated Gaussian fields}

One of the key ingredients in the proof of Theorem~\ref{thm:axiom} is the equivalence between the white-noise decomposition of the LGF and convolution with an explicit family of kernels. This extends Lemma 2.4 in~\cite{chg-support} which is for the heat kernel.


Suppose $\mathfrak{K} : \mathbb{R}^d \to \mathbb{R}$ is a radially symmetric, compactly supported, smooth function satisfying $\int_{\mathbb{R}^d} \mathfrak{K}(x)^2 \,dx = 1$. Let $W(dy,dt)$ denote the space-time white noise on $\mathbb{R}^d \times (0,\infty)$. For $\epsilon>0$, define 
\begin{equation}\label{eq:white-noise}
\mathfrak h_\epsilon(x) := \int_{\mathbb{R}^d} \int_{\epsilon}^\infty \mathfrak K(\frac{x-y}{t}) t^{-\frac{d+1}{2}} \,W(dy,dt), \quad x \in \mathbb{R}^d.
\end{equation}
Then $\mathfrak h_\epsilon$ can be defined as a random continuous function modulo a global additive constant (see, e.g., Section 4.1.1 of~\cite{lgf-survey}). Recall the kernel $\mathsf K$ from Remark~\ref{rmk:def-K}, which can be obtained explicitly from $\mathfrak{K}$.

\begin{prop}\label{kernelbaby}
    Let $h$ be a whole-space LGF. For any $\epsilon>0$ and $x \in \mathbb{R}^d$, let $\mathsf K_\epsilon(x) = \epsilon^{-d} \mathsf K(x/\epsilon)$. Then $\mathfrak h_\epsilon$ agrees in law with $h * \mathsf K_\epsilon$, viewed as a random continuous function modulo a global constant.
\end{prop}
See Lemma \ref{lem:kernel} for properties of $\mathsf{K}$ and for a proof.

\subsection{Proof strategy}

The proof adapts the two-dimensional arguments of~\cite{lqg-metric-estimates, local-metrics, Pfeffer-weak-metric, GP-kpz, dg-constant} with modifications to handle higher dimensions. We highlight three technical obstacles and explain how we address them.

First, the approximation scheme in~\cite{dgz-exponential-metric} is given by the white-noise decomposition, which is not convenient for verifying the axioms, in particular Axiom~\ref{axiom-local} (locality). To circumvent this, in Proposition~\ref{kernelbaby} we show that the same approximation can be realized as convolution of $h$ with an explicit family of kernels, following computations in~\cite{chg-support}. This representation allows us to adapt the arguments of~\cite{lqg-metric-estimates, local-metrics} to prove Theorem~\ref{thm:axiom}.

Second, the annulus independence lemma~\cite{MQ18-geodesic, gm-existence}, which is used to propagate regularity events and control the LQG metric (see~\cite[Section 4.2]{DDG-ICM}), is specific to the two-dimensional GFF. We therefore prove an analogue for the LGF, namely near independence across shells (Lemma~\ref{lem:shell-independence}). The proof uses the domain Markov property for the LGF based on $d/2$-harmonic extension as described in~\cite{fgf-survey}. Such near independence is used repeatedly throughout this paper and may be of independent interest for the study of the LGF. Some care is also needed to deal with the spherical average process of the LGF, which does not have the same law as Brownian motion.

Third, planar arguments that force intersections of macroscopic paths do not extend to higher dimensions. In particular, the distance around an annulus used, for example, in~\cite[Definition 3.7]{lqg-metric-estimates} does not make sense in higher dimensions, since paths do not disconnect shells. We therefore use the notion of distance around a shell (see~\eqref{eq:def-around}) and use it to build stringing arguments that replace planar intersection arguments.

The organization of the paper is as follows. Section~\ref{sec:lgf} records properties and estimates for log-correlated Gaussian fields, Section~\ref{sec:axiom} proves Theorem~\ref{thm:axiom}, Section~\ref{sec:moment} establishes the properties of weak exponential metrics and proves all results stated in Section~\ref{subsec:properties}.

\subsection{Basic notation}

\subsubsection*{Numbers}

For quantities $a=a(r)$ and $b=b(r)$ depending on a parameter $r$, we write $a=o_r(b)$ (respectively $a=O_r(b)$) if $a/b \to 0$ (respectively $a/b$ remains bounded) as $r \to \infty$ or $r \to 0$, as indicated by context. We write $a \asymp_r b$ if $a =O_r(b)$ and $b = O_r(a)$. For a family of events $\{E^\epsilon\}_{\epsilon>0}$, we say that $E^\epsilon$ occurs with \textit{polynomially high probability} if there exists $p>0$ such that $1-\mathbb{P}[E^\epsilon]=O(\epsilon^p)$ as $\epsilon \to 0$. It occurs with \textit{superpolynomially high probability} if this holds for every $p>0$.

\subsubsection*{Euclidean space}

Throughout we work in $\mathbb{R}^d$ with $d \ge 2$. We write $|\cdot|$ for the Euclidean norm. For $x \in \mathbb{R}^d$ and $r>0$, let $B_r(x)=\{z \in \mathbb{R}^d: |z-x|<r\}$. For $0<r<R<\infty$ define the shell
\[
A_{r,R}(x) := B_R(x) \setminus \overline{B_r(x)}.
\]
The Fourier transform is
\[
\mathcal F(f)(\zeta)=\hat f(\zeta) := \int_{\mathbb{R}^d} e^{-2\pi i x\cdot \zeta} f(x)\, dx \qquad \mbox{for } \zeta \in \mathbb{R}^d.
\]

\subsubsection*{Metrics}

Let $D$ be a metric on $\mathbb{R}^d$. For a shell $A=A_{r,R}(x)$, the \textit{$D$-distance across $A$} is the $D$-distance between its boundary components, namely $D(\partial B_r(x),\partial B_R(x))$. The \textit{$D$-distance around $A$} is
\begin{equation}\label{eq:def-around}
\mbox{$D$(around $A$)} := \sup_{P_1,P_2} D\big(P_1,P_2;A\big),
\end{equation}
where the supremum is over all continuous paths $P_1,P_2 \subset A$ that each join the inner to the outer boundary, and $D(\cdot,\cdot;A)$ denotes the internal $D$-distance on $A$. Both quantities are determined solely by the internal $D$-metric on $A$.

\medskip

\noindent\textbf{Acknowledgement.} The authors thank Ewain Gwynne for suggesting the problem and for useful discussions. Z.Z. was partially supported by NSF grant DMS-1953848.
 
\section{Properties of log-correlated Gaussian fields}\label{sec:lgf}

In Section~\ref{subsec:basic-lgf}, we review basic properties of log-correlated Gaussian fields (LGF). In Section~\ref{subsec:shell-independence}, we prove the shell independence lemma for the LGF, which is the higher-dimensional analog of the annulus independence lemma (see Section 4.2 of~\cite{DDG-ICM}). In Section~\ref{subsec:kernel}, we show that the white-noise approximations of the log-correlated Gaussian field have the same law as convolutions with an explicit family of kernels. In Section~\ref{subsec:localize}, we show that for this family of kernels, the convolution approximation can be localized.

\subsection{Basic properties of log-correlated Gaussian fields}\label{subsec:basic-lgf}

We first review the background on \textit{log-correlated Gaussian fields} (LGF). Most of the material in this section is from~\cite{fgf-survey} and Section 2 of~\cite{JSW-log}. We also refer to~\cite{sheffield-gff} for the two-dimensional case, which is the Gaussian free field. We let
$$
d \geq 2 \quad \mbox{and} \quad s = \frac{d}{2} \,.
$$
Let $\mathcal{S}(\mathbb{R}^d)$ be the space of Schwartz functions on $\mathbb{R}^d$. Define $\mathcal{S}_0(\mathbb{R}^d) \subset \mathcal{S}(\mathbb{R}^d)$ as the subspace of Schwartz functions $\phi$ satisfying $\hat \phi(0) = 0$, or equivalently, $\int_{\mathbb{R}^d} \phi(x) dx = 0$. The space $\mathcal{S}'(\mathbb{R}^d)$ of generalized functions is the dual space of $\mathcal{S}(\mathbb{R}^d)$. Similarly, $\mathcal{S}_0'(\mathbb{R}^d)$ is the dual space of $\mathcal{S}_0(\mathbb{R}^d)$, i.e., the space of generalized functions defined modulo an additive constant. For each $f \in \mathcal{S}(\mathbb{R}^d)$ (resp.\ $f \in \mathcal{S}_0(\mathbb{R}^d$)) and $a>-\frac{d}{2}$ (resp.\ $a > -\frac{d+1}{2}$), we define
$$
(-\Delta)^a f:= \mathcal{F}^{-1}[\xi \to (2 \pi |\xi|)^{2a} \hat f(\xi)].
$$
This is well-defined since $|\xi|^{2a} \hat f(\xi) \in L^1(\mathbb{R}^d)$.

We first consider the whole-space LGF, which is a random generalized function defined modulo an additive constant. For all $f,g \in \mathcal{S}_0(\mathbb{R}^d)$, define the inner product
$$
(f, g)_{\mathcal{H}^{-s}(\mathbb{R}^d)} := \frac{1}{k_d} (f, (-\Delta)^{-s} g)_{L^2(\mathbb{R}^d)} = \iint_{\mathbb{R}^d \times \mathbb{R}^d} f(x) g(y) \log \frac{1}{|x-y|} dxdy,
$$
where \[k_d = \frac{2^{1-d} \pi^{-d/2}}{\Gamma(\frac{d}{2})}.\] Let $\mathcal{H}^{-s}(\mathbb{R}^d)$ be its Hilbert space closure. A whole-space LGF $h$ is a random element in $\mathcal{S}_0'(\mathbb{R}^d)$ such that
$$
(h, \phi) \sim \mathcal{N}(0, \Vert \phi \Vert_{\mathcal{H}^{-s}(\mathbb{R}^d)}^2) \quad \mbox{for $\phi \in \mathcal{H}^{-s}(\mathbb{R}^d)$}.
$$
We refer to~\cite{fgf-survey} for well-definedness. Throughout this paper, with slight abuse of notation, we call a random generalized function $h$ a whole-space LGF if it has the same law as a whole-space LGF viewed modulo an additive constant. We specify the additive constant when needed. Let $\{h_r(z)\}_{z\in\mathbb{R}^d,\ r>0}$ denote the spherical average process of $h$, which admits a continuous modification; see Section~11 of~\cite{fgf-survey}. On many occasions, we choose the global constant of $h$ so that $h_1(0)=0$.

Next, we consider the LGF defined on an allowable subdomain of $\mathbb{R}^d$ (not necessarily bounded). For an open domain $D \subset \mathbb{R}^d$, let $C_c^\infty(D)$ denote the space of smooth functions on $\mathbb{R}^d$ that are compactly supported in $D$. Let $\mathcal{H}_0^s(D)$ denote the Hilbert space closure of $C_c^\infty(D)$ equipped with the inner product \[(f, g)_{\mathcal{H}^s(\mathbb{R}^d)} := k_d \, (f, (-\Delta)^{s} g)_{L^2(\mathbb{R}^d)}.\]

\begin{definition}[Allowable domain]\label{def:allowable}
    We call an open domain $D \subset \mathbb{R}^d$ \textit{allowable} if, for all $\phi \in \mathcal{S}(\mathbb{R}^d)$, there exists $C = C(D, \phi) > 0$ such that
\begin{equation}\label{eq:allowable}
|(\phi, g)_{L^2(\mathbb{R}^d)}| \leq C \Vert g \Vert_{\mathcal{H}^s(\mathbb{R}^d)} \quad \mbox{for all }g \in \mathcal{H}_0^s(D).
\end{equation}
By~\cite[Lemma 4.2]{fgf-survey}, an open domain $D$ such that $\mathbb{R}^d \setminus D$ contains an open set is allowable.
\end{definition}

When $D$ is allowable, define the seminorm $\Vert \cdot \Vert_{\mathcal{H}^{-s}(D)}$ on $\mathcal{S}(\mathbb{R}^d)$ as in~\cite{fgf-survey}. Specifically, for $\phi \in \mathcal{S}(\mathbb{R}^d)$, by~\eqref{eq:allowable} and the Riesz representation theorem there exists a unique $f \in \mathcal{H}_0^s(D)$ such that $(\phi,g)_{L^2(\mathbb{R}^d)} = (f,g)_{\mathcal{H}^s (\mathbb{R}^d)}$ for all $g \in \mathcal{H}^s_0(D)$. Define $\Vert \phi \Vert_{\mathcal{H}^{-s}(D)}:= \Vert f \Vert_{\mathcal{H}^s(\mathbb{R}^d)}$. Note that if $\phi$ is supported in $\mathbb{R}^d \setminus D$, then $f = 0 $ and $\Vert \phi \Vert_{\mathcal{H}^{-s}(D)} = 0$. Let $\mathcal{H}^{-s}(D)$ denote the Hilbert space closure of $\mathcal{S}(\mathbb{R}^d)$ equipped with $\Vert \cdot \Vert_{\mathcal{H}^{-s}(D)}$. A LGF $h$ on $D$ is a random element in $\mathcal{S}'(\mathbb{R}^d)$ such that
$$
(h,\phi) \sim \mathcal{N}(0, \Vert \phi \Vert^2_{\mathcal{H}^{-s}(D)}) \quad \mbox{for } \phi \in \mathcal{H}^{-s}(D).
$$
We again refer to~\cite{fgf-survey} for well-definedness.

We note that when $d$ is even, $(-\Delta)^s$ is a local operator and thus (with $\nabla^s$ understood as the Fourier multiplier by $(2\pi|\xi|)^s$)
\begin{equation}\label{eq:norm-even-dimen}
\Vert f \Vert^2_{\mathcal{H}^s(\mathbb{R}^d)} = k_d \, \int_{D} |\nabla^s f(x)|^2 dx \quad \mbox{for all }f \in C_c^\infty(D).
\end{equation}
However, this does not hold when $d$ is odd. In the two-dimensional case, the LGF on $D$ coincides with the  Gaussian free field on $D$ as considered in~\cite{sheffield-gff}.

There is another equivalent way to define the LGF using eigenfunctions; see Section 9 of~\cite{fgf-survey} and Section 2 of~\cite{JSW-log}. Let $D$ be a bounded domain (hence allowable). Let $\{f_n\}_{n \geq 1}$ denote an orthonormal basis of $\mathcal{H}_0^s(D)$ and define the LGF on $D$ as the formal sum
\begin{equation}\label{eq:eigenvalue}
h = \sum_{n=1}^\infty \alpha_n f_n,
\end{equation}
where $\{\alpha_n\}_{n \geq 1}$ are i.i.d.\ standard Gaussian random variables. Then the above sum converges a.s. in $H^{-\epsilon}(\mathbb{R}^d)$ for any $\epsilon>0$. 

As follows from the definition, the log-correlated Gaussian field has a scaling invariance property.

\begin{lemma}\label{lem:scale-invariance}
    The following holds:

    \begin{enumerate}
        \item Let $h$ be a whole-space LGF with the additive constant chosen so that $h_1(0) = 0$. Then, for any $r>0$, $h(r \cdot) - h_r(0)$ has the same law as $h(\cdot)$.

        \item Let $D$ be an allowable domain and $r>0$. Let $h_D$ be a LGF on $D$ and $h_{rD}$ be a LGF on $rD$. Then $h_D(\cdot)$ and $h_{rD}(r \cdot)$ have the same law.
    \end{enumerate}
    
\end{lemma}

Given a domain $D$ and a generalized function $f$ defined on $\mathbb{R}^d \setminus D$, we say that a generalized function $g : \mathbb{R}^d \to \mathbb{R}$ is the \textit{$s$-harmonic extension of $f$} if it satisfies
\begin{align*}
    f|_{\mathbb{R}^d \setminus D} = g|_{\mathbb{R}^d \setminus D} \quad \mbox{and} \quad ((-\Delta)^s g)|_D = 0.
\end{align*}

The following domain Markov property is taken from Section 5 of~\cite{fgf-survey}.

\begin{lemma}[Domain Markov property]\label{lem:markov}
    Let $U \subset D$ be allowable domains. Let $h$ be a LGF on $D$, and let $h^U$ be the $s$-harmonic extension of $h|_{D \setminus U}$ into $U$. Then $h - h^U$ has the same law as a LGF on $D$ and is independent of $h|_{D \setminus U}$.
\end{lemma}

We also have the domain Markov property for the whole-space LGF, again from Section 5 of~\cite{fgf-survey}.

\begin{lemma}[Domain Markov property for the whole-space case]\label{lem:markov-whole}
    Fix $z \in \mathbb{R}^d$ and $r>0$. Let $h$ be a whole-space LGF with the additive constant so that $h_r(z) = 0$. For each allowable domain $V \subset \mathbb{R}^d$, let $h^V$ denote the $s$-harmonic extension of $h|_{\mathbb{R}^d \setminus V}$ into $V$. Then, $h - h^V$ has the same law as a LGF on $V$ minus its average over $\partial B_r(z)$ and is independent of $h|_{\mathbb{R}^d \setminus V}$.
\end{lemma}

\begin{proof}
    By Section 5 of~\cite{fgf-survey}, the whole-space LGF $h$ has a decomposition $h_1 + h_2$. Here, $h_1$ is a random generalized function defined up to a global constant, which is $s$-harmonic in $V$ and determined by $h|_{\mathbb{R}^d \setminus V}$. The field $h_2$ has the same law as a LGF on $V$ and is independent of $h|_{\mathbb{R}^d \setminus V}$. The lemma follows by subtracting from both sides the spherical average over $\partial B_r(z)$.
\end{proof}

Using the eigenfunction definition, one can establish the Cameron-Martin property of LGF; see Section 2 of~\cite{JSW-log}.

\begin{prop}\label{prop:cameron}
    Let $D$ be a bounded open domain and $h$ be a LGF on $D$. Fix $f \in \mathcal{H}^s_0(D)$. Let $\mathbb{P}$ denote the law of $h$ and $\mathbb{P}_f$ denote the law of $h + f$. Then, $\mathbb{P}$ and $\mathbb{P}_f$ are absolutely continuous with respect to each other. Furthermore, the Radon-Nikodyn derivative is given by
    $$
    \frac{d\mathbb{P}}{d\mathbb{P}_f}(h) = \exp\left(\frac{1}{2} \Vert f \Vert^2_{\mathcal{H}^s(\mathbb{R}^d)} - k_d \,(h, (-\Delta)^s f) \right)\,.
    $$
    In particular, for any $\alpha \in \mathbb{R}$,
    $$\mathbb{E}\left[\left(\frac{d\mathbb{P}}{d\mathbb{P}_f}(h)\right)^\alpha\right] = \exp\left(\frac{\alpha^2 + \alpha}{2} \Vert f \Vert_{\mathcal{H}^s(\mathbb{R}^d)}^2\right) \quad \mbox{and} \quad \mathbb{E}_f\left[\left(\frac{d\mathbb{P}}{d\mathbb{P}_f}(h)\right)^\alpha\right] = \exp\left(\frac{\alpha^2 - \alpha}{2} \Vert f \Vert_{\mathcal{H}^s(\mathbb{R}^d)}^2\right)\,.$$
\end{prop}

\begin{proof}
    By~\eqref{eq:eigenvalue} the law of $h$ is the standard centered Gaussian measure on the Hilbert space $\mathcal H_0^s(D)$. For any $f\in H_0^s(D)$, the Cameron-Martin theorem for Gaussian measures on Hilbert spaces therefore yields
    \[
    \frac{d\mathbb{P}}{d\mathbb{P}_f}(h)
    = \exp\Big(\tfrac12 \|f\|_{\mathcal H_0^s(D)}^2 - (h,f)_{\mathcal H_0^s(D)}\Big)
    = \exp\Big(\tfrac12 \|f\|_{\mathcal{H}^s(\mathbb{R}^d)}^2 - k_d\,(h,(-\Delta)^s f)\Big),
    \]
    which is exactly the desired formula. The moment identities follow from the fact that
    $$
    \mathbb{E} \left[ \left(h, (-\Delta)^s f \right)^2 \right] = \Vert (-\Delta)^s f \Vert^2_{\mathcal{H}^{-s}(D)} = \frac{1}{k_d^2} \Vert f \Vert_{\mathcal{H}^s(\mathbb{R}^d)}^2. \qedhere
    $$
\end{proof}

\subsection{Shell independence lemma}\label{subsec:shell-independence}

The annulus independence lemma for two-dimensional GFF, introduced in~\cite{MQ18-geodesic, local-metrics}, is a useful tool to study LQG metric; see Section 4.2 of~\cite{DDG-ICM}. It says that the restrictions of the GFF to disjoint concentric annuli (viewed modulo additive constant) are nearly independent. In this section, we prove the higher-dimensional analog for the LGF: the restrictions of the LGF to disjoint concentric shells (viewed modulo additive constant) are also nearly independent.

Let $h$ be a whole-space LGF. For $r>0$, define the $\sigma$-algebra \[\mathcal{F}_r := \sigma\left((h - h_r(0))|_{\mathbb{R}^d \setminus B_r(0)}\right).\] Then, for $r'>r$, we have $\mathcal{F}_{r'} \subset \mathcal{F}_r$. Note that $\widetilde h:= h- h_r(0)$ has the law as a whole-space LGF with the additive constant chosen so that the average over $\partial B_r(0)$ is 0. Let $\widetilde{\mathfrak{h}}^r$ denote the $s$-harmonic extension of $\widetilde h|_{\mathbb{R}^d \setminus B_r(0)}$ into $B_r(0)$. Then, by Lemma~\ref{lem:markov-whole}, $\widetilde h - \widetilde{\mathfrak{h}}^r$ has the same law as a LGF on $B_r(0)$ and is independent of $\mathcal{F}_r$. For $M>0$, $r>0$, and $q \in (0,1)$, let
\begin{equation}\label{eq:sec2-def-G}
\mathcal{G}_{qr}^r(M):= \Big{\{} \sup_{u,v \in B_{qr}(0)} |{\rm D}^j \widetilde{\mathfrak{h}}^r(u) - {\rm D}^j \widetilde{\mathfrak{h}}^r(v)| \leq M \quad \mbox{for all $0 \leq j \leq \lceil s \rceil $} \Big{\}},
\end{equation}
where $\lceil s \rceil$ denotes the smallest integer not less than $s$. The following property of the LGF is a consequence of the Cameron-Martin property (Proposition~\ref{prop:cameron}).

\begin{lemma}\label{lem:absolute-continuity}
    (Absolute continuity) Fix $0 < q < q' < 1$. For any $M>0$ and $\alpha>0$, there exists a constant $C = C(q, q', \alpha, M)>0$ such that the following holds. Let $h$ be a whole-space LGF. On the event $\mathcal{G}_{q'r}^r(M)$, conditioned on $\mathcal{F}_r$, the law of $(h - h_r(0))|_{B_{qr}(0)}$ is absolutely continuous with respect to its marginal law. Let $H$ denote the corresponding Radon-Nikodyn derivative. Then, on the event $\mathcal{G}_{q'r}^r(M)$,
    $$
        \max\{ \mathbb{E}[H^{-\alpha}\,|\,\mathcal{F}_r], \mathbb{E}[H^\alpha\,|\,\mathcal{F}_r] \} \leq C.
    $$
\end{lemma}

\begin{proof}

Let $\hat q = (q + q')/2$. Fix a bump function $\psi$ that equals 1 in $B_q(0)$ and is supported in $B_{\hat q}(0)$, and define $\psi_r(x):= \psi(x/r)$ for $x \in \mathbb{R}^d$. We claim that for any function $g:\mathbb{R}^d \to \mathbb{R}$ that is $s$-harmonic in $B_r(0)$ and satisfies 
\begin{equation}\label{eq:lem2.6-cond}
    \sup_{u,v \in B_{q'r}(0)}|{\rm D}^j g(u) - {\rm D}^j g(v)| \leq M \quad \mbox{for all $0 \leq j \leq \lceil s \rceil$},
\end{equation}we have
\begin{equation}\label{eq:lem2.6-claim}
\Vert \psi_r \cdot g \Vert_{\mathcal{H}^s(\mathbb{R}^d)}^2 \leq C 
\end{equation}
for some constant $C = C(\psi, q,q', M)>0$. The lemma then follows by applying Proposition~\ref{prop:cameron} with $D = B_{\hat qr}(0)$ and $f = \psi_r \cdot \widetilde{\mathfrak{h}}^r$.

We now prove~\eqref{eq:lem2.6-claim}. When $d$ is even and hence $s$ is an integer, Claim~\eqref{eq:lem2.6-claim} follows from~\eqref{eq:lem2.6-cond} and~\eqref{eq:norm-even-dimen}. Indeed, by~\eqref{eq:lem2.6-cond} and the chain rule, we have $|{\rm D}^s(\psi_r \cdot g)| \leq C r^{-s}$ on $B_{\hat q r}(0)$. Combining this with~\eqref{eq:norm-even-dimen} yields Claim~\eqref{eq:lem2.6-claim}. 

When $d$ is odd and hence $s$ is not an integer, \eqref{eq:lem2.6-cond} implies
$$\big{|}{\rm D}^{s-1/2}(\psi_r \cdot g)\big{|} \leq C r^{-s+1/2} \quad \mbox{and} \quad \big{|}{\rm D}^{s+1/2}(\psi_r \cdot g)\big{|} \leq C r^{-s-1/2} \quad \mbox{on $B_{\hat q r}(0)$}.$$
By interpolation (or H\"older's inequality in Fourier space),
$$
\Vert \psi_r \cdot g \Vert_{\mathcal{H}^s(\mathbb{R}^d)}^2 \lesssim \Vert \psi_r \cdot g \Vert_{\mathcal{H}^{s-1/2}(\mathbb{R}^d)} \cdot \Vert \psi_r \cdot g \Vert_{\mathcal{H}^{s+1/2}(\mathbb{R}^d)} \leq C,
$$
which proves~\eqref{eq:lem2.6-claim}. \qedhere

\end{proof}

We can now state the shell independence lemma for the LGF.

\begin{lemma}\label{lem:shell-independence}
    Let $h$ be a whole-space LGF. Fix $q \in (0,1)$ and $q'>1$. Let $\{r_k\}_{k \geq 1}$ be a decreasing sequence of positive numbers such that $r_{k+1}/r_k \leq q$ for all $k \geq 1$. Let $\{E_{r_k}\}_{k \geq 1}$ be a sequence of events measurable with respect to $\sigma\big((h - h_{r_k}(0))|_{A_{r_k/q', q'r_k}(0)}\big)$. For $K \geq 1$, let $\mathcal{N}(K)$ be the number of $k \in [1,K] \cap \mathbb{Z}$ for which $E_{r_k}$ occurs.
    \begin{enumerate}
        \item For any $a>0$ and $b \in (0,1)$, there exists $p = p(a,b,q,q') \in (0,1)$ such that if $\mathbb{P}[E_{r_k}] \geq p$ for all $k \geq 1$, then
        $$
        \mathbb{P}[\mathcal{N}(K) \geq  bK] \geq 1 - e^{-aK} \quad \mbox{for all $K \geq 1$}.
        $$
        \item For each $p \in (0,1)$, there exists $a>0$ and $b \in (0,1)$ depending only on $p,q,q'$ such that the above relation holds.
    \end{enumerate}
\end{lemma}

Lemma~\ref{lem:shell-independence} is a consequence of the following lemma together with Lemma~\ref{lem:absolute-continuity}.

\begin{lemma}\label{lem:2.8}
    Fix $0<q<1$. Let $\{r_k\}_{k \geq 1}$ be a decreasing sequence of positive numbers such that $r_{k+1}/r_k \leq q$ for all $k \geq 1$. For $K \geq 1$ and $M>0$, let $\mathcal{L}(K, M)$ be the number of $k \in [1,K] \cap \mathbb{Z}$ for which $\mathcal{G}^{r_k}_{qr_k}(M)$ occurs. For each $a>0$ and $b \in (0,1)$, there exists $M = M(a,b,q) >0 $ such that
    $$
        \mathbb{P}[\mathcal{L}(K, M) \geq bK] \geq 1 - e^{-aK} \quad \mbox{for all $K \geq 1$}.
    $$
\end{lemma}

\begin{proof}
    The proof is identical to the argument in Proposition 4.3 of~\cite{MQ18-geodesic}. The main inputs there are Lemma 4.4 and the domain Markov property for the GFF, both of which extend to the LGF. Specifically, Lemma 4.4 of~\cite{MQ18-geodesic} requires Gaussian tail bounds for the oscillation of the harmonic extension. In the case of the LGF, we can show that there exists $c = c(q)>0$ such that for all $r >0$ and $M>0$,
    \begin{equation}\label{eq:lem2.8-1}
    \mathbb{P}\left[ \mathcal{G}^r_{qr}(M)^c \right] \leq c^{-1} e^{-cM^2}.
    \end{equation}
    In fact, for $u,v \in B_r(0)$, ${\rm D}^j \widetilde{\mathfrak{h}}^r(u) - {\rm D}^j \widetilde{\mathfrak{h}}^r(v)$ is a centered smooth Gaussian process in $(u,v)$ (see \cite[Theorem 8.3]{fgf-survey}), and~\eqref{eq:lem2.8-1} then follows from the Borel-TIS inequality \cite{itsjustborell,sudakov}. Using the scaling invariance of $h$ (Lemma~\ref{lem:scale-invariance}), we see that the constant $c$ can be chosen independent of $r$. 
    Lemmas~\ref{lem:markov} and~\ref{lem:markov-whole} give the desired Markov property for the LGF. Thus, we can follow the argument in Proposition 4.3 of~\cite{MQ18-geodesic} verbatim. Finally, although Proposition 4.3 of~\cite{MQ18-geodesic} includes an additional constant $c_0(a,b)$, we can absorb it into $a$ by choosing $M$ sufficiently large.
\end{proof}

\begin{proof}[Proof of Lemma~\ref{lem:shell-independence}]
    The proof follows the argument in Section 3 of~\cite{local-metrics}. The main inputs are Lemmas 3.3 and 3.4 therein, which are generalized by Lemmas~\ref{lem:absolute-continuity} and~\ref{lem:2.8}, respectively.
\end{proof}

\subsection{Different approximations of the LGF}\label{subsec:kernel}

In this section, we prove Proposition \ref{kernelbaby}, more precisely the version below (Lemma~\ref{lem:kernel}) which implies it. 

We assume that
\begin{equation}\label{eq:asmp-kernel}
\mbox{$\mathfrak K$ is a radially symmetric, compactly supported, smooth function satisfying $\int_{\mathbb{R}^d} \mathfrak K(x)^2 dx = 1$.}
\end{equation}
Note that such a function $\mathfrak K$ is automatically Schwartz. Let $W(dy,dt)$ denote the space-time white noise on $\mathbb{R}^d \times (0,\infty)$. For $\epsilon>0$, recall that 
\begin{equation*}
\mathfrak h_\epsilon(x) = \int_{\mathbb{R}^d} \int_{\epsilon}^\infty \mathfrak K(\frac{x-y}{t}) t^{-\frac{d+1}{2}} W(dy,dt), \quad x \in \mathbb{R}^d.
\end{equation*}
Then $\mathfrak h_\epsilon$ can be defined as a random continuous function modulo a global additive constant.

We now state a more precise version of Proposition \ref{kernelbaby}.

\begin{lemma}\label{lem:kernel}
    Let $h$ be a whole-space LGF. For any $\mathfrak K$ satisfying~\eqref{eq:asmp-kernel}, set
    \[
        \mathsf K=\mathcal F^{-1}\!\left(\,\zeta\mapsto \sqrt{\frac{2\pi^{d/2}}{\Gamma(d/2)}\int_{|\zeta|}^{\infty} t^{d-1}\,|\hat{\mathfrak K}(t)|^{2}\,dt}\right),
    \]
    which is a radially symmetric real-valued function with $\int_{\mathbb{R}^d} \mathsf K(x) \, dx = 1$ and satisfies
    \begin{equation} \label{eq:K-tail}
         \sup_{x \in \mathbb{R}^d} \max \{ |\mathsf K(x)|,  |x|^{2d-1} |\mathsf K(x)|  \}<\infty.
    \end{equation}
    Moreover, when $d$ is even, $\mathsf K$ is Schwartz. For $\epsilon>0$ and $x \in \mathbb{R}^d$, let $\mathsf K_\epsilon(x) = \epsilon^{-d} \mathsf K(x/\epsilon)$. Then $\mathfrak h_\epsilon$ agrees in law with $h * \mathsf K_\epsilon$, viewed as a random continuous function modulo a global constant.
\end{lemma}

Recall from~\eqref{eq:def-exponential-metric} that $D_h^\epsilon$ is the exponential metric defined with respect to $h_\epsilon^* = h * \mathsf K_\epsilon$. The above lemma combined with \cite[Theorem 1.2]{dgz-exponential-metric} implies that the rescaled exponential metrics $\mathsf{a}_\epsilon^{-1} D_h^{\epsilon}$ are tight in the local uniform topology. Furthermore, each possible subsequential limit is a continuous metric. We refer to Lemma~\ref{lem:tightness-whole-plane} for details.

\begin{proof}[Proof of Lemma~\ref{lem:kernel}]
We proceed as in Lemma 2.4 in \cite{chg-support}. Let $f:\R^d\to \R$ be compactly supported and smooth such that $\int_{\R^d} f(z) \,dz = 0.$ 

\medskip

\noindent\textbf{Step 1: Covariance of the white-noise approximation.} Let $R>\epsilon$, and define the truncated field
$$
\mathfrak h_{\epsilon,R}(x) := \int_{\mathbb{R}^d} \int_{\epsilon}^R \mathfrak K(\frac{y-x}{t}) t^{-\frac{d+1}{2}} W(dy,dt), \quad x \in \mathbb{R}^d.
$$
Then
\[
\mathrm{Var}\left(\int_{\mathbb{R}^d} f(z) \mathfrak h_{\epsilon, R}(z)dz\right) = \mathbb{E}\left(\left(\int_{\R^d} f(z) \mathfrak h_{\epsilon, R}(z)dz\right)^2\right),
\]
since $\mathbb{E}\big(\int f(z) \mathfrak h_{\epsilon,R}(z)\,dz\big)=0$. By the definition of $\mathfrak h_{\epsilon,R}$ and the translation invariance of white noise,
\begin{align}\label{eqn:wn-var0}
\mathrm{Var}\left(\int_{\mathbb{R}^d} f(z) \mathfrak h_{\epsilon, R}(z)\,dz\right)
 &= \int_{\mathbb{R}^d}\int_{\mathbb{R}^d} f(x)f(y)\,
     \kappa_{\epsilon,R}(x-y)\,dx\,dy,
\end{align}
where
\[
\kappa_{\epsilon,R}(x-y)
 := \mathbb{E}\big[\mathfrak h_{\epsilon,R}(x)\,\mathfrak h_{\epsilon,R}(y)\big].
\]
A direct computation using the definition of white noise yields
\begin{equation}\label{eq:kappa-epsilon-R}
\kappa_{\epsilon,R}(x)
 = \int_{\mathbb{R}^d}\int_{\epsilon}^R
   t^{-(d+1)}
   \mathfrak K\!\left(\frac{x-u}{t}\right)\mathfrak K\!\left(\frac{u}{t}\right)
   \,dt\,du .
\end{equation}

Let $\mathcal F$ denote the Fourier transform, 
\[
\mathcal{F}(f)(\zeta)=\hat{f}(\zeta) := \int_{\R^d} e^{-2\pi i x\cdot \zeta} f(x) \,dx.
\]
Using Plancherel's theorem and the fact that $\mathcal{F}(f\ast g)=\hat{f}\hat{g}$, we obtain
$$
\int_{\R^d}\int_{\R^d} f(x)f(y)\kappa_{\epsilon, R}(x-y)dxdy = \int_{\R^d} \hat{\kappa}_{\epsilon, R}(\zeta) \vert \hat{f}(\zeta)\vert^2 d\zeta.
$$
From~\eqref{eq:kappa-epsilon-R},
\begin{align*}
    \hat \kappa_{\epsilon, R}(\zeta) &= \int_{\epsilon}^R t^{-(d+1)} \mathcal{F}\left(x \mapsto\int_{\R^d} \mathfrak K(\frac{x-u}{t})\mathfrak K(\frac{u}{t})du\right)(\zeta)\, dt \\
    &= \int_{\epsilon}^R t^{-(d+1)} t^{2d} |\hat{\mathfrak K}(t\xi)|^2 \, dt \\
    &= \int_{\epsilon}^R t^{d-1} |\hat{\mathfrak K}(t\zeta)|^2 \,dt.
\end{align*}
    Since $\mathfrak K$ is Schwartz, $\hat{\mathfrak K}$ is also Schwartz and $|\hat{\mathfrak K}(x)| \leq C |x|^{-d}$ for all $x \in \mathbb{R}^d$. This implies that the last integral converges for all $\zeta \neq 0$, and moreover
    \begin{equation}\label{eq:def-hatkappa}
    \hat \kappa_{\epsilon}(\zeta) := \lim_{R \to \infty} \hat{\kappa}_{\epsilon, R}(\zeta) = \int_{\epsilon}^\infty t^{d-1} |\hat{\mathfrak K}(t\zeta)|^2 dt.
    \end{equation}
Thus,
\begin{equation} \label{eqn:kernel-plancherel}
\mathrm{Var}\left(\int_{\R^d} f(z) \mathfrak h_{\epsilon}(z)dz\right) = \int_{\R^d} \hat{\kappa}_{\epsilon}(\zeta) \vert \hat{f}(\zeta)\vert^2 d\zeta.
\end{equation}
Since $\mathfrak K$ is radial, it follows from~\eqref{eq:kappa-epsilon-R} that $\kappa_{\epsilon,R}$ is radial and positive definite. Therefore, $\hat{\kappa}_{\epsilon}$ is radial and non-negative.

\medskip

\noindent\textbf{Step 2: Covariance of the convolution approximation.} By the covariance of the LGF in~\eqref{eq:lgf-cov},
\begin{eqnarray}\label{lhseq}
&&\mathrm{Var} \left( \int_{\R^d} f(z) h\ast \mathsf K_\epsilon(z) dz\right)\nonumber\\
&&= \mathbb{E} \left( \int_{\R^d}\int_{\R^d} f(x)f(y) (h\ast \mathsf K_\epsilon) (x) (h\ast \mathsf K_\epsilon)(y) dxdy \right)\nonumber\\
&& = \int_{\R^d}\int_{\R^d} \int_{\R^d} \int_{\R^d} f(x)f(y) \mathsf{K}_\epsilon (x-u) \mathsf{K}_\epsilon (y-v) \log \left( \frac{1}{|u-v|} \right)dudv dxdy.
\end{eqnarray}
Let
\[
F_x(v):=\int_{\R^d} \mathsf K_\epsilon (x-u)\log\left(\frac{1}{\vert u-v\vert}\right)du.
\]
Then
\[
\hat{F}_x(\zeta) = \mathcal{F}(\mathsf K_\epsilon (x-\cdot)) \frac{c_d}{\vert \zeta\vert^d} = \frac{c_d}{\vert \zeta\vert^d} e^{-2\pi i x\cdot \zeta} \hat{\mathsf K}_\epsilon(-\zeta) ,
\]
where $c_d = \tfrac{\Gamma(d/2)}{2 \pi^{d/2}}$ is the constant such that \[\mathcal{F}\left(\log \frac{1}{|x|}\right)(\zeta) = c_d |\zeta|^{-d} \quad \mbox{for $\zeta \neq 0$.}\] Hence,
\begin{eqnarray*}
&&\mathrm{Var} \left( \int_{\R^d} f(z) h\ast \mathsf K_\epsilon(z) dz\right)\nonumber\\
&& = \int_{\R^d}\int_{\R^d} f(x)f(y) \left(\int_{\R^d} \mathsf K_\epsilon(x-u) F_y(u) du\right) dxdy\nonumber\\
&& = \int_{\R^d} \int_{\R^d} f(x)f(y) \left(\int_{\R^d} \frac{c_d}{\vert \zeta\vert^d} e^{2\pi i (y-x)\cdot \zeta} \vert \hat{\mathsf K}_\epsilon(\zeta)\vert^2 d\zeta\right)dxdy\nonumber\\
&&= \int_{\R^d} \vert \hat{f}(\zeta)\vert^2 \frac{c_d}{\vert \zeta\vert^d} \vert \hat{\mathsf K}_\epsilon (\zeta)\vert^2 d\zeta.
\end{eqnarray*}

\medskip

\noindent\textbf{Step 3: Choosing the kernel.} Choose ${\mathsf K}_\epsilon$ such that
\begin{equation}\label{eq:kernel-formula}
\frac{c_d \vert \hat{\mathsf K}_\epsilon \vert^2}{\vert \zeta\vert^d} = \hat{\kappa}_\epsilon, \quad \mbox{i.e.,} \quad {\mathsf K}_\epsilon = \mathcal F^{-1} \left( \zeta \mapsto \sqrt{\frac{2\pi^{d/2}}{\Gamma(d/2)} \vert \zeta\vert^{d}\hat{\kappa}_\epsilon(\zeta)} \right)\,.
\end{equation}
With this choice, $\mathfrak h_\epsilon$ agrees in law with $h * \mathsf K_\epsilon$ when viewed as a random continuous function modulo an additive constant.

Since $|{\rm D}^j(\sqrt{|\zeta|^d \hat{\kappa}_\epsilon(\zeta)})| \in L^1(\mathbb{R}^d)$ for all $0 \leq j \leq 2d-1$ (Lemma~\ref{lem:schwartz}), we obtain \[\sup_{x \in \mathbb{R}^d} \max \,\{ |\mathsf K_\epsilon(x)|,  |x|^{2d-1} |\mathsf K_\epsilon(x)|  \}<\infty.\] Moreover, when $d$ is even, $\sqrt{|\zeta|^d \hat{\kappa}_\epsilon(\zeta)}$ is Schwartz (Lemma~\ref{lem:schwartz}), hence $\mathsf K_\epsilon$ is Schwartz. By definition, $\hat{\kappa}_{\epsilon}$ is radial and non-negative, so ${\mathsf K}_\epsilon$ is radial and real-valued. Moreover, by~\eqref{eq:def-hatkappa} and polar coordinates,
\begin{align*}
\widehat{\mathsf K}_\epsilon(0)
= \sqrt{\frac{2\pi^{d/2}}{\Gamma(d/2)}
   \int_0^\infty t^{d-1}|\hat{\mathfrak K}(t)|^2\,dt} = \sqrt{\int_{\mathbb{R}^d} |\hat{\mathfrak K}(x)|^2\,dx}.
\end{align*}
Using Plancherel's theorem and $\int_\mathbb{R^d} \mathfrak K(x)^2 \,dx = 1$, we see that $\int_{\mathbb R^d} \mathsf K_\epsilon(x)\,dx = \widehat{\mathsf K}_\epsilon(0) = 1$.

It follows from~\eqref{eq:def-hatkappa} that $\hat{\kappa}_{\epsilon}(\zeta) = \epsilon^d \hat{\kappa}_1(\epsilon \zeta)$ for $\zeta \neq 0$. Therefore, from~\eqref{eq:kernel-formula}, we have $\mathsf K_\epsilon(x) = \epsilon^{-d} \mathsf K_1(x/\epsilon)$ for $x \in \mathbb{R}^d$. Taking $\mathsf K = \mathsf K_1$ completes the proof of the lemma.
\end{proof}

\begin{lemma}\label{lem:schwartz}
    Suppose $\mathfrak K$ satisfies~\eqref{eq:asmp-kernel}. For $\epsilon>0$, let
    \[
    \hat \kappa_{\epsilon}(\zeta) = \int_{\epsilon}^\infty t^{d-1} |\hat{\mathfrak K}(t\zeta)|^2 dt \quad \mbox{and} \quad o(\zeta):= \sqrt{|\zeta|^d \hat \kappa_{\epsilon}(\zeta)}, \quad \zeta \in \mathbb{R}^d.
    \]
    Then, for all $0 \leq j \leq 2d-1$, $|{\rm D}^j o| \in L^1(\mathbb{R}^d)$. Moreover, when $d$ is even, $o(\zeta)$ is a Schwartz function on $\mathbb{R}^d$.
\end{lemma}

To prove Lemma~\ref{lem:schwartz}, we need the following result.

\begin{lemma}\label{lem:sqrt-schwartz}
    Let $f: (0,\infty) \to [0,\infty)$ be a non-negative function such that for any integers $\alpha, n\geq 0$, $\sup_{r>0} |r^\alpha f^{(n)}(r)| < \infty$. Let \[g(r) = \left(\int_r^\infty f(s) ds\right)^{1/2}, \quad r>0.\]Then for any integers $\alpha, n\geq 0$, $\sup_{r>0} |r^\alpha g^{(n)}(r)| < \infty$.
\end{lemma}

\begin{proof}
    The constants $c$ below may change from line to line but are independent of $f$. We first prove by induction that, for integers $n \geq 1$,
    \begin{equation}\label{eq:est-derivative-new}
    \int_x^\infty f(y) dy \geq c_n\frac{f(x)^{1+\tfrac{1}{n}}}{\big(\sup_{y \geq x} |f^{(n)}(y)|\big) ^{\tfrac{1}{n}}}, \quad x>0,
    \end{equation}
    where $c_n>0$ depends only on $n$. The case $n = 1$ follows from
    \[
    f(w) \geq \frac{1}{2} f(x), \quad x < w < x + \frac{f(x)}{2 \sup_{y \geq x} |f^{(1)}(y)| }.
    \]
    Suppose \eqref{eq:est-derivative-new} holds for $n = k$. We claim that
    \begin{equation}\label{eq:observ-new}
    \int_x^\infty f(y) dy \geq c \frac{|f^{(k)}(x)|^{k+2}}{\big(\sup_{y \geq x} |f^{(k+1)}(y)|\big)^{k+1}}, \quad x>0.
    \end{equation}
    Indeed, 
    \[
    |f^{(k)}(w)| \geq \frac{1}{2} |f^{(k)}(x)|, \quad x < w < x + \frac{|f^{(k)}(x)|}{2 \sup_{y \geq x} |f^{(k+1)}(y)| }.
    \]
    Integrating over $w$ iteratively, this implies that for $0 \leq l \leq k$,
    \[
    |f^{(l)}(w)| \geq c \frac{|f^{(k)}(x)|^{k-l+1}}{\big(\sup_{y \geq x} |f^{(k+1)}(y)|\big)^{k-l}}
    \]
    for a positive fraction of $w$ in $(x, x + \frac{|f^{(k)}(x)|}{2 \sup_{y \geq x} |f^{(k+1)}(y)| })$. Taking $l = 0$ yields~\eqref{eq:observ-new}. Furthermore, \eqref{eq:observ-new} implies that for any $w \geq x$,
    \[|f^{(k)}(w)|^{k+2} \leq c \int_w^\infty f(y) dy \,\big(\sup_{y \geq w} |f^{(k+1)}(y)|\big)^{k+1} \leq c \int_x^\infty f(y) dy \,\big(\sup_{y \geq x} |f^{(k+1)}(y)|\big)^{k+1} .\]
    Therefore,
    \begin{equation}\label{eq:observ-1-new}
    \int_x^\infty f(y) dy \geq c \frac{\big(\sup_{y \geq x} |f^{(k)}(y)|\big)^{k+2}}{\big(\sup_{y \geq x} |f^{(k+1)}(y)|\big)^{k+1}}.
    \end{equation}
    Finally, \eqref{eq:est-derivative-new} for $n = k+1$ follows from the case $n = k$ and~\eqref{eq:observ-1-new}:
    \begin{align*}
    \int_x^\infty f(y) dy &\geq \frac{1}{k(k+2)+1} \left[ \underbrace{c_k\frac{f(x)^{1+\tfrac{1}{k}}}{\big(\sup_{y \geq x} |f^{(k)}(y)|\big) ^{\tfrac{1}{k}}} + \ldots + c_k\frac{f(x)^{1+\tfrac{1}{k}}}{\big(\sup_{y \geq x} |f^{(k)}(y)|\big)^{\tfrac{1}{k}}}}_{k(k+2) \mbox{ copies}}  + c \frac{\big(\sup_{y \geq x} |f^{(k)}(y)|\big)^{k+2}}{\big(\sup_{y \geq x} |f^{(k+1)}(y)|\big)^{k+1}} \right] \\
    &\geq c' \left( f(x)^{(1+\tfrac{1}{k})k(k+2) }  \frac{1}{\big(\sup_{y \geq x} |f^{(k+1)}(y)|\big)^{k+1}} \right)^{\tfrac{1}{k(k+2)+1}} = c' \frac{f(x)^{1+\tfrac{1}{k+1}}}{\big(\sup_{y \geq x} |f^{(k+1)}(y)|\big) ^{\tfrac{1}{k+1}}}.
    \end{align*}

    Similarly, using~\eqref{eq:observ-new}, \eqref{eq:observ-1-new} and an induction, for all integers $k \geq 0$ and $m \geq 1$ we have
    \[
    \int_x^\infty f(y) dy \geq c_{k,m} \frac{|f^{(k)}(x)|^{\tfrac{k+1+m}{m}}}{\big(\sup_{y \geq x} |f^{(k+m)}(y)|\big)^{\tfrac{k+1}{m}}}, \quad x>0,
    \]
    Since $m$ can be taken arbitrarily large, this implies that the decay rate of $g(r)^2 = \int_r^\infty f(s) ds$ at infinity is at least as fast as that of $|f^{(k)}(r)|$ for each $k$. Using that the decay rate of $|f^{(k)}(r)|$ is faster than any power of $r$, it follows that the derivatives of $g$ also decay faster than any power of $r$, since the derivatives of $g$ are of the form
    \[
    g^{(n)}(r) = \sum_{m=1}^n \frac{1}{g(r)^{2m-1}}\sum_{0 \leq a_1 \leq a_2 \leq \ldots \leq a_m \leq n-1} c_{m,a_1,a_2,\ldots, a_m}\prod_{i=1}^m f^{(a_i)}(r). \qedhere
    \]
    
\end{proof}

\begin{proof}[Proof of Lemma~\ref{lem:schwartz}]
    By definition and changing variables $s=t |\zeta|$, 
    \[
    o(\zeta)=\sqrt{|\zeta|^d \hat \kappa_{\epsilon}(\zeta)} = \sqrt{\int_{\epsilon|\zeta|}^\infty s^{d-1} |\hat{\mathfrak K}(s)|^2 ds}.
    \]
    Since $s^{d-1} |\hat{\mathfrak K}(s)|^2 $ is smooth on $(0,\infty)$, it follows that $o(\zeta)$ is smooth on $\mathbb{R}^d \setminus \{0\}$. Furthermore, Lemma~\ref{lem:sqrt-schwartz} implies that the derivatives of $o(\zeta)$ have decay rate at infinity faster than $|\zeta|^{-\alpha}$ for any integers $\alpha \geq 0$. 
    
    It remains to analyze the behavior near $\zeta = 0$. Since $|\hat{\mathfrak K}(s)|^2$ is radial and smooth at 0, it admits an expansion of the form
    \[
    |\hat{\mathfrak K}(s)|^2 = \sum_{m \geq 0} a_m s^{2m} \quad \mbox{for small $s$}.
    \]
    Hence
    \[
    \int_{\epsilon|\zeta|}^\infty s^{d-1} |\hat{\mathfrak K}(s)|^2 ds = \int_0^\infty s^{d-1} |\hat{\mathfrak K}(s)|^2 ds  + \sum_{m \geq 0} b_m(\epsilon) |\zeta|^{d+2m}
    \]
    for suitable coefficients $b_m(\epsilon)$, and thus
    \[
    o(\zeta) = \sqrt{\int_0^\infty s^{d-1} |\hat{\mathfrak K}(s)|^2 ds  + \sum_{m \geq 0} b_m(\epsilon) |\zeta|^{d+2m}} = C + \sum_{m \geq 0} b_m'(\epsilon) |\zeta|^{d+2m},
    \]
    with $C>0$. When $d$ is even, this expansion shows that $o(\zeta)$ is smooth at 0 and hence Schwartz. When $d$ is odd, $o(\zeta)$ is not necessarily smooth at 0, but for any multi-index $\alpha$,
    \[
    |\partial^\alpha o(\zeta)| \leq C' |\zeta|^{d - |\alpha|} \quad \mbox{for small }|\zeta|.
    \]
    In particular, for $0 \leq j \leq 2d-1$, $|{\rm D}^j o| \in L^1(\mathbb{R}^d)$.
\end{proof}

\subsection{Localization of the LGF}\label{subsec:localize}

In this section, we show that the convolution considered in Lemma~\ref{lem:kernel} can be localized through truncation (Lemma~\ref{lem:equifield}). Throughout this section, let $h$ be a whole-space LGF with the additive constant chosen such that $h_1(0) = 0$. 

Define
\begin{equation}\label{eq:sec2-field1}
h_\epsilon^*(z) = h\ast {\mathsf K}_\epsilon(z) = \int_{\R^d} h(w) {\mathsf K}_\epsilon (z-w) dw,
\end{equation}
and
\begin{equation}\label{eq:sec2-field2}
\bar{h}_\epsilon^*(z) = h\ast (\Psi_\e {\mathsf K}_\epsilon)(z) = \int_{\R^d} h(w) {\mathsf K}_\e(z-w)\Psi_\epsilon(z-w) dw,
\end{equation}
where $\Psi_\e : \mathbb{R}^d \to [0,1]$ is a radial bump function such that $\Psi_\e(z)=1$ if $z \in B_{\sqrt{\e}/2}(0)$ and $\Psi_\e(z)=0$ if $z \in \R^d \setminus B_{\sqrt{\e}}(0).$ We have the following comparison between $h_\epsilon^\ast$ and $\bar{h}_\epsilon^*.$
\begin{lemma}\label{lem:equifield}
Let $h$ be a whole-space LGF with the additive constant chosen so that $h_1(0) = 0$, and let $\mathsf K$ be defined as in Lemma~\ref{lem:kernel}. Then for each bounded open set $U \subset \R^d,$ a.s.
\[
\lim_{\e\to 0} \sup_{z \in \overline{U}}|h_{\epsilon}^*(z)-\bar{h}_{\epsilon}^*(z)|=0.
\]
\end{lemma}
To prove Lemma \ref{lem:equifield}, we will need the following technical lemma. Recall that by \cite[Section 11.1]{fgf-survey} the sphere average process $\{h_r(z)\}_{r > 0, z \in \mathbb{R}^d}$ is well defined and admits a modification continuous in both $r$ and $z$.
\begin{lemma}\label{lem:sphereavg} Let $h$ be a whole-space LGF with the additive constant chosen so that $h_1(0) = 0$, and let $\{h_r(z)\}_{r > 0, z \in \mathbb{R}^d}$ be its sphere average process. For each $R>0,$ $\zeta>0,$ a.s. we have
\[
\sup_{z\in B_R(0)}\sup_{r > 0} \frac{|h_r(z)|}{\max\{ \log\frac{1}{r},(\log r)^{\frac{1}{2}+\zeta},1\}} < \infty.
\]
\end{lemma}
We now prove Lemma \ref{lem:equifield} assuming Lemma \ref{lem:sphereavg}.

\begin{proof}[Proof of Lemma \ref{lem:equifield}.]
Using polar coordinates and radiality of ${\mathsf K}_{\epsilon}$, we have
\[
h_{\epsilon}^*(z) = \frac{2 \pi^{d/2}}{\Gamma(d/2)} \, \int_0^\infty r^{d-1}h_r(z) {\mathsf K}_{\epsilon} (r)dr,
\]
\[
\bar{h}_{\epsilon}^*(z) = \frac{2 \pi^{d/2}}{\Gamma(d/2)} \, \int_0^\infty r^{d-1}h_r(z) \Psi_{\epsilon}(r) {\mathsf K}_{\epsilon} (r)dr.
\]
Therefore,
\[
|h_{\epsilon}^*(z)-\bar{h}_{\epsilon}^*(z)| \leq \frac{2 \pi^{d/2}}{\Gamma(d/2)} \, \int_{\sqrt{\e}/2}^\infty r^{d-1} \,|h_r(z)| \, |{\mathsf K}_{\epsilon} (r)| \, dr.
\]
By Lemma \ref{lem:sphereavg} (taking $\zeta = 1/2$), there exists a random constant $C$ such that 
\[
|h_r(z)| \leq C\max\{\log \frac{1}{r},\log r,1\}
\]
for all $z \in \overline{U}$ and $r>0$. Hence
\begin{equation*}
\sup_{z \in \overline{U}} |h_{\epsilon}^*(z)-\bar{h}_{\epsilon}^*(z)| \leq C \int_{\sqrt{\e}/2}^\infty r^{d-1} \max\{\log \frac{1}{r},\log r,1\} \epsilon^{-d} |\mathsf K (r/\epsilon)| \,dr.
\end{equation*}
Substituting $r = \epsilon t$, we obtain
\begin{align*}
    \sup_{z \in \overline{U}} |h_{\epsilon}^*(z)-\bar{h}_{\epsilon}^*(z)| &\leq C \int_{\frac{1}{2 \sqrt{\epsilon}}}^\infty t^{d-1} \max\{\log \frac{1}{\epsilon t},\log \epsilon t,1\} \, |\mathsf K(t)| \,dt \\
    &\leq C \int_{\frac{1}{2 \sqrt{\epsilon}}}^\infty t^{d-1} \, \sqrt{t} \, |\mathsf K(t)| \,dt
\end{align*}
By Lemma~\ref{lem:kernel}, there exists $C'>0$ such that $|{\mathsf K}(x)| \leq C' |x|^{-2d+1}$ for $x \in \mathbb{R}^d$. Combining this with the above inequality,
\[
\sup_{z \in \overline{U}} |h_{\epsilon}^*(z)-\bar{h}_{\epsilon}^*(z)|
 \leq C'' \int_{\frac{1}{2 \sqrt{\epsilon}}}^\infty t^{-d+1/2}\,dt,
\]
which tends to $0$ as $\epsilon \to 0$ for all $d \ge 2$. This proves the lemma. \qedhere
\end{proof}

We now prove Lemma \ref{lem:sphereavg}. We will need the following estimate.
\begin{lemma}\label{lem:sphereavg2}
For all $R>0$ and $\zeta>0,$ we have a.s.
\[
\lim_{r \to \infty} \sup_{z\in B_R(0)}\frac{|h_r(z)|}{(\log r)^{\frac{1}{2}+\zeta}} = 0.
\]
\end{lemma}
\begin{proof}
The process $\{h_r(z)-h_r(0):z \in B_R(0), r \in [\frac{1}{2},1]\}$ is a centered Gaussian process with variances bounded above by an $R$-dependent constant. By \cite[Section 11.1]{fgf-survey}, this process has a continuous modification. Therefore,
\[
\sup_{z\in B_R(0)}\sup_{r \in [\frac{1}{2},1]} |h_r(z) - h_r(0)| < \infty.
\]
By the Borel-TIS inequality \cite{itsjustborell,sudakov} (see e.g. \cite[Theorem 2.1.1]{rfgbook}), we have
\[
\mathbb{E}\left(\sup_{z \in B_R(0)}\sup_{r \in [\frac{1}{2},1]}|h_r(z)-h_r(0)|\right)<\infty
\]
and
\[
\mathbb{P}\left(\sup_{z \in B_R(0)} \sup_{r \in[\frac{1}{2},1]}|h_r(z)-h_r(0)|>A\right) \leq c^{-1}e^{-cA^2}
\]
for some constant $c>0$ only depending on $R$. Using the scaling invariance of $h$ (Lemma~\ref{lem:scale-invariance}), we obtain
\[
\mathbb{P}\left(\sup_{z \in B_{R2^k}(0)}\sup_{r \in[2^{k-1},2^k]}|h_r(z)-h_r(0)|>A\right) \leq c^{-1}e^{-cA^2}
\]
for all $k \in \mathbb{N}$. Applying this with $A = k^{\frac{1}{2}+\frac{\zeta}{2}}$ and the Borel-Cantelli lemma, we obtain
\[
\lim_{k\to\infty}\sup_{z\in B_{R2^k}(0)}\sup_{r \in[2^{k-1},2^k]} \frac{h_r(z)-h_r(0)}{(\log r)^{\frac{1}{2}+\zeta}} = 0 \quad \mbox{a.s.}
\]
Finally, using that $\frac{h_r(0)}{(\log r)^{\frac{1}{2}+\zeta}} \to 0$ as $r \to \infty$ (as can be seen from direct computation), we obtain the result.
\end{proof}

\begin{proof}[Proof of Lemma \ref{lem:sphereavg}.]
By \cite{thick,otherthick}, we have a.s.
\[
\sup_{z \in B_R(0)} \sup_{r \in [0,\frac{1}{2}]} \frac{|h_r(z)|}{\log\left(\frac{1}{r}\right)} < \infty.
\]
By continuity of the sphere average process, for any fixed $r_0 > \frac{1}{2},$
\[
\sup_{z \in B_R(0)}\sup_{r \in [\frac{1}{2},r_0]}|h_r(z)| < \infty \quad \mbox{a.s.}
\]
Combining these bounds with Lemma \ref{lem:sphereavg2} and using that $(\log r)^{\frac{1}{2}+\zeta}$ controls $|h_r(z)|$ for large $r$, we obtain the stated uniform bound over all $r>0$.
\end{proof}

\section{Weak exponential metrics: proof of Theorem~\ref{thm:axiom}}\label{sec:axiom}

In this section, we prove Theorem~\ref{thm:axiom}. Namely, we verify that all subsequential limits of the rescaled exponential metrics $\mathsf a_\epsilon^{-1} D_h^\epsilon$ as $\epsilon \to 0$ satisfy Axioms~\ref{axiom-length}---\ref{axiom-tight}. The proof mainly follows the argument in~\cite{lqg-metric-estimates} (see also~\cite{Pfeffer-weak-metric}) with modifications to handle higher dimensions. In this section, we use $h$ to denote a whole-space LGF (we specify the choice of additive constant when needed), and $\mathsf h$ to denote a whole-space LGF plus a bounded continuous function. Throughout this section, we fix $\mathfrak K$ satisfying~\eqref{eq:asmp-kernel} and choose $\mathsf K$ as in Lemma~\ref{lem:kernel}.

\subsection{Convergence among dyadic boxes, Weyl scaling, and tightness
across scales}

In this section, we consider a whole-space LGF plus a bounded continuous function $\mathsf h$. Let $D_{\mathsf h}^\epsilon$ be the exponential metric associated with $\mathsf h_\epsilon^*$ defined as in~\eqref{eq:def-exponential-metric} with $\mathsf h$ in place of $h$, and recall $\mathsf a_\epsilon$ from~\eqref{eq:def-a-epsilon} which depends on $\mathsf h$. In fact, we will show in the proof of Lemma~\ref{lem:tightness-whole-plane} that for different $\mathsf h$, the normalizing constants $\mathsf a_\epsilon$ are up-to-constants equivalent as $\epsilon \to 0$. For an open domain $U \subset \mathbb{R}^d$, let $D_{\mathsf h}^\epsilon(\cdot, \cdot; U)$ denote the $D_{\mathsf h}^\epsilon$-internal metric on $U$. We call a domain of the form $x + [0, 2^{-n}]^d$ a dyadic box, where $n$ is an integer and $x \in 2^{-n} \mathbb{Z}^d$. An open domain $U$ is said to be a dyadic domain if it can be written as the interior of a finite union of dyadic boxes. Let $\mathcal{W}$ be the set of all dyadic domains.

\begin{lemma}\label{lem:dyadicconv}
    Let $\mathsf h$ be a whole-space LGF plus a bounded continuous function. Then the following hold.
    \begin{enumerate}
        \item As $\epsilon \to 0$, the laws of the rescaled metrics $\mathsf{a}_\epsilon^{-1} D_{\mathsf h}^\epsilon(\cdot, \cdot)$ are tight in the local uniform topology on $\mathbb{R}^d \times \mathbb{R}^d$. Furthermore, each subsequential limit is a continuous length metric on $\mathbb{R}^d$.

        \item For every sequence $\{\epsilon_k\}_{k \geq 1}$ tending to 0, there exists a subsequence $\mathcal{E}$ such that along $\mathcal{E}$,
        $$
        (\mathsf{a}_\epsilon^{-1} D_{\mathsf h}^\epsilon, \{\mathsf{a}_\epsilon^{-1} D_{\mathsf h}^\epsilon(\cdot, \cdot; W) \}_{W \in \mathcal{W}}) \to (\widetilde D, \{\widetilde D_W\}_{W \in \mathcal{W}})
        $$
        where the convergence is in the local uniform topology on $\mathbb{R}^d \times \mathbb{R}^d$ and on $W \times W$, respectively. Moreover, for all $W \in \mathcal{W}$, $\widetilde D_W$ is a continuous metric on $W$, and we have 
        \begin{equation}\label{eq:internal-consistent}
            \widetilde D(\cdot, \cdot; W) = \widetilde D_W(\cdot, \cdot; W),
        \end{equation} where $\widetilde D(\cdot, \cdot; W)$ is the $\widetilde D$-internal metric on $W$ and $\widetilde D_W(\cdot, \cdot; W)$ is the $\widetilde D_W$-internal metric on $W$.
    \end{enumerate}
\end{lemma}

\begin{remark}
    In the two-dimensional case~\cite{lqg-metric-estimates}, it is known that $\widetilde D_W$ is a length metric on $W$, i.e., $\widetilde D_W(\cdot, \cdot; W) = \widetilde D_W(\cdot, \cdot)$. The main obstacle to extending this result to higher dimensions lies in establishing the tightness of $\mathsf{a}_\epsilon^{-1} D_{\mathsf h}^\epsilon(\cdot, \cdot; \overline W)$ in the uniform topology on $W \times W$. In~\cite{dgz-exponential-metric}, only tightness in the local uniform topology was obtained, whereas tightness in the uniform topology was proved in two dimensions in~\cite{DDDF-tightness}. Nevertheless, we expect that the tightness result in the uniform topology can be achieved by adapting the arguments in~\cite{dgz-exponential-metric}. If so, this would imply that $\widetilde D_W$ is also a length metric on $W$ in higher dimensions.

\end{remark}

We first tailor Theorem 1.2 of~\cite{dgz-exponential-metric} to show the tightness of $\mathsf{a}_\epsilon^{-1} D_{\mathsf h}^\epsilon$ as $\epsilon \to 0$.

\begin{lemma}\label{lem:tightness-whole-plane}
    Let $\mathsf h$ be a whole-space log-correlated Gaussian field plus a  bounded continuous function. Then as $\epsilon \to 0$, the family of rescaled metrics $\{\mathsf{a}_\epsilon^{-1} D_{\mathsf h}^\epsilon(\cdot,\cdot)\}_{\epsilon>0}$ is tight with respect to the local uniform topology on $\mathbb{R}^d \times \mathbb{R}^d$. Moreover, any subsequential limit is supported on continuous metrics on $\mathbb{R}^d$.
\end{lemma}

\begin{proof}
    Let $\mathfrak h_{\epsilon, 1}$ (resp.\ $\mathfrak h_{1,\infty}$) denote the variant of $\mathfrak h_\epsilon$ in~\eqref{eq:white-noise} obtained by integrating over the time interval $(\epsilon,1)$ (resp.\ $(1, \infty)$) rather than $(\epsilon, \infty)$. Then $\mathfrak h_{\epsilon, 1}$ is a random continuous function and $\mathfrak h_{1,\infty}$ is a random continuous function viewed modulo a global constant. Let $\hat D_{\mathfrak h}^\epsilon$ be the exponential metric defined with respect to $\mathfrak h_{\epsilon, 1}$, and let $\hat{\mathsf{a}}_\epsilon$ be the median of $\hat D_{\mathfrak h}^\epsilon(0,e_1)$. As proved in Theorem 1.2 of~\cite{dgz-exponential-metric}, as $\epsilon \to 0$ the family of rescaled metrics $\hat{\mathsf{a}}_\epsilon^{-1} \hat D_{\mathfrak h}^\epsilon(\cdot,\cdot)$ is tight with respect to the local uniform topology on $\mathbb{R}^d \times \mathbb{R}^d$. Moreover, any subsequential limit is supported on continuous metrics on $\mathbb{R}^d$.

    Let $h$ be a whole-space LGF with the additive constant chosen so that $h_1(0) = 0$. We now explain how to extend the tightness result to $\mathsf a_\epsilon^{-1} D_h^\epsilon$. The key ingredient is Lemma~\ref{lem:kernel}, which states that $h * \mathsf K_\epsilon$ and $\mathfrak h_\epsilon = \mathfrak h_{\epsilon,1} + \mathfrak h_{1,\infty}$ have the same law when viewed as random continuous functions modulo a global constant. Therefore, there exists a coupling of $h$ and the white noise such that \[h * \mathsf K_\epsilon - \mathfrak h_{\epsilon, 1}\] is a random continuous function whose law is independent of $\epsilon$. Note that the coupling may depend on $\epsilon$. With slight abuse of notation, we still use $\mathfrak h_{1,\infty}$ to denote $h * \mathsf K_\epsilon - \mathfrak h_{\epsilon, 1}$. We now work with this coupling.
    
    By \cite[Lemma 6.7]{dgz-exponential-metric}, for any $r>0$ and $p \in (0,1)$, there exists $R >r $ (depending on $r,p$) such that
    $$
    \liminf_{\epsilon \to 0} \mathbb{P}\left[\sup_{u,v \in B_r(0)} \hat D_{\mathfrak h}^\epsilon(u,v) < \hat D_{\mathfrak h}^\epsilon(\partial B_r(0), \partial B_R(0))\right] \geq p.
    $$
    On this event, the $\hat D_{\mathfrak h}^{\epsilon}$-geodesic between any two points $u,v \in B_r(0)$ cannot cross $A_{r,R}(0) = B_R(0) \setminus \overline{B_r(0)}$ and thus lies in $B_R(0)$. Therefore, for all $u,v \in B_r(0)$, 
    \begin{equation}\label{eq:lem-whole-plane-1}
         D_h^{\epsilon}(u,v) \leq \exp\left(\xi \sup_{w \in B_R(0)} \mathfrak h_{1, \infty}(w)\right) \hat D_{\mathfrak h}^{\epsilon}(u,v) .
    \end{equation}
    Furthermore, on the same event, by considering whether the $D_h^\epsilon$-geodesic between $u,v$ crosses $A_{r,R}(0)$, for all $u,v \in B_r(0)$ we have
    \begin{equation}\label{eq:lem-whole-plane-2}
    \begin{aligned}
        D_h^{\epsilon}(u,v) &\geq \exp\left(\xi \inf_{w \in B_R(0)}  \mathfrak h_{1, \infty}(w)\right) \min \{ \hat D_{\mathfrak h}^\epsilon (u,v), \hat D_{\mathfrak h}^\epsilon(\partial B_r(0), \partial B_R(0)) \} \\
        &= \exp\left( \xi \inf_{w \in B_R(0)} \mathfrak h_{1, \infty}(w) \right) \hat D_{\mathfrak h}^\epsilon (u,v),
    \end{aligned}
    \end{equation}
    where in the last equality we use $\sup_{u,v \in B_r(0)} \hat D_{\mathfrak h}^\epsilon(u,v) < \hat D_{\mathfrak h}^\epsilon(\partial B_r(0), \partial B_R(0))$.

    Combining~\eqref{eq:lem-whole-plane-1} and~\eqref{eq:lem-whole-plane-2}, it follows that as $\epsilon \to 0$, the ratios $\mathsf a_\epsilon / \hat{\mathsf a}_\epsilon$ are bounded from above and below, and the family of rescaled metrics $\{\mathsf{a}_\epsilon^{-1} D_h^\epsilon(\cdot,\cdot)\}_{\epsilon>0}$ is tight with respect to the local uniform topology on $\mathbb{R}^d \times \mathbb{R}^d$. 
    
    By~\cite[Proposition 6.4]{dgz-exponential-metric}, for any $R>r>0$ and $p \in (0,1)$, there exists $c = c(r,R,p)>0$ such that
    $$
    \liminf_{\epsilon \to 0} \mathbb{P}\left[\hat D_{\mathfrak h}^\epsilon(\partial B_r(x), \partial B_{R}(x)) \geq c \hat{\mathsf{a}}_\epsilon \right] \geq p \quad \mbox{for all $x \in \mathbb{R}^d$}.
    $$
    Combining this with
    \begin{equation}\label{eq:lem-whole-plane-3}
    D_h^{\epsilon}(\partial B_r(x), \partial B_{R}(x)) \geq \exp\left( \xi \inf_{w \in B_{R}(x)} \mathfrak h_{1, \infty}(w) \right) \hat D_{\mathfrak h}^{\epsilon}(\partial B_r(x), \partial B_{R}(x))
    \end{equation}
    implies that every subsequential limit of $\{\mathsf{a}_\epsilon^{-1} D_h^\epsilon(\cdot,\cdot)\}_{\epsilon>0}$ as $\epsilon \to 0$ is a continuous metric, i.e., induces the Euclidean topology. This proves the lemma for the whole-space LGF. The case of a whole-space LGF plus a bounded continuous function can be proved similarly, since the analogs of~\eqref{eq:lem-whole-plane-1}, ~\eqref{eq:lem-whole-plane-2}, and~\eqref{eq:lem-whole-plane-3} hold.
\end{proof}

The following lemma is used to show that any subsequential limit of $\mathsf{a}_\epsilon^{-1} D_{\mathsf h}^\epsilon$ is a length metric. It follows from Lemma~\ref{lem:tightness-whole-plane} and the scaling invariance of the whole-space LGF (Lemma~\ref{lem:scale-invariance}).

\begin{lemma}\label{lem:cross-estimate}
    Let $\mathsf h$ be a whole-space LGF plus a bounded continuous function. For each $p \in (0,1)$ and $C>0$, there exists a constant $R>0$, depending on $p, C$, and the law of $\mathsf h$, such that for all $r>0$ and $x \in \mathbb{R}^d$,
    \begin{equation}\label{eq:lem3.4-1}
    \liminf_{\epsilon \to 0} \mathbb{P}\left[\sup_{u,v \in B_r(x)} D_{\mathsf h}^\epsilon(u,v) \leq \frac{1}{C} D_{\mathsf h}^\epsilon(\partial B_r(x), \partial B_{Rr}(x))\right] \geq p.
    \end{equation}
\end{lemma}

\begin{proof}
    We first consider the case of the whole-space LGF $h$. By Lemma~\ref{lem:tightness-whole-plane}, for each $p \in (0,1)$ and $C>0$, there exists $R>0$ such that 
    \begin{equation*}
        \liminf_{\epsilon \to 0} \mathbb{P}\left[\sup_{u,v \in B_{1/R}(0)} \mathsf{a}_\epsilon^{-1} D_h^\epsilon(u,v) \leq \frac{1}{C} \mathsf{a}_\epsilon^{-1} D_h^\epsilon(\partial B_{1/R}(0), \partial B_1(0))\right] \geq p.
    \end{equation*}
    Combining this with the scaling invariance of $h$ (Lemma~\ref{lem:scale-invariance}) and the definition of $D_h^\epsilon$ implies~\eqref{eq:lem3.4-1}. For the case where $\mathsf h$ is a whole-space LGF plus a bounded continuous function, say $\mathsf h = h + g$ with \(g\) bounded, we note that  
    $$
    \exp\left(\xi \inf_{w \in \mathbb{R}^d} g(w)\right) D_h^\epsilon 
    \leq D_{\mathsf h}^\epsilon 
    \leq \exp\left(\xi \sup_{w \in \mathbb{R}^d} g(w)\right) D_h^\epsilon ,
    $$
    which gives the desired result. \qedhere
    
\end{proof}

The following technical lemma is used to show~\eqref{eq:internal-consistent}. It follows from standard arguments for length metrics. 

\begin{lemma}\label{lem:internal-metric}
    Let $V \subset U \subset \mathbb{R}^d$ be two open domains, and let $\{D_n\}$ be a sequence of continuous length metrics on $U$. Suppose that $\{D_n\}$ converges with respect to the local uniform topology on $U \times U$ to a continuous metric $D$, and $\{D_n(\cdot, \cdot; V)\}$ converges with respect to the local uniform topology on $V \times V$ to a continuous metric $\widehat D$. Then, $D(u,v; V) = \widehat D(u,v; V)$ for all $u,v \in V$. (Note that $D$ and $\widehat D$ are not assumed to be length metrics).
\end{lemma}

\begin{proof}
    Since $D_n(\cdot, \cdot) \leq D_n(\cdot, \cdot; V)$, we have $D \leq \widehat D$ on $V \times V$. Similar to \cite[Lemma 2.11]{lqg-metric-estimates}, one can show that for all $u,v \in V$ with $D(u,v) < D(u, \partial V)$, we have $D(u,v) = \widehat D(u,v)$. It follows that the $D$-length of any path in $V$ equals its $\widehat D$-length. Hence $D(u,v; V) = \widehat D(u,v; V)$ for all $u,v \in V$. \qedhere
    
\end{proof}

We are now ready to prove Lemma~\ref{lem:dyadicconv}.

\begin{proof}[Proof of Lemma~\ref{lem:dyadicconv}]
    The tightness of the rescaled metrics $\{\mathsf{a}_\epsilon^{-1} D_{\mathsf h}^\epsilon\}_{\epsilon >0}$ as $\epsilon \to 0$ follows from Lemma~\ref{lem:tightness-whole-plane}. Suppose that along $\epsilon_n \to 0$, the laws of $\{\mathsf{a}_{\epsilon_n}^{-1} D_{\mathsf h}^{\epsilon_n}\}$ converge in the local uniform topology to $\widetilde D$. Again by Lemma~\ref{lem:tightness-whole-plane}, $\widetilde D$ is a.s.\ a continuous metric on $\mathbb{R}^d$. The fact that $\widetilde D$ is a length metric follows from Lemma~\ref{lem:cross-estimate}, as we now elaborate. 
    
    Fix $u,v \in \mathbb{R}^d$ and $\delta>0$. By definition, for any $\epsilon>0$ we can choose a path $P^\epsilon$ connecting $u$ and $v$ in $\mathbb{R}^d$ such that \[{\rm len}(P^\epsilon; \mathsf{a}_\epsilon^{-1} D_{\mathsf h}^\epsilon) \leq  \mathsf{a}_\epsilon^{-1} D_{\mathsf h}^\epsilon(u,v) + \delta.\] Parametrize $P^\epsilon$ by its $\mathsf{a}_\epsilon^{-1} D_{\mathsf h}^\epsilon$-length. It follows from Lemma~\ref{lem:cross-estimate} that $P^\epsilon$ is uniformly bounded as $\epsilon \to 0$. Moreover, Lemma~\ref{lem:tightness-whole-plane} implies that $P^\epsilon$ is uniformly equicontinuous. By the Arzel\`a-Ascoli theorem, we may pass to a subsequence $\{\epsilon_n'\}_{n \geq 1}$ along which
    \[
    (\mathsf{a}_{\epsilon_n'}^{-1} D_{\mathsf h}^{\epsilon_n'}, P^{\epsilon_n'}) \to (\widetilde D, P) \quad \mbox{jointly in law},
    \]
    where the second convergence is in the topology of parametrized curves. 
    
    We claim that ${\rm len}(P; \widetilde D) \leq \widetilde D(u,v) + \delta$. Indeed, for each fixed partition $0 = t_0 < t_1 <\ldots < t_n $ of $P$, we have
    \begin{align*}
    \sum_{i=0}^{n-1} \widetilde D(P(t_i), P(t_{i+1})) &= \sum_{i=0}^{n-1} \lim_{\epsilon_n' \to 0} \mathsf{a}_{\epsilon_n'}^{-1} D_{\mathsf h}^{\epsilon_n'}(P^{\epsilon_n'}(t_i), P^{\epsilon_n'}(t_{i+1}))\\
    &\leq \liminf_{\epsilon_n' \to 0} {\rm len}(P^{\epsilon_n'}; \mathsf{a}_{\epsilon_n'}^{-1} D_{\mathsf h}^{\epsilon_n'} )\\
    &\leq \liminf_{\epsilon_n' \to 0} \big(\mathsf{a}_{\epsilon_n'}^{-1} D_{\mathsf h}^{\epsilon_n'}(u,v) + \delta \big)= \widetilde D (u,v) + \delta.
    \end{align*}
    Taking the supremum over all partitions, we obtain ${\rm len}(P; \widetilde D) \leq \widetilde D(u,v) + \delta$. Since $\delta$ is arbitrary, we see that $\widetilde D$ is a length metric. This proves the first claim in the lemma.

    We now prove the second claim. The tightness of $\{\mathsf{a}_\epsilon^{-1} D_{\mathsf h}^\epsilon(\cdot, \cdot; W)\}_{\epsilon >0}$ as $\epsilon \to 0$ for dyadic domains $W$ follows from Lemma~\ref{lem:tightness-whole-plane}. Using a diagonal argument over $W \in \mathcal W$, we can pass to a subsequence along which the laws of $\mathsf{a}_\epsilon^{-1} D_{\mathsf h}^\epsilon$ and $\{\mathsf{a}_\epsilon^{-1} D_{\mathsf h}^\epsilon(\cdot, \cdot; W)\}_{W \in \mathcal W}$ jointly converge. The equality~\eqref{eq:internal-consistent} then follows from Lemma~\ref{lem:internal-metric}.
\end{proof}

The following lemma is used to prove Weyl scaling (Axiom~\ref{axiom-weyl}). Its proof is similar to \cite[Lemma 2.12]{lqg-metric-estimates}, so we omit it here.

\begin{lemma}\label{lem:Weyl-scaling}
    Let $\mathsf h$ be a whole-space LGF plus a bounded continuous function. Let $\{\epsilon_n\}_{n \geq 1}$ be a sequence tending to 0 such that $\mathsf{a}_{\epsilon_n}^{-1} D_{\mathsf h}^{\epsilon_n}$ converges in law, in the local uniform topology on $\mathbb{R}^d \times \mathbb{R}^d$, to a limit $\widetilde D$. By the Skorokhod representation theorem, there exists a coupling under which this convergence holds almost surely. The following property holds almost surely under this coupling.
    
    For every sequence $(f_n)$ of continuous functions on $\mathbb{R}^d$  that converges in the local uniform topology to a bounded continuous function $f$, let $D_{\mathsf h+f_n}^{\epsilon_n}$ be the exponential metric associated with $\mathsf h+f_n$. Then, as $\epsilon_n \to 0$, $\mathsf{a}_{\epsilon_n}^{-1} D_{\mathsf h+f_n}^{\epsilon_n}$ converges to $e^{\xi f} \cdot \widetilde D$, as defined in~\eqref{eq:weyl}.
\end{lemma}


Next, we show tightness across scales (Axiom~\ref{axiom-tight}).

\begin{lemma}\label{lem:tight-scale}
    Let $h$ be a whole-space LGF with the additive constant chosen so that $h_1(0) = 0$. Let $\{\epsilon_n\}_{n \geq 1}$ be a sequence tending to 0 such that $(h, \mathsf a_{\epsilon_n}^{-1} D_h^{\epsilon_n})$ converges jointly in law to $( h, D_h)$. Then there exists a deterministic function $r \mapsto \mathfrak{c}_r >0$ such that the family of random metrics $\{\mathfrak{c}_r^{-1} e^{-\xi h_r(0)} D_h(r \cdot, r \cdot)\}_{0<r<\infty}$ has tight laws with respect to the local uniform topology on $\mathbb{R}^d \times \mathbb{R}^d$. Moreover, each subsequential limit of these laws is supported on continuous metrics on $\mathbb{R}^d$.

\end{lemma}

\begin{proof}

For $r>0$ define
\[
  h^r(z) := h(rz) - h_r(0).
\]
By the scaling invariance of the whole-space LGF (Lemma~\ref{lem:scale-invariance}), we have $h^r \overset{d}{=} h$.

\medskip

\noindent\textbf{Step 1: Construction of $\mathfrak c_r$.} 
By~\eqref{eq:def-exponential-metric}, we have
\[
  D_{h^r}^{\varepsilon/r}(u,v)
  = r^{-1} e^{-\xi h_r(0)} D_h^\varepsilon(ru,rv), \qquad u,v\in\mathbb{R}^d.
\]
Dividing by $\mathsf a_{\varepsilon/r}$ gives
\begin{equation}\label{eq:scaled-h-r}
  \mathsf a_{\varepsilon/r}^{-1} D_{h^r}^{\varepsilon/r}(u,v)
  = \frac{\mathsf a_\varepsilon}{r \mathsf a_{\varepsilon/r}}\,
    e^{-\xi h_r(0)} \bigl( \mathsf a_\varepsilon^{-1} D_h^\varepsilon(ru,rv) \bigr).
\end{equation}
By Lemma~\ref{lem:tightness-whole-plane}, the family $\{\mathsf a_\varepsilon^{-1} D_h^\varepsilon\}_{0<\varepsilon<1}$
is tight and any subsequential limit is a continuous metric on $\mathbb{R}^d$.
Since $h^r \overset{d}= h$, the same holds for
$\{\mathsf a_{\varepsilon/r}^{-1} D_{h^r}^{\varepsilon/r}\}_{0<\varepsilon<r,\ r>0}$.

Along the sequence $\{\epsilon_n \}_{n \geq 1}$, the right-hand side of \eqref{eq:scaled-h-r} converges in law to
\[
  \mathfrak c_r\, e^{-\xi h_r(0)} D_h(r\cdot,r\cdot),
  \qquad
  \mbox{where} \quad \mathfrak c_r := \lim_{\varepsilon_n\to0}
         \frac{\mathsf a_{\varepsilon_n}}{r \mathsf a_{\varepsilon_n/r}},
\]
provided the limit $\mathfrak c_r$ exists. On the other hand, by definition of $\mathsf a_{\varepsilon/r}$ as the median of $D_{h^r}^{\varepsilon/r}(0, e_1)$, any subsequential limit of
$\mathsf a_{\varepsilon/r}^{-1} D_{h^r}^{\varepsilon/r}$ has median distance between $0$ and $e_1$ equal to $1$. Comparing this median distance in the limit of \eqref{eq:scaled-h-r} shows that $\mathfrak c_r$
exists, is finite and positive, and is uniquely determined by the law of $D_h$.

\medskip

\noindent\textbf{Step 2: Tightness across scales.} Rewriting \eqref{eq:scaled-h-r} using $\mathfrak c_r = \lim_{\varepsilon_n\to0}
        \frac{\mathsf a_{\varepsilon_n}}{r \mathsf a_{\varepsilon_n/r}}$ and the convergence $\mathsf a_{\varepsilon_n}^{-1} D_h^{\varepsilon_n} \to D_h$ in law, we get
\[
  \mathsf a_{\varepsilon_n/r}^{-1} D_{h^r}^{\varepsilon_n/r}(u,v)
  \;\to\;
  \mathfrak c_r^{-1} e^{-\xi h_r(0)} D_h(ru,rv) \quad \mbox{in law as $\epsilon_n \to 0$.}
\]
Since the family
$\{\mathsf a_{\varepsilon/r}^{-1} D_{h^r}^{\varepsilon/r}\}_{0<\varepsilon<r,\ r>0}$ is tight
and any subsequential limit is a continuous metric on $\mathbb{R}^d$, it follows that the family of laws
\[
  \bigl\{ \mathfrak c_r^{-1} e^{-\xi h_r(0)} D_h(r\cdot,r\cdot) \bigr\}_{0<r<\infty}
\]
is tight in the local uniform topology on
$\mathbb{R}^d\times\mathbb{R}^d$, and each subsequential limit is a subsequential limit of $\{\mathsf a_{\varepsilon/r}^{-1} D_{h^r}^{\varepsilon/r}\}_{0<\varepsilon<r,\ r>0}$, hence supported on continuous metrics on $\mathbb{R}^d$. \qedhere

\end{proof}

\begin{remark}
    Note that in Lemma~\ref{lem:tight-scale}, the tightness is through $r \in (0,\infty)$ instead of bounded $r$. Using Weyl scaling (Axiom~\ref{axiom-weyl}), the tightness across scales also holds for a whole-space LGF plus a bounded continuous function $\mathsf h$. However, for a whole-space LGF plus a general continuous function, the tightness across scales only holds for bounded $r$, say $r \in (0,R)$ where $R$ is fixed; see Axiom~\ref{axiom-tight}. In Axiom V of~\cite{lqg-metric-estimates}, the authors restricted to the whole-space GFF, so this is consistent with our treatment.
\end{remark}

\subsection{Local metric}

In this section, we introduce the notion of a local metric and prove that any subsequential limit of $\{\mathsf{a}_\epsilon^{-1} D_h^\epsilon \}_{\epsilon>0}$ as $\epsilon \to 0$ is a local metric. The two-dimensional version was introduced in~\cite{local-metrics} and proved in~\cite{lqg-metric-estimates}.

\begin{definition}[Local metric]\label{def:local-metric}

    Let $h$ be a whole-space LGF, and let $D$ be a random continuous length metric on $\mathbb{R}^d$, coupled with $h$. We say that $D$ is a \emph{local metric for $h$} if for every open set $V \subset \mathbb{R}^d$, the internal metric $D(\cdot, \cdot; V)$ is conditionally independent of the pair $(h, \; D(\cdot, \cdot; \mathbb{R}^d \setminus \overline{V}))$ given $h|_{\overline V}$.

\end{definition}

\begin{definition}[$\xi$-additive local metric]

Suppose $(h,D)$ is a coupling of a whole-space LGF $h$ and a random continuous length metric on $\mathbb{R}^d$. For $\xi \in \mathbb{R}$, we say that $D$ is a \emph{$\xi$-additive local metric for $h$} if for each $z \in \mathbb{R}^d$ and each $r>0$, the rescaled metric $e^{-\xi h_r(z)}D$ is a local metric for the field $h-h_r(z)$, which is a whole-space LGF normalized so that its average over $\partial B_r(z)$ equals zero.

\end{definition}

We will show the following.

\begin{lemma}\label{lem:locality}
Let $h$ be a whole-space LGF, and let $(h, D_h)$ be any subsequential limit of the laws of the pairs $(h, \mathsf a_\epsilon^{-1} D_h^\epsilon)$. Then $D_h$ is a $\xi$-additive local metric for $h$. More explicitly, fix $z \in \mathbb{R}^d$ and $r > 0,$ and choose the additive constant of $h$ so that the average over $\partial B_r(z)$ equals zero. Let $V \subset \mathbb{R}^d$ be an open set. Then the internal metric $D_h(\cdot, \cdot; V)$ is conditionally independent of the pair $(h, D_h(\cdot, \cdot; \mathbb{R}^d \setminus \overline{V}))$ given $h|_{\overline V}$.
\end{lemma}

\begin{proof}[Proof of Lemma~\ref{lem:locality}]
We argue in several steps, similarly to the proof of Lemma 2.17 in \cite{lqg-metric-estimates}.

\medskip

\noindent\textbf{Step 1: Reduction to $B_r(z) \subset V.$} Suppose Lemma \ref{lem:locality} is true whenever $B_r(z)\subset V.$ Fix any $z_0\in\mathbb{R}^d$ and $r_0>0$ such that $B_{r_0}(z_0) \subset V$, and assume moreover that $h$ is normalized so that $h_{r_0}(z_0)=0.$

Let $z \in \mathbb{R}^d$ and $r > 0$. Let $\tilde{h}$ be the field defined by
\[
\tilde{h} := h - h_r(z),
\]
so that $\tilde{h}$ is a whole-space LGF normalized so that $\tilde{h}_r(z) = 0$. Then Lemma~\ref{lem:Weyl-scaling} implies that $\mathsf{a}_\epsilon^{-1} D_{\tilde h}^\epsilon \to e^{-\xi h_r(z)} \cdot D_h =: D_{\tilde{h}}$ in law along the same subsequence for which $\mathsf{a}_\epsilon^{-1} D_h^\e \to D_h$ in law. We need to show that the internal metric $D_{\tilde h}(\cdot, \cdot; V)$ is conditionally independent of the pair $(\tilde h, D_{\tilde h}(\cdot, \cdot; \mathbb{R}^d \setminus \overline{V}))$ given ${\tilde h}|_{\overline V}$.

Note that $\tilde{h}_{r_0}(z_0) = - h_r(z)$. Since $B_{r_0}(z_0) \subset V$, this implies that $h_r(z) \in \sigma(\tilde h|_{\overline V})$, and hence $h|_{\overline V} = \tilde h|_{\overline V} + h_r(z) \in \sigma(\tilde h|_{\overline V}) $. By assumption, the internal metric $D_h(\cdot, \cdot; V)$ is conditionally independent of the pair $(h, D_h(\cdot, \cdot; \mathbb{R}^d \setminus \overline{V}))$ given $h|_{\overline V}$ (and hence given ${\tilde h}|_{\overline V}$). Note that $\tilde h = h - h_r(z)$ is determined by $h$ and ${\tilde h}|_{\overline V}$. In addition, we have $D_{\tilde{h}}(\cdot, \cdot; V) = e^{-\xi h_r(z)} \cdot D_h(\cdot, \cdot; V)$, so $D_{\tilde{h}}(\cdot, \cdot; V)$ is determined by $\tilde{h}|_{\overline V}$ and $D_h(\cdot, \cdot; V)$. Similarly, $D_{\tilde{h}}(\cdot, \cdot; \mathbb{R}^d \setminus \overline V)$ is determined by $\tilde{h}|_{\overline V}$ and $D_h(\cdot, \cdot; \mathbb{R}^d \setminus \overline V)$. Hence $D_{\tilde h}(\cdot, \cdot; V)$ is conditionally independent of the pair $(\tilde h, D_{\tilde h}(\cdot, \cdot; \mathbb{R}^d \setminus \overline V))$ given $\tilde{h}|_{\overline V}$. This completes the reduction.

\medskip

\noindent\textbf{Step 2: Further reductions.} Recall $\bar h_\epsilon^*$ from Lemma~\ref{lem:equifield}, and let $\bar D_h^\e$ be the associated exponential metric. By Lemma~\ref{lem:equifield}, along the same subsequence, $(h, \mathsf a_\epsilon^{-1} \bar D_h^\e)$ converges in law to $(h, D_h)$. From now we work with $\bar D_h^\e$. 

Assume without loss of generality that $\overline{V} \neq \mathbb{R}^d$ (hence $\mathbb{R}^d \setminus \overline V$ is allowable). Let $B_r(z) \subset V$, and let $h$ be normalized so that $h_r(z)=0$. By Lemma~\ref{lem:markov-whole}, we can then decompose the field $h$ as
\begin{equation}\label{eq:Markov-decomp}
h|_{\mathbb{R}^d\setminus \overline{V}} = \mathfrak{h}+ \mathring{h}
\end{equation}
where $\mathfrak{h}$ is a random $s$-harmonic function on $\mathbb{R}^d\setminus \overline{V}$ which is determined by $h|_{\overline{V}}$, and $\mathring{h}$ is a LGF in $\mathbb{R}^d\setminus \overline {V}$ which is independent of $h|_{\overline{V}}$. 

\medskip

\noindent\textbf{Step 3: LFPP independence.} Fix dyadic domains $W\subset V$ and $W_0\subset\mathbb{R}^d\setminus \overline{V}$. Choose a bump function $\phi$ which is identically $1$ on a neighborhood of $W_0$
and whose support is contained in $\mathbb{R}^d \setminus \overline{V}$.
For $\varepsilon>0$ sufficiently small, the $\varepsilon$-neighborhood of
$W_0$ is contained in $\{\phi=1\}$ and the $\varepsilon$-neighborhood of $W$
is contained in $V$. Define a new field
\[
h' := h - \phi\, \mathfrak{h}.
\]
On the set $\{\phi=1\}\supset W_0$ we have $h' = h - \mathfrak{h} = \mathring{h}$, so $h'$ and $\mathring{h}$ agree there. By locality of the
exponential metrics,
\begin{align*}
  \bar D_{h'}^\varepsilon(\cdot,\cdot;W_0)
  \in \sigma(\mathring{h})\qquad \mbox{and} \qquad \bar D_{h}^\varepsilon(\cdot,\cdot;W)
  \in \sigma(h|_{\overline V})
\end{align*}
for all sufficiently small $\varepsilon$. Since $h|_{\overline V}$ and $\mathring{h}$ are
independent, it follows that for each such $\varepsilon$ 
\[
\bigl(h|_{\overline V},\; \mathsf a_\varepsilon^{-1} \bar D_h^\varepsilon(\cdot,\cdot;W)\bigr)
\quad\text{and}\quad
\bigl(\mathring{h},\; \mathsf a_\varepsilon^{-1}\bar D_{h'}^\varepsilon(\cdot,\cdot;W_0)\bigr)
\]
are independent.

\medskip

\noindent\textbf{Step 4: Passing to the subsequential limit.} 
Let $\varepsilon_k\to0$ be a subsequence along which
$(h,\mathsf a_{\varepsilon_k}^{-1} \bar D_h^{\varepsilon_k})\to(h,D_h)$ in law.
Using Lemma~\ref{lem:dyadicconv}, we may pass to a subsequence $\epsilon_k' \to 0$ along which we have the joint convergence in law
\[
\bigl(h,\mathring{h},
      \mathsf a_{\varepsilon_k'}^{-1} \bar D_h^{\varepsilon_k'},
      \mathsf a_{\varepsilon_k'}^{-1} \bar D_{h'}^{\varepsilon_k'}\bigr)
\;\longrightarrow\;
\bigl(h,\mathring{h},D_h,D_{h'}\bigr),
\]
where $D_{h'} = e^{-\xi \phi \mathfrak{h}}\cdot D_h$ by Weyl scaling (Lemma~\ref{lem:Weyl-scaling}). By Lemma~\ref{lem:dyadicconv}, after possibly passing to a further deterministic subsequence $\epsilon_k'' \to 0$ we can construct continuous metrics $D_{h,W}$ and $D_{h',W_0}$ on $W$ and $W_0$, respectively, such that
\[
D_{h,W}(\cdot,\cdot;W) = D_h(\cdot,\cdot;W),
\qquad
D_{h',W_0}(\cdot,\cdot;W_0) = D_{h'}(\cdot,\cdot;W_0)
\]
and such that the internal metrics
$\mathsf a_{\varepsilon_k''}^{-1} \bar D_h^{\varepsilon_k''}(\cdot,\cdot;W)$ and
$\mathsf a_{\varepsilon_k''}^{-1} \bar D_{h'}^{\varepsilon_k
''}(\cdot,\cdot;W_0)$
converge jointly in law to $D_{h,W}$ and $D_{h',W_0}$. The independence from Step 3 is preserved under convergence in law, so in the
limit we obtain
\[
\bigl(h|_{\overline V}, D_h(\cdot,\cdot;W)\bigr)
\;\perp\!\!\!\perp\;
\bigl(\mathring{h}, D_{h'}(\cdot,\cdot;W_0)\bigr).
\]

\medskip

\noindent\textbf{Step 5: Adding the harmonic part.} By the previous step, $D_h(\cdot, \cdot; W)$ is conditionally independent of $(\mathring{h}, D_{h'}(\cdot, \cdot; W_0))$ given $h|_{\overline V}$. By Lemma~\ref{lem:Weyl-scaling}, a.s.\ we have that $D_h(\cdot, \cdot; W_0) = (e^{\xi \phi \mathfrak{h}} \cdot D_{h'})(\cdot, \cdot; W_0)$. Hence $D_h(\cdot, \cdot; W_0)$ is a measurable function of $\mathfrak{h} \in \sigma(h|_{\overline{V}})$ and $D_{h'}(\cdot, \cdot; W_0)$. Since $h|_{\mathbb{R}^d \setminus \overline
{V}} = \mathring{h} + \mathfrak{h}$ and $\mathfrak{h}$ is determined by $h|_{\overline{V}}$, we see that $h$ is a measurable function of $\mathring{h}$ and $h|_{\overline{V}}$. Thus $D_h(\cdot, \cdot; W)$ is conditionally independent of $(h, D_h(\cdot, \cdot; W_0))$ given $h|_{\overline{V}}$. Letting $W$ increase to $V$ and $W_0$ increase to $\mathbb{R}^d \setminus \overline{V}$ now proves the result.
\end{proof}

\subsection{Each subsequential limit is a measurable function of $h$}

In this section, we prove that any subsequential limit of $\{\mathsf a_\epsilon^{-1} D_h^\epsilon\}_{\epsilon>0}$ as $\epsilon \to 0$ is a measurable function of $h$ (see Lemma~\ref{lem:measurable}). The main idea follows \cite{local-metrics}: if a local metric is determined by $h$ up to bi-Lipschitz equivalence, then it is measurable with respect to $h$ (see Proposition~\ref{prop:measurable}). Along the way, we extend the results of~\cite{local-metrics} to higher dimensions.

\begin{definition}[Joint local metrics]

    Let $h$ be a whole-space LGF, and let $\{D_i\}_{1\leq i \leq n}$ be random continuous length metrics on $\mathbb{R}^d$, coupled with $h$. We say that $\{D_i\}_{1\leq i \leq n}$ are \emph{joint local metrics for $h$} if for every open set $V \subset \mathbb{R}^d$, the internal metrics $\{D_i(\cdot, \cdot; V)\}_{1\leq i \leq n}$ are conditionally independent of the pair $(h, \; \{D_i(\cdot, \cdot; \mathbb{R}^d \setminus \overline{V})\}_{1\leq i \leq n})$ given $h|_{\overline V}$. We say that $\{D_i\}_{1 \leq i \leq n}$ are \emph{joint $\xi$-additive local metrics for $h$} if for each $z \in \mathbb{R}^d$ and each $r>0$, the rescaled metrics $\{e^{-\xi h_r(z)}D_i\}_{1 \leq i \leq n}$ are joint local metrics for the shifted field $h-h_r(z)$.

\end{definition}

Suppose $D$ is a local metric for $h$ and, conditionally on $h$, we sample $D_1$ and $D_2$ independently from the conditional law of $D$. Then $(D_1, D_2)$ are joint local metrics for $h$. The proof is analogous to that of \cite[Lemma 1.4]{local-metrics}, so we omit it. If $D$ is $\xi$-additive, then $(D_1, D_2)$ are joint $\xi$-additive for $h$.

\begin{prop}[Bi-Lipschitz equivalence for local metrics]\label{prop:bi-lipschitz}
    Let $d\ge 2$ and let $h$ be a whole-space LGF on $\mathbb{R}^d$. Let $(D_1, D_2)$ be random continuous length metrics on $\mathbb{R}^d$, coupled with $h$, which are joint $\xi$-additive local metrics for $h$. There exists a universal constant $p \in (0,1)$ with the following property. Suppose there is a deterministic constant $C>0$ such that for every $z\in \mathbb{R}^d$ and every $r > 0$, the event
    $$
    \mathcal{E}(z,r):= \Bigl\{ \sup_{u,v \in \partial B_r(z)} D_2(u,v; B_{2r}(z) \setminus \overline{B_{r/2}(z)}) \leq C \, D_1(\partial B_{r/2}(z), \partial B_r(z)) \Bigr\}
    $$
    holds with probability at least $p$.
    Then almost surely $D_2(x,y) \leq C\, D_1(x,y)$ for all $x,y \in \mathbb{R}^d$.
\end{prop}

This result is the higher-dimension analogue of \cite[Theorem 1.6]{local-metrics}, proved there in Section 4. The proof idea is as follows. By choosing $p$ sufficiently close $1$ and using the near-independence across disjoint concentric shells (see Lemma~\ref{lem:metric-shell-independence} below), we can show that for each point in $\mathbb{R}^d$ there exists a pair $(z,r)$ such that the ball $B_{r/2}(z)$ contains it and the event $\mathcal{E}(z,r)$ occurs. Then, for any path connecting $x$ and $y$, we can use these events to detour the path so that the $D_2$-length of the new path is bounded above by $C$ times the $D_1$-length of the original path.

The $\xi$-additive of local metric allows us to extend the shell independence lemma (Lemma~\ref{lem:shell-independence}) to this setting. Its proof follows verbatim that of Lemma~\ref{lem:shell-independence}, so we omit it.

\begin{lemma}\label{lem:metric-shell-independence}
   Let $h$ be a whole-space LGF on $\mathbb{R}^d$. Let $(D_1, D_2)$ be random continuous length metrics on $\mathbb{R}^d$, coupled with $h$, which are joint $\xi$-additive local metrics for $h$. Fix $q \in (0,1)$ and $q'>1$. Let $\{r_k\}_{k \geq 1}$ be a decreasing sequence of positive numbers such that $r_{k+1}/r_k \leq q$ for all $k \geq 1$. Fix $z \in \mathbb{R}^d$. Let $\{E_{r_k}(z)\}_{k \geq 1}$ be a sequence of events measurable with respect to \[\sigma\big((h - h_{r_k}(z))|_{A_{r_k/q', q'r_k}(z)}\big), \quad e^{-\xi h_{r_k}(z)} D_1(\cdot,\cdot; A_{r_k/q', q'r_k}(z)), \quad \mbox{and} \quad e^{-\xi h_{r_k}(z)} D_2(\cdot,\cdot; A_{r_k/q', q'r_k}(z)).\] For $K \geq 1$, let $\mathcal{N}(K)$ be the number of $k \in [1,K] \cap \mathbb{Z}$ for which $E_{r_k}(z)$ occurs. For any $a>0$ and $b \in (0,1)$, there exists $p = p(a,b,q,q') \in (0,1)$ (independent of the laws of $D_1, D_2$) such that if $\mathbb{P}[E_{r_k}(z)] \geq p$ for all $k \geq 1$, then
    $$
        \mathbb{P}[\mathcal{N}(K) \geq  bK] \geq 1 - e^{-aK} \quad \mbox{for all $K \geq 1$}.
    $$
\end{lemma}

\begin{proof}[Proof of Proposition~\ref{prop:bi-lipschitz}]
    The event $\mathcal{E}(z,r)$ is measurable with respect to
    \[
    e^{-\xi h_r(z)} D_2(\cdot, \cdot; A_{r/2,2r}(z)) \qquad \mbox{and} \qquad e^{-\xi h_r(z)} D_1(\cdot, \cdot; A_{r/2,r}(z)),
    \]
    so we can apply Lemma~\ref{lem:metric-shell-independence} with $E_{r_k}(z) = \mathcal{E}(z,r_k)$ and $r_k = 2^{-k}$. Choosing $a>100d$ and any $b>0$, and letting $p$ be the corresponding constant from Lemma~\ref{lem:metric-shell-independence}, our assumption $\mathbb{P}[\mathcal{E}(z,r)] \geq p$ implies that, for each fixed $z \in \mathbb{R}^d$, with probability at least $1 - O(\epsilon^{10d})$ as $\epsilon \to 0$,
    \begin{equation}\label{eq:prop3.13-event}
    \mbox{there exists $r \in \left[\epsilon, \sqrt{\epsilon}\right] \cap \{2^{-k} : k \in \mathbb{N}\}$} \quad \mbox{such that}\quad \mbox{  $\mathcal{E}(z,r)$ occurs.}
    \end{equation}
    By a union bound, we see that the event in~\eqref{eq:prop3.13-event} holds simultaneously for all $z \in \tfrac{\epsilon}{10d} \mathbb{Z}^d \cap B_{\epsilon^{-1}}(0)$ with polynomially high probability as $\epsilon \to 0$.

    For now on, fix a sequence $\epsilon \to 0$ along which the event in~\eqref{eq:prop3.13-event} holds for every  $z \in \tfrac{\epsilon}{10d} \mathbb{Z}^d \cap B_{\epsilon^{-1}}(0)$. We will show that $D_2(x,y) \leq C D_1(x,y)$ for all $x, y \in \mathbb{R}^d$.
    
    Fix $x,y \in \mathbb{R}^d$ and let $\delta>0$. Let $P$ a path connecting $x$ and $y$ such that \[{\rm len}(P; D_1) \leq D_1(x,y) + \delta.\] We will build a path from $B_{2\sqrt{\epsilon}}(x)$ to $B_{2\sqrt{\epsilon}}(y)$ whose $D_2$-length is bounded by $C \, {\rm len}(P; D_1)$.

    Let $T = {\rm len}(P; D_1)$ and parameterize $P$ by its $D_1$-length. We inductively define a sequence of times $\{t_k\}_{k \geq 0} \subset [0,T]$. Set $t_0 = 0$ and, if $t_k$ has been defined and $t_k<T$, choose $u_k \in \tfrac{\epsilon}{10d} \mathbb{Z}^d \cap B_{\epsilon^{-1}}(0)$ closest to $P(t_k)$ and let $r_k \in \left[\epsilon, \sqrt{\epsilon}\right] \cap \{2^{-m} : m \in \mathbb{N}\}$ be such that $\mathcal{E}(u_k, r_k)$ occurs. By construction, $P(t_k) \in B_{r_k/2}(u_k)$. We then define $t_{k+1}$ to be the first time after $t_k$ at which $P$ leaves $B_{r_k}(u_k)$. If there is no such time, we stop the iteration. Suppose that we obtain $0 = t_0 < t_1 < \ldots < t_M$ in this way. Then, for $0 \leq j \leq M-1$
    \begin{equation}\label{eq:prop3.13-ineq-0}
    t_{j+1} - t_j \geq D_1(\partial B_{r_j/2}(u_j), \partial B_{r_j}(u_j)).
    \end{equation}
    Moreover, the event $\mathcal{E}(u_j, r_j)$ implies that for $0 \leq j \leq M-1$
    \begin{equation}\label{eq:prop3.13-ineq-1}
    \sup_{u,v \in \partial B_{r_j}(u_j)} D_2(u,v; B_{2r_j}(u_j) \setminus \overline{B_{r_j/2}(u_j)}) \leq C \, D_1(\partial B_{r_j/2}(u_j), \partial B_{r_j}(u_j)).
    \end{equation}
    
    We now consider the balls $\{ B_{r_j}(u_j)\}_{0 \leq j \leq M-1}$. By construction $P \subset \cup_{0 \leq j \leq M-1} B_{r_j}(u_j)$ and the latter union is connected. Therefore, the set $\cup_{0 \leq j \leq M-1} \partial B_{r_j}(u_j)$ contains a path connecting $B_{2\sqrt{\epsilon}}(x)$ to $B_{2\sqrt{\epsilon}}(y)$. Let $\widetilde P$ be any such path. Recording the times at which $\widetilde P$ moves from one sphere $\partial B_{r_j}(u_j)$ to another and then performing a loop erasure so that no sphere is visited more than once, we obtain a sequence of points $a_1, a_2,\ldots, a_L$ such that $a_1 \in B_{2\sqrt{\epsilon}}(x)$, $a_L \in B_{2\sqrt{\epsilon}}(y)$, and for each $1 \leq i \leq L - 1$,
    \begin{align*}
    &\mbox{there exists $0 \leq j(i) \leq M - 1$ so that $a_i$ and $a_{i+1}$ both lie on $\partial B_{r_{j(i)}}(u_{j(i)})$,}\\
    &\mbox{and  $j(i_1) \neq j(i_2)$} \quad \mbox{whenever $i_1 \neq i_2$.}
    \end{align*}
    Therefore,
    \begin{align*}
    D_2(B_{2\sqrt{\epsilon}}(x), B_{2\sqrt{\epsilon}}(y)) \leq \sum_{i=1}^{L-1} D_2(a_i, a_{i+1}) \leq \sum_{i=1}^{L-1} \sup_{u,v \in \partial B_{r_{j(i)}}(u_{j(i)})}D_2( u,v; B_{2r_{j(i)}}(u_{j(i)}) \setminus \overline{B_{r_{j(i)/2}}(u_{j(i)})}).
    \end{align*}
    Since the indices $j(i)$ are all different,
    \begin{align*}
    D_2(B_{2\sqrt{\epsilon}}(x), B_{2\sqrt{\epsilon}}(y))  \leq \sum_{j=0}^{M-1} \sup_{u,v \in \partial B_{r_j}(u_j)}D_2( u,v;B_{2r_j}(u_j) \setminus \overline{B_{r_j/2}(u_j)}).
    \end{align*}
    Combining this with~\eqref{eq:prop3.13-ineq-0} and~\eqref{eq:prop3.13-ineq-1} yields 
    \begin{align*}
    D_2(B_{2\sqrt{\epsilon}}(x), B_{2\sqrt{\epsilon}}(y))  &\leq \sum_{j=0}^{M-1} C \,D_1(\partial B_{r_j/2}(u_j), \partial B_{r_j}(u_j)) \leq C  \sum_{j=0}^{M-1} (t_{j+1} - t_j) \\
    &\leq C \,{\rm len}(P; D_1) \leq C(D_1(x,y) + \delta).
    \end{align*}
    First letting $\epsilon \to 0$ and then letting $\delta \to 0$ yields $D_2(x, y) \leq C \,D_1(x,y)$ for all $x,y \in \mathbb{R}^d$. \qedhere
    
\end{proof}

We now show that if a local metric is determined by $h$ up to bi-Lipschitz equivalence, then it is measurable with respect to $h$. The main input is the Efron-Stein inequality.

\begin{prop}[Measurability of local metrics]\label{prop:measurable}
Let $d\ge 2$ and let $h$ be a whole-space LGF. Let $D$ be a random continuous length metric on $\mathbb{R}^d$ coupled with $h$. Assume that $D$ is a local metric for $h$ and that $D$ is determined by $h$ up to bi-Lipschitz equivalence: if we condition on $h$ and sample $D,\widetilde D$ i.i.d.\ from the conditional law of $D$ given $h$, then a.s.\ there exists a random constant $C>1$ (depending on $h$) such that
\[
\widetilde D(z,w)\leq C \, D(z,w),\qquad \forall  z,w\in \mathbb{R}^d.
\]
Then $D$ is a.s.\ determined by $h$ (i.e., we can in fact take $C=1$).
\end{prop}

\begin{proof}
The proof is the higher-dimensional analogue of \cite[Theorem 1.7]{local-metrics}, proved there in Section 5. Let $U \subset \mathbb{R}^d$ be an open bounded connected set. We will show that $D(\cdot, \cdot; U)$ is a.s.\ determined by $h$. Letting $U$ increase to $\mathbb{R}^d$ yields the desired statement. 

The key input is Efron-Stein inequality~\cite{efron-stein}: for any measurable function $F=F(X_1,\dots,X_n)$ of independent random variables $X_1,\dots,X_n$, and any family $X_1',\dots,X_n'$ of i.i.d.\ copies, writing
\[
F^{(i)} := F(X_1,\dots,X_{i-1},X'_i,X_{i+1},\dots,X_n),
\]
we have
\[
\mathrm{Var}[F]\;\le\; \frac{1}{2} \sum_{i=1}^n \mathbb{E}\bigl[(F-F^{(i)})^2\bigr].
\]


\medskip

\noindent\textbf{Step 1: Random shifted grid and path decomposition.}
Fix $\varepsilon>0$. Sample a random shift $\theta$ uniformly from $[0,\varepsilon]^d$ and let $\mathcal G_\theta$ be the union of all coordinate hyperplanes $\{x_k\in\varepsilon\mathbb Z+\theta_k\}$, $k=1,\dots,d$. Let $\mathcal S_\theta^\varepsilon$ be the collection of open boxes of side length $\varepsilon$ determined by $\mathcal G_\theta$. Let $P$ be any (random) path in $U$ with finite $D$-length, chosen measurably from $(h,D)$. Parametrize $P$ by its $D$-length. Exactly as in \cite[Lemma~5.2]{local-metrics}, using that the parameterized path is independent of $\theta$, one proves that the set of times at which $P$ hits $\mathcal G_\theta$ has Lebesgue measure $0$ a.s., and hence
\begin{equation}\label{eq:lem-measure-1}
\len(P;D)\;=\;\sum_{S\in\mathcal S_\theta^\varepsilon}\len\bigl(P\cap S;D\bigr)\qquad\text{a.s.}
\end{equation}
Let $\widetilde D$ be independently sampled from the conditional law of $D$ given by $h$. By assumption, for some random constant $C>1$, we have $C^{-1} D(x,y) \leq \widetilde D(x,y) \leq C D(x,y)$ for all $x,y \in \mathbb{R}^d$. Therefore, if we parametrize $P$ by its $\widetilde D$-length, then the set of times at which $P$ hits $\mathcal G_\theta$ also has Lebesgue measure $0$ a.s., which implies that
\begin{equation}\label{eq:lem-measure-2}
\len(P;\widetilde D)\;=\;\sum_{S\in\mathcal S_\theta^\varepsilon}\len\bigl(P\cap S; \widetilde D\bigr)\qquad\text{a.s.}
\end{equation}

\medskip\noindent\textbf{Step 2: Reconstruction from grid internal metrics.} For $S \in \mathcal{S}_\theta^\varepsilon$, let $D(\cdot,\cdot;S)$ be the internal metric of $D$ on $S$. Let $\mathcal{S}_\theta^\varepsilon(U)$ be the boxes in $\mathcal{S}_\theta^\varepsilon$ that intersect $U$. We claim that $D(\cdot, \cdot; U)$ is a.s.\ determined by
\[
(h,\theta,\{D(\cdot,\cdot;S):S\in\mathcal{S}_\theta^\varepsilon(U)\}).
\]
Indeed, given these data, let $D,D^0$ be two conditionally i.i.d.\ samples with the same internal metrics on each box:
\[
D(\cdot,\cdot;S) = D^0(\cdot,\cdot;S)\quad\forall S \in \mathcal{S}_\theta^\varepsilon(U).
\]
Fix $x,y \in U$ and $\delta>0$. Let $P$ be a path from $x$ and $y$ in $U$ with ${\rm len}(P; D) \leq D(x,y; U) + \delta$. Using~\eqref{eq:lem-measure-1}, \eqref{eq:lem-measure-2}, and the fact that $D$ and $D^0$ agree on each box, we get a.s. 
\[
{\rm len}(P;D) = \sum_S \len(P\cap S;D) \quad \mbox{and} \quad {\rm len}(P;D^0) = \sum_S \len(P\cap S;D^0) = \sum_S \len(P\cap S;D).
\]
Thus ${\rm len}(P; D^0) = {\rm len}(P; D)$ and hence 
\[D^0(x,y; U) \leq {\rm len}(P; D^0) \leq D(x,y; U) +\delta.\] Letting $\delta \to 0$ gives $D^0(x,y; U)\le D(x,y; U)$. The same argument with $D,D^0$ swapped yields $D(x,y;U)\leq D^0(x,y;U)$. Thus $D=D^0$ a.s.\ on rational pairs, and hence everywhere by continuity.

\medskip

\noindent\textbf{Step 3: Independence across boxes from locality.}
By the definition of a local metric (Definition~\ref{def:local-metric}), for any countable family of pairwise disjoint open sets the corresponding internal metrics are conditionally independent given the field. Thus, conditioning on $(h,\theta)$, the family $\bigl\{D( \cdot,\cdot;S): S\in\mathcal S_\theta^\varepsilon(U)\bigr\}$ is conditionally independent.

\medskip

\noindent\textbf{Step 4: Resampling on a single box and Efron-Stein inequality.}
Fix $x,y \in U$ with $D(x,y) < D(x,\partial U)$. For $S\in\mathcal S_\theta^\varepsilon(U)$, define $D^S$ by resampling the internal metric on $S$ from its conditional law given $(h,\theta)$, while keeping $D(\cdot,\cdot; \tilde S)$ fixed for all $\tilde S \in S_\theta^\varepsilon(U) \setminus \{S\}$. Conditioning on $(h,\theta)$ and applying the Efron-Stein inequality with $F = D(x,y; U)$ and $F^{(S)} = D^S(x,y; U)$ gives
\begin{equation}\label{eq:ES}
\begin{aligned}
\Var\bigl[D(x,y;U) \big| h,\theta\bigr] &\le \frac{1}{2}\sum_{S\in\mathcal S_\theta^\varepsilon(U)} \mathbb E\left[\bigl(D^S(x,y;U)-D(x,y;U)\bigr)^2\Big| h,\theta\right] \\
&=\sum_{S\in\mathcal S_\theta^\varepsilon(U)} \mathbb E\left[\bigl(D^S(x,y;U)-D(x,y;U)\bigr)_+^2\Big| h,\theta\right],
\end{aligned}
\end{equation}
where we write $a_+ = \max\{a, 0\}$ for $a \in \mathbb{R}$ and the second line uses symmetry.

Now we show that the right-hand side of~\eqref{eq:ES} tends to 0 as $\epsilon \to 0$. Let $P$ be a path from $x$ to $y$ in $U$ with $\len(P;D)\leq D(x,y;U)+\epsilon^d$. Using~\eqref{eq:lem-measure-1} and~\eqref{eq:lem-measure-2} together with $C^{-1} D \leq D^S \leq C D$, we have a.s.
\begin{align*}
D^S(x,y; U) - D(x,y; U) &\leq {\rm len}(P; D^S) -  D(x,y; U) \\
&\leq {\rm len}(P; D^S) - {\rm len}(P; D) + \epsilon^d\\
&={\rm len}(P \cap S; D^S) - {\rm len}(P \cap S; D) + \epsilon^d \\
&\leq C \, {\rm len}(P \cap S; D) + \epsilon^d.
\end{align*}
Therefore,
\begin{equation}\label{eq:prop3.15-measure-1}
\begin{aligned}
\sum_{S\in\mathcal S_\theta^\varepsilon(U)} \mathbb E\left[\bigl(D^S(x,y;U)-D(x,y;U)\bigr)_+^2\Big| h,\theta\right] &\leq \sum_{S\in\mathcal S_\theta^\varepsilon(U)} \mathbb{E} \left[ (C \, {\rm len}(P \cap S; D) + \epsilon^d)^2 \Big| h,\theta\right] \\
&\leq 2 C^2 \sum_{S\in\mathcal S_\theta^\varepsilon(U)} \mathbb{E} \left[ {\rm len}(P \cap S; D)^2 \Big| h,\theta\right] + 2 \epsilon^{2d} |S_\theta^\varepsilon(U)|.
\end{aligned}
\end{equation}
Note that
\begin{equation}\label{eq:prop3.15-measure-3}
\sum_{S\in\mathcal S_\theta^\varepsilon(U)} \mathbb{E} \left[ {\rm len}(P \cap S; D) \Big| h,\theta\right] = \mathbb{E} \left[ {\rm len}(P; D) | h,\theta\right] \leq \mathbb{E}[D(x,y; U)|h, \theta] + \epsilon^d,
\end{equation}
which is finite due to the bi-Lipschitz equivalence (see Lemma 5.1 in~\cite{local-metrics}). Moreover, since $D$ is continuous, we have
\begin{equation}\label{eq:prop3.15-measure-2}
\max_{S \in S_\theta^\varepsilon(U)} \len(P\cap S;D)\to 0\quad\mbox{a.s. as $\epsilon \to 0$.}
\end{equation}
Indeed, since $D(x,y) < D(x, \partial U)$, the path $P$ stays uniformly bounded away from $\partial U$ as $\epsilon \to 0$. If the left-hand side of \eqref{eq:prop3.15-measure-2} does not converge to $0$, then there exists $\lambda>0$ and a sequence of $\epsilon_n \to 0$ such that for each $\epsilon_n$ we can find $S \in S_\theta^{\varepsilon_n}(U)$ with $\len(P\cap S;D) \geq \lambda$. In that case, we could replace the segment of $P$ inside $S$ by a path whose $D$-length is $o_{\epsilon_n}(1)$. This produces a path connecting $x$ and $y$ with $D$-length at most 
\[D(x,y; U) + \epsilon_n^d - \lambda + o_{\epsilon_n}(1) < D(x,y;U)\]
for all sufficiently small $\epsilon_n$, which is a contradiction. 

Combing~\eqref{eq:prop3.15-measure-3} and~\eqref{eq:prop3.15-measure-2}, we see that the first term on the right-hand side of~\eqref{eq:prop3.15-measure-1} converges to 0 as $\epsilon \to 0$. Since $|S_\theta^\varepsilon(U)| = O(\epsilon^{-d})$, the second term is of order $O(\epsilon^d)$ and thus also converges to 0. Taking $\epsilon \to 0$ in~\eqref{eq:ES}, we conclude that $D(x,y; U)$ is a.s.\ determined by $h$ whenever $D(x,y) < D(x, \partial U)$. 

Finally, by considering rational pairs $(x,y)$ and using that $D(\cdot, \cdot; U)$ is a continuous length metric, we see that $D(\cdot,\cdot; U)$ is a.s. determined by $h$. Letting $U$ increase to $\mathbb{R}^d$ yields the desired result.
\end{proof}

The main result of this section is the following.

\begin{lemma}\label{lem:measurable}
    Let $h$ be a whole-space LGF with the additive constant chosen so that $h_1(0) = 0$, and let $(h, D_h)$ be the limit in law of $(h, \mathsf{a}_\epsilon^{-1}D_h^\epsilon)$ along a sequence $\epsilon_n \to 0$. Then $D_h$ is almost surely determined by $h$. In particular, $\mathsf{a}_\epsilon^{-1}D_h^\epsilon$ converges in probability to $D_h$ along $\epsilon_n \to 0$.
\end{lemma}

\begin{proof}
    Let $p \in (0,1)$ be the universal constant in Proposition~\ref{prop:bi-lipschitz}. Lemma~\ref{lem:locality} implies that $D_h$ is a $\xi$-additive local metric for $h$, and hence also a local metric for $h$. Furthermore, Lemma~\ref{lem:tight-scale} implies that there exists $A>1$ such that for all $z \in \mathbb{R}^d$ and $r>0$,
    \begin{align*}
        &\mathbb{P}\left[ D_h(\partial B_{r/2}(z), \partial B_r(z)) \geq A^{-1} \mathfrak c_r e^{\xi h_r(z)} \right] \geq \frac{1 + p}{2} , \\
        &\mathbb{P}\left[ \sup_{u,v \in \partial B_r(z)} D_h(u,v ; B_{2r}(z) \setminus \overline{\partial B_{r/2}(z)}) \leq A \mathfrak c_r e^{\xi h_r(z)} \right] \geq \frac{1 + p}{2}.
    \end{align*}
    If we condition on $h$ and sample two conditionally independent copies $D_h^{(1)}, D_h^{(2)}$ of $D_h$, then the pair $(D_h^{(1)}, D_h^{(2)})$ is a joint $\xi$-additive local metric for $h$, and the above bounds imply that the events $\mathcal{E}(z,r)$ of Proposition~\ref{prop:bi-lipschitz} hold for $D_1=D_h^{(1)}$ and $D_2=D_h^{(2)}$ with probability at least $p$, with $C=A^2$. Therefore, Proposition~\ref{prop:bi-lipschitz} implies that $D_h^{(1)}$ and $D_h^{(2)}$ are bi-Lipschitz equivalent. In particular, $D_h$ is determined by $h$ up to bi-Lipschitz equivalence.

    Applying Proposition~\ref{prop:measurable} now shows that $D_h$ is a.s.\ determined by $h$. Finally, by a standard measurability argument (see Lemma 1.3 in~\cite{lqg-metric-estimates}), we deduce that $\mathsf{a}_\epsilon^{-1}D_h^\epsilon$ converges in probability to $D_h$ along this subsequence. \qedhere

\end{proof}

\subsection{Proof of Theorem~\ref{thm:axiom}}

In this section we finish the proof of Theorem~\ref{thm:axiom}. Let $h$ be a whole-space LGF with the additive constant chosen so that $h_1(0)=0$.

\medskip\noindent\textbf{Step 1: Defining $D_{\mathsf h}$ for a whole-space LGF plus a bounded function and checking Axioms~\ref{axiom-length}, \ref{axiom-weyl}, \ref{axiom-translation}, \ref{axiom-tight}.}
By Lemma~\ref{lem:tightness-whole-plane}, for every sequence $\varepsilon_n \to 0$ there exists a subsequence $\epsilon_n' \to 0$ along which $(h,\mathsf a_{\varepsilon_n'}^{-1}D^{\varepsilon_n'}_h)$ converges in law to $(h,D_h)$. Lemma~\ref{lem:measurable} implies that $D_h$ is almost surely determined by $h$, and the convergence holds in probability. Therefore, we can pass to a further subsequence $\epsilon_n'' \to 0$ along which $(h,\mathsf a_{\varepsilon_n''}^{-1}D^{\varepsilon_n''}_h)$ converges almost surely to $(h,D_h)$.

Let now $f:\mathbb{R}^d\to\R$ be bounded and continuous. By Lemma~\ref{lem:Weyl-scaling}, almost surely, along $\epsilon_n'' \to 0$, we have $\mathsf a_{\varepsilon_n''}^{-1}D^{\varepsilon_n''}_{h+f} \to e^{\xi f}\cdot D_h$ for every such $f$, where $e^{\xi f}\cdot D_h$ is as defined in~\eqref{eq:weyl}. We set
\[
D_{h+f} := e^{\xi f}\cdot D_h .
\]
Then $D_{h+f}$ is a.s.\ determined by $h+f$ and $\mathsf a_{\varepsilon_n''}^{-1}D^{\varepsilon_n''}_{h+f}$ converges almost surely to $D_{h+f}$ as $\epsilon_n'' \to 0$.

Thus we obtain a measurable map $\mathsf h\mapsto D_{\mathsf h}$ from generalized functions to continuous metrics, a.s.\ defined whenever $\mathsf h$ is a whole-space LGF plus a bounded continuous function. By Lemma~\ref{lem:dyadicconv}, $D_\mathsf h$ is a length metric (Axiom~\ref{axiom-length}). Weyl scaling for bounded continuous $f$ (Axiom~\ref{axiom-weyl}) follows from the definition. Axiom~\ref{axiom-translation} is immediate from the almost sure convergence. Lemma~\ref{lem:tight-scale} yields the tightness across scales (Axiom~\ref{axiom-tight}).

\medskip\noindent\textbf{Step 2: Locality (Axiom~\ref{axiom-local}) for a whole-space LGF plus a bounded function.} Let $\mathsf h$ be a whole-space LGF plus a bounded continuous function. We now show that for an open set $V \subset \mathbb{R}^d$, $D_{\mathsf h}(\cdot,\cdot; V)$ is a.s.\ determined by $\mathsf h|_V$. Let $O$ be a bounded open subset of $V$. It suffices to show that for $u, v \in O$ with $D_{\mathsf h}(u,v) < D_{\mathsf h}(u, \partial O)$, the internal distance $D_{\mathsf h}(u,v; O)$ is a.s.\ determined by $\mathsf h|_V$. Then we can increase $O$ to $V$ and deduce that $D_{\mathsf h}(\cdot,\cdot; V)$ is a.s.\ determined by $\mathsf h|_V$.

Let $\bar{\mathsf h}_\epsilon^*$ be the convolution of $\mathsf h$ with $\Psi_\epsilon \mathsf K_\epsilon$ as in~\eqref{eq:sec2-field2}, and let $\bar D_{\mathsf h}^\e$ be the exponential metric associated with it. Lemma~\ref{lem:equifield} implies that $\lim_{\epsilon_n'' \to 0} \sup_{z \in \overline{O}} |\bar{\mathsf h}_\epsilon^*(z) - \mathsf h_\epsilon^*(z)| = 0$. Therefore, the almost sure convergence of $\mathsf a_{\varepsilon_n''}^{-1}D^{\varepsilon_n''}_{\mathsf h}$ to $D_{\mathsf h}$ implies that
\[
\lim_{\epsilon_n'' \to 0} \mathsf a_{\varepsilon_n''}^{-1}\bar D^{\varepsilon_n''}_{\mathsf h}(u,v; O) = D_{\mathsf h}(u,v; O) \quad \mbox{almost surely}.
\]
For all sufficiently small $\epsilon_n''$, the left-hand side is determined by $\mathsf h|_V$, and hence $D_{\mathsf h}(u,v; O)$ is also determined by $\mathsf h|_V$. Increasing $O$ to $V$ yields the desired result. This checks Axiom~\ref{axiom-local} for $\mathsf h$. 

\medskip\noindent\textbf{Step 3: Extension to a whole-space LGF plus a general continuous function.}
Let $h$ be a whole-plane GFF and let $f$ be a (possibly random) unbounded continuous function. For a bounded open set $V\subset\mathbb{R}^d$, choose a smooth compactly supported function $\phi$ with $\phi\equiv 1$ on $V$ and define
\[
D^{V}_{h+f} := D_{h+\phi f}(\cdot,\cdot;V).
\]
By locality (Axiom~\ref{axiom-local}) in the bounded case, $D^{V}_{h+f}$ depends only on $(h+f)|_V$ and is independent of the choice of $\phi$.  

Given a continuous path $P$ contained in some bounded $V$, define its $D_{h+f}$-length to be the $D^{V}_{h+f}$-length of $P$; this does not depend on $V$ by construction. For $z,w\in\mathbb{R}^d$ set $D_{h+f}(z,w)$ to be the infimum of the $D_{h+f}$-lengths of continuous paths from $z$ to $w$. Then $D_{h+f}$ is a length metric on $\mathbb{R}^d$ which is a.s.\ determined by $h+f$, and for each bounded $V$ we have $D_{h+f}(\cdot,\cdot;V)=D^{V}_{h+f}$, hence determined by $(h+f)|_V$, which implies Axiom~\ref{axiom-local}. Axioms~\ref{axiom-weyl},~\ref{axiom-translation}, and~\ref{axiom-tight} are inherited from the bounded case.

\medskip
Combining the three steps, for any sequence $\varepsilon_n\to0$, we have constructed a map $h\mapsto D_h$, a.s.\ defined for every whole-space LGF plus a continuous function, which satisfies Axioms~\ref{axiom-length}--\ref{axiom-tight} and such that there exists a subsequence $\epsilon_n'' \to 0$ with the property for any $\mathsf h$ that is a whole-space LGF plus a bounded continuous function,
\[
\lim_{\epsilon_n'' \to 0} \mathsf a_{\varepsilon_n''}^{-1}D^{\varepsilon_n''}_{\mathsf h} = D_{\mathsf h} \quad \mbox{in probability},
\]
which is exactly the statement of Theorem~\ref{thm:axiom}.

\section{Quantitative properties of weak exponential metrics}\label{sec:moment}

In this section we establish quantitative properties of weak exponential metrics. In Section~\ref{subsec:c-asymp} we prove Theorem~\ref{thm:sharp-c}. We prove Theorem~\ref{thm:moments}, which records moment bounds for several types of distances, in Sections~\ref{subsec:superpolynomial}--\ref{subsec:moment-other}. In Section~\ref{subsec:holder} we prove Theorem~\ref{thm:holder-continuous}. In Section~\ref{sec:KPZ}, we prove Theorem~\ref{thm:kpz}. 

Throughout this section we fix $\gamma \in (0,\sqrt{2d})$ and let $\xi, Q, \mathsf d_\gamma$ be as defined in~\eqref{eq:def-xi-q}. Let $D_h$ be a weak $\gamma$-exponential metric, and let $h$ be a whole-space LGF with additive constant chosen so that $h_1(0)=0$.

\subsection{Up-to-constants bounds for the scaling constants $\mathfrak c_r$}\label{subsec:c-asymp}

In this section we prove Theorem~\ref{thm:sharp-c}. The argument follows the planar strategy of~\cite{dg-constant} with modifications in higher dimensions. We also show that any two weak $\gamma$-exponential metrics are up-to-constants equivalent (Proposition~\ref{prop:up-to-constant-eq}), a fact that may be useful for establishing uniqueness.

Recall from~\eqref{eq:def-exponential-metric} that $D_h^\epsilon$ is the exponential metric built from the mollification $h_\epsilon^*$. In Section~\ref{subsec:localize} we introduced a localized truncation $\bar h_\epsilon^*$, depending only on the $\sqrt{\epsilon}$-neighborhood of its argument. To prove Theorem~\ref{thm:sharp-c}, we need to consider a different truncation depending on a smaller neighborhood. Define
\begin{equation}\label{eq:def-widehat-h-epsilon}
\widehat h_\epsilon^*(z) := Z_\epsilon^{-1} \int_{\mathbb{R}^d} h(w) \widehat\Psi_\epsilon(z-w) \mathsf K_\epsilon(z-w)\, dw \qquad \mbox{for } z \in \mathbb{R}^d,
\end{equation}
where $\widehat\Psi_\epsilon\colon \mathbb{R}^d \to [0,1]$ is a radial bump function with $\widehat\Psi_\epsilon(u)=1$ for $|u|\le \tfrac{1}{2}\epsilon (\log \epsilon^{-1})^{10}$ and $\widehat\Psi_\epsilon(u)=0$ for $|u|\ge \epsilon (\log \epsilon^{-1})^{10}$, and
\[
Z_\epsilon := \int_{\mathbb{R}^d} \widehat\Psi_\epsilon(u) \mathsf K_\epsilon(u)\, du
\]
is a normalizing constant. This normalization makes $\widehat h_\epsilon^*$ shift by the same additive constant as $h$. Let $\widehat D_h^\epsilon$ denote the exponential metric associated with $\widehat h_\epsilon^*$.

The next lemma is the analogue of Lemma~\ref{lem:equifield} for $\widehat h_\epsilon^*$.

\begin{lemma}\label{lem:equifield-sec4}
Let $h$ be a whole-space LGF with additive constant chosen so that $h_1(0)=0$. For any bounded open set $U \subset \mathbb{R}^d$, almost surely
\[
\lim_{\epsilon \to 0} \sup_{z \in \overline U} \big| h_\epsilon^*(z) - \widehat h_\epsilon^*(z) \big| = 0.
\]

\end{lemma}

\begin{proof}
The argument parallels the proof of Lemma~\ref{lem:equifield}, using Lemma~\ref{lem:sphereavg}. By Lemma~\ref{lem:sphereavg} there is a random constant $C<\infty$ such that
\[
|h_r(z)| \le C \max\{\log\frac{1}{r}, \log r, 1\} \qquad \mbox{for all } z \in \overline U \mbox{ and } r>0.
\]
Using polar coordinates and radiality of $\widehat\Psi_\epsilon$ and $\mathsf K_\epsilon$, we have
\[
h_\epsilon^*(z) = \frac{2\pi^{d/2}}{\Gamma(d/2)} \int_0^\infty r^{d-1} h_r(z) \mathsf K_\epsilon(r)\, dr
\qquad \mbox{and} \qquad
Z_\epsilon \widehat h_\epsilon^*(z) = \frac{2\pi^{d/2}}{\Gamma(d/2)}\int_0^\infty r^{d-1} h_r(z) \widehat\Psi_\epsilon(r) \mathsf K_\epsilon(r)\, dr.
\]
Hence, for $z \in \overline U$,
\begin{align*}
\big| Z_\epsilon \widehat h_\epsilon^*(z) - h_\epsilon^*(z) \big|
&\leq \frac{2\pi^{d/2}}{\Gamma(d/2)} \int_{\epsilon \log(\epsilon^{-1})^{10}/2}^\infty r^{d-1} |h_r(z)| |\mathsf K_\epsilon(r)|\, dr \\
&\le C \int_{\epsilon \log(\epsilon^{-1})^{10}/2}^\infty r^{d-1} \max\{\log\frac{1}{r}, \log r, 1\}\, \epsilon^{-d} |\mathsf K(r/\epsilon)|\, dr.
\end{align*}
By Lemma~\ref{lem:kernel}, there exists $C'>0$ with $|\mathsf K(x)| \le C' |x|^{-2d+1}$ for all $x \in \mathbb{R}^d$. Using this and substituting $r = \epsilon u$, we get
\begin{equation}\label{eq:lem4.1-1}
\begin{aligned}
\big| Z_\epsilon \widehat h_\epsilon^*(z) - h_\epsilon^*(z) \big| &\leq C \int_{(\log \epsilon^{-1})^{10}/2}^\infty u^{d-1} \max\{\log\frac{1}{\epsilon u}, \log \epsilon u, 1\} u^{-2d+1} \,du \\
&\leq C \int_{(\log \epsilon^{-1})^{10}/2}^\infty u^{d-1} \,\sqrt{u}  \,u^{-2d+1} \,du \leq (\log \epsilon^{-1})^{-2},
\end{aligned}
\end{equation}
which tends to $0$ as $\epsilon \to 0$, uniformly in $z \in \overline U$.

Similarly, for all sufficiently small $\epsilon$,
\[
|Z_\epsilon - 1| = \left| \int_{\mathbb{R}^d} \big(\widehat\Psi_\epsilon(u)-1\big)\, \mathsf K_\epsilon(u)\, du \right| \le (\log \epsilon^{-1})^{-2}.
\]
Moreover, $\mathrm{Var}(h_\epsilon^*(z)) = \log \epsilon^{-1} + O(1)$ implies that, almost surely for all sufficiently small $\epsilon$,
\[
\sup_{z \in \overline U} |h_\epsilon^*(z)| \le (\sqrt{2d} + 1)\log \epsilon^{-1}.
\]
Combining these estimates and using the triangle inequality gives
\[
\sup_{z \in \overline U} \big| h_\epsilon^*(z) - \widehat h_\epsilon^*(z) \big|
\le \sup_{z \in \overline U} \big| Z_\epsilon \widehat h_\epsilon^*(z) - h_\epsilon^*(z) \big|
+ |Z_\epsilon - 1|\, Z_\epsilon^{-1} \sup_{z \in \overline U} \big| Z_\epsilon \widehat h_\epsilon^*(z) \big| \to 0 \qquad \mbox{as } \epsilon \to 0. \qedhere
\]
\end{proof}

Recall $D_h^\epsilon$ and $\mathsf a_\epsilon$ from~\eqref{eq:def-exponential-metric} and~\eqref{eq:def-a-epsilon}. The main step in proving Theorem~\ref{thm:sharp-c} is Proposition~\ref{prop:compare-lqg-lfpp} below, which compares $\mathsf a_\epsilon^{-1} D_h^\epsilon$ and $D_h$. Based on Proposition~\ref{prop:compare-lqg-lfpp}, we will then show that $\mathfrak c_r \mathfrak c_s \asymp \mathfrak c_{rs}$ with constants depending only on $D_h$, which implies Theorem~\ref{thm:sharp-c}. To prove Proposition~\ref{prop:compare-lqg-lfpp}, we rely on the shell independence lemma. By comparing the metric $\mathsf a_\epsilon^{-1} \widehat D_h^\epsilon$ (which is close to $\mathsf a_\epsilon^{-1} D_h^\epsilon$ by Lemmas~\ref{lem:sec4.1-close-to-1}) and $D_h$ across shells with radii in $[\epsilon^{1-\zeta/2}, \epsilon^{1-\zeta}]$ and showing that regularity events occur around each point with high probability, we can compare them globally. One subtlety is that the scaling constants for $\mathsf a_\epsilon^{-1} \widehat D_h^\epsilon$ and $D_h$ differ by the factor $\tfrac{r \mathsf a_{\epsilon/r}}{\mathfrak c_r \mathsf a_\epsilon}$, where $r$ is the radius, and there is no \textit{a prior} bound on this ratio. We argue by contradiction that at a large fraction of scales these scaling constants are up-to-constants equivalent; see Lemma~\ref{lem:good-ratios}.

\begin{prop} \label{prop:compare-lqg-lfpp}
Fix $\mathfrak K$ satisfying~\eqref{eq:asmp-kernel} and choose $\mathsf K$ as in Lemma~\ref{lem:kernel}. For each $\zeta \in (0,1)$, there exists a constant $C > 0$, depending only on $\zeta$, $\mathfrak K$, and the law of $D_h$, such that the following is true. 
Let $U\subset \mathbb{R}^d$ be a bounded connected open set.   With probability tending to 1 as $\epsilon \to 0$, 
\begin{equation}\label{eq:prop4.2-1}
\mathsf a_\epsilon^{-1} D_h^\epsilon ( B_{ \epsilon^{1-\zeta}}(z) , B_{ \epsilon^{1-\zeta}}(w) ; B_{ \epsilon^{1-\zeta}}(U) ) 
\leq C\, D_h(z,w ; U)  ,\quad \forall z,w\in U
\end{equation}
and
\begin{equation}\label{eq:prop4.2-2}
D_h ( B_{ \epsilon^{1-\zeta}}(z) , B_{ \epsilon^{1-\zeta}}(w) ; B_{ \epsilon^{1-\zeta}}(U) ) 
\leq C\, \mathsf  a_\epsilon^{-1} D_h^\epsilon(z,w ; U) ,\quad \forall z,w \in U . 
\end{equation}
\end{prop}

We now prove Proposition~\ref{prop:compare-lqg-lfpp}. Fix $\zeta \in (0,1)$. For $z\in \mathbb{R}^d$, $\epsilon \in (0,1)$, $r\in (\epsilon,1)$, and $C>0$, let $E_r^\epsilon(z;C)$ be the event that the following holds:
\begin{equation}\label{eq:def-Er(z,C)}
\begin{aligned}
D_h( \mbox{across $A_{r/2,r}(z)$} ) &\geq C^{-1} \mathfrak c_r e^{\xi h_r(z)}, \\
\sup_{u,v \in \partial B_r(z)} D_h(u,v; A_{r/2,2r}(z)) &\leq C \mathfrak c_r e^{\xi h_r(z)}, \\
\mathsf a_\epsilon^{-1} \widehat{D}_h^\epsilon( \mbox{across $A_{r/2,r}(z)$} ) &\geq C^{-1} \frac{r \mathsf a_{\epsilon/r}}{\mathsf a_\epsilon}   e^{\xi h_r(z)}, \\
\mathsf a_\epsilon^{-1} \sup_{u,v \in \partial B_r(z)} \widehat{D}_h^\epsilon(u,v; A_{r/2,2r}(z)) &\leq C \frac{r  \mathsf a_{\epsilon/r}}{\mathsf a_\epsilon}   e^{\xi h_r(z)}  .
\end{aligned}
\end{equation}
Note that the roles of $\mathsf a_\epsilon^{-1} \widehat{D}_h^\epsilon$ and $D_h$ are simply interchanged. 

\begin{lemma}\label{lem:sec4.1-close-to-1}
    For fixed $p \in (0,1)$, we can choose $C >0$ sufficiently large (depending only on $p$, $\mathfrak K$, and the law of $D_h$) such that
\begin{equation}\label{eq:sec4.1-close-to-1}
\mathbb{P}[E_r^\epsilon(z;C)] \geq p, \quad \forall 0<\epsilon < r < 1, \quad z \in \mathbb{R}^d.
\end{equation}
\end{lemma}
\begin{proof}
    Since $D_h$ and $\widehat D_h^\epsilon$ are both translation invariant in law by Axiom~\ref{axiom-translation} and definition, it suffices to consider the case $z = 0$. Estimates for the first two events in~\eqref{eq:def-Er(z,C)} for $D_h$ follow directly from Axiom~\ref{axiom-tight}. Now we consider the last two events for $\mathsf a_\epsilon^{-1} \widehat{D}_h^\epsilon$. Define $h^r ( \cdot) := h(r \cdot) - h_r(0)$. By Lemma~\ref{lem:scale-invariance}, we have $h^r \overset{d}{=} h$. Recall from~\eqref{eq:scaled-h-r} that
    $$
    \mathsf a_\epsilon^{-1} D_h^\epsilon(ru, rv) = \frac{r \mathsf a_{\epsilon/r}}{\mathsf a_\epsilon} e^{\xi h_r(0)} \times \mathsf a_{\epsilon/r}^{-1} D_{h^r}^{\epsilon/r}(u,v) \qquad \forall u,v \in \mathbb{R}^d.
    $$
    Recall from the proof of Lemma~\ref{lem:tight-scale} that the family $\{\mathsf a_{\varepsilon/r}^{-1} D_{h^r}^{\varepsilon/r}\}_{0<\varepsilon<r,\ r>0}$ is tight and any subsequential limit is a continuous metric on $\mathbb{R}^d$. Therefore, the laws of \begin{equation}\label{eq:lem4.3-tight} \Big{\{} \mathsf a_\epsilon^{-1} D_h^\epsilon(r\cdot , r\cdot)\Big{/}\Big{(} \frac{r \mathsf a_{\epsilon/r}}{\mathsf a_\epsilon} e^{\xi h_r(0)} \Big{)} \Big{\}}_{0 < \epsilon < r,\  r>0}\end{equation} are also tight and any subsequential limit induces the Euclidean topology. This implies that the last two events in~\eqref{eq:def-Er(z,C)}, with $\widehat{D}_h^\epsilon$ replaced by $D_h^\epsilon$, occur with probability arbitrarily close to 1 as $C \to \infty$, uniformly over $\epsilon$ and $r$. Applying Lemma~\ref{lem:equifield-sec4}, the same conclusion applies to $\widehat{D}_h^\epsilon$.
\end{proof}

Define
\begin{equation}
\mathcal{N}^\epsilon := [\epsilon^{1-\zeta/2}, \epsilon^{1-\zeta}/10] \cap \{ 2^{-k} : k \in \mathbb{N} \}.
\end{equation}
We now apply the shell independence lemma (Lemma~\ref{lem:shell-independence}) to this set of radii to compare $\mathsf a_\epsilon^{-1} \widehat D_h^\epsilon$ and $D_h$; see Proposition~\ref{prop:bi-lipschitz} for similar arguments. 

\begin{lemma} \label{lem:compare-subset}
There exists a constant $C > 0$, depending only on $\zeta$, $\mathfrak K$, and the law of $D_h$, such that the following is true. Let $U\subset \mathbb{R}^d$ be a bounded connected open set. Let $\epsilon \in (0,1)$ and let $\mathcal{R}^\epsilon$ be a deterministic subset of $\mathcal{N}^\epsilon$ with $|\mathcal{R}^\epsilon| \geq \tfrac{1}{10} |\mathcal{N}^\epsilon|$. Then, with probability tending to 1 as $\epsilon \to 0$, at a rate uniform in $\mathcal{R}^\epsilon$, we have for all $z,w \in U$
\begin{equation} \label{eq:lem4.7-compare-upper}
\mathsf a_\epsilon^{-1} \widehat D_h^\epsilon ( B_{ \epsilon^{1-\zeta}}(z) , B_{  \epsilon^{1-\zeta}}(w) ; B_{ \epsilon^{1-\zeta}}(U) ) 
\leq C \left(\max_{r \in \mathcal{R}^\epsilon } \frac{r \mathsf a_{\epsilon/r}}{\frk c_r \mathsf a_\epsilon} \right) D_h(z,w ; U) 
\end{equation}
and
\begin{equation} \label{eq:lem4.7-compare-lower}
D_h ( B_{ \epsilon^{1-\zeta}}(z) , B_{ \epsilon^{1-\zeta}}(w) ; B_{ \epsilon^{1-\zeta}}(U) ) 
\leq C \left(\min_{r \in \mathcal{R}^\epsilon } \frac{r \mathsf a_{\epsilon/r}}{\mathfrak c_r \mathsf a_\epsilon} \right)^{-1}  \mathsf a_\epsilon^{-1} \widehat D_h^\epsilon(z,w ; U) .
\end{equation}
\end{lemma}

\begin{proof}
    By Axioms~\ref{axiom-weyl} and~\ref{axiom-local}, the first two events in~\eqref{eq:def-Er(z,C)} for $D_h$ are determined by $(h - h_r(z))|_{A_{r/2, 2r}(z)}$. By the definition of $\widehat h_\epsilon$ from~\eqref{eq:def-widehat-h-epsilon}, the last two events in~\eqref{eq:def-Er(z,C)} for $\widehat D_h^\epsilon$ is determined by $(h - h_r(z))|_{A_{r/4, 4r}(z)}$ for $r \in \mathcal{N}^\epsilon$ and all sufficiently small $\epsilon$. Therefore, similarly to~\eqref{eq:prop3.13-event}, we can use the near-independence across shells (Lemma~\ref{lem:shell-independence}) and~\eqref{eq:sec4.1-close-to-1} to show that there exists $C >0$ (depending only on $\zeta, \mathfrak K$, and $D_h$) such that with polynomially high probability as $\epsilon \to 0$, at a rate uniform in $\mathcal R^\epsilon$, the following event holds:
    \begin{equation}\label{eq:lem4.4-1}
    \mbox{for each $u \in \epsilon \mathbb{Z}^d \cap B_{\epsilon^{-1}}(0)$, there exists $r \in \mathcal{R}^\epsilon$ such that $E_r^\epsilon(u;C)$ occurs.}
    \end{equation}
    
    For now on, we assume~\eqref{eq:lem4.4-1} holds and show~\eqref{eq:lem4.7-compare-upper} and~\eqref{eq:lem4.7-compare-lower}. We will only prove~\eqref{eq:lem4.7-compare-upper}, and~\eqref{eq:lem4.7-compare-lower} follows from essentially the same argument with the roles of $\mathsf a_\epsilon^{-1} \widehat D_h^\epsilon$ and $D_h$ interchanged. Fix $z,w \in U$. Let $P$ a path connecting $z$ and $w$ in $U$ such that ${\rm len}(P; D_h) \leq 2 D_h(z,w; U)$. We will build a path from $B_{\epsilon^{1-\zeta}}(z)$ to $B_{\epsilon^{1-\zeta}}(w)$ in $B_{\epsilon^{1-\zeta}}(U)$ whose $\mathsf a_\epsilon^{-1} \widehat D_h^\epsilon$-length can be upper-bounded. 

    We parameterize $P$ by its $D_h$-length and let $T = {\rm len}(P; D_h)$. We inductively define a sequence of times $\{t_k\}_{k \geq 0} \subset [0,T]$. Let $t_0 = 0$. If $t_k$ has been defined and $t_k<T$, choose $u_k \in \epsilon \mathbb{Z}^d \cap B_{\epsilon^{-1}}(0)$ closest to $P(t_k)$ and let $r_k \in \mathcal{R}^\epsilon$ such that $E_{r_k}^\epsilon(u_k;C)$ occurs. By construction, $P(t_k) \in B_{r_k/2}(u_k)$. We then define $t_{k+1}$ to be the first time after $t_k$ at which $P$ leaves $B_{r_k}(u_k)$. If there is no such time, we stop the iteration. Suppose that we obtain $0 = t_0 < t_1 < \ldots < t_M$ in this way. Then, for $0 \leq j \leq M-1$,
    \[
    t_{j+1} - t_j \geq D_h(\mbox{across $A_{r_j/2,r_j}(u_j)$}).
    \]
    Moreover, the first and fourth inequalities in the event $E_{r_j}^\epsilon(u_j;C)$ imply that for $0 \leq j \leq M-1$,
    \begin{align*}
        \mathsf a_\epsilon^{-1} \sup_{u,v \in \partial B_{r_j}(u_j)} \widehat{D}_h^\epsilon(u,v; A_{r_j/2,2r_j}(u_j)) &\leq C \frac{r_j  \mathsf a_{\epsilon/r_j}}{\mathsf a_\epsilon}   e^{\xi h_{r_j}(u_j)} \leq C^2 \frac{r_j  \mathsf a_{\epsilon/r_j}}{\mathfrak c_{r_j} \mathsf a_\epsilon} D_h(\mbox{across $A_{r_j/2,r_j}(u_j)$}).
    \end{align*}

    Similarly to the proof of Proposition~\ref{prop:bi-lipschitz}, we can construct a path $\widetilde P$ connecting $B_{\epsilon^{1-\zeta}}(z)$ and $B_{\epsilon^{1-\zeta}}(w)$ in $\cup_{0 \leq j \leq M-1} B_{2r_j}(u_j)$ (hence in $B_{\epsilon^{1-\zeta}}(U)$) such that 
    \[
    {\rm len}(\widetilde P; \mathsf a_\epsilon^{-1} \widehat{D}_h^\epsilon) \leq \sum_{j=0}^{M-1} \mathsf a_\epsilon^{-1} \sup_{u,v \in \partial B_{r_j}(u_j)} \widehat{D}_h^\epsilon(u,v; A_{r_j/2,2r_j}(u_j)).
    \]
    Combining this with the preceding inequalities yields
    \begin{align*}
        \mathsf a_\epsilon^{-1} \widehat{D}_h^\epsilon(B_{\epsilon^{1-\zeta}(z)}, B_{\epsilon^{1-\zeta}}(w); B_{\epsilon^{1-\zeta}}(U)) &\leq C^2 \left(\max_{r \in \mathcal{R}^\epsilon } \frac{r \mathsf a_{\epsilon/r}}{\frk c_r \mathsf a_\epsilon} \right) \sum_{j=0}^{M-1} D_h(\mbox{across $A_{r_j/2,r_j}(u_j)$}) \\
        &\leq C^2 \left(\max_{r \in \mathcal{R}^\epsilon } \frac{r \mathsf a_{\epsilon/r}}{\frk c_r \mathsf a_\epsilon} \right) \,{\len}(P; D_h) \leq  2 C^2 \left(\max_{r \in \mathcal{R}^\epsilon } \frac{r \mathsf a_{\epsilon/r}}{\frk c_r \mathsf a_\epsilon} \right)  \,D_h(z,w; U),
    \end{align*}
    which proves~\eqref{eq:lem4.7-compare-upper}. Inequality~\eqref{eq:lem4.7-compare-lower} follows from a similar argument. \qedhere
\end{proof}

In order to deduce Proposition~\ref{prop:compare-lqg-lfpp} from Lemma~\ref{lem:compare-subset}, it suffices to find $\mathcal{R}^\epsilon \subset \mathcal{N}^\epsilon$ such that $\tfrac{r \mathsf a_{\epsilon/r}}{\mathfrak c_r \mathsf a_\epsilon}$ is bounded from above and below, which is in fact a consequence of Lemma~\ref{lem:compare-subset} due to the arbitrary choice of $\mathcal{R}^\epsilon$.

\begin{lemma} \label{lem:good-ratios}
There exists $C > 1$, depending only on $\zeta$, $\mathfrak K$, and the law of $D_h$, such that for each $\epsilon \in (0,1)$, there are at least $|\mathcal{N}^\epsilon|/2$ values of $r\in \mathcal{N}^\epsilon$ such that 
\begin{equation} \label{eq:good-ratios}
C^{-1} \leq \frac{r \mathsf a_{\epsilon/r}}{\mathfrak c_r \mathsf a_\epsilon} \leq C.
\end{equation} 
\end{lemma}

\begin{proof}
    Take $z = 0$, $w = 2 e_1 = (2,0,\ldots, 0)$ and $U = B_4(0)$ in~\eqref{eq:lem4.7-compare-upper} from Lemma~\ref{lem:compare-subset}. We have for any deterministic $\mathcal{R}^\epsilon \subset \mathcal{N}^\epsilon$ with $|\mathcal{R}^\epsilon| \geq \tfrac{1}{10} |\mathcal{N}^\epsilon|$, with probability tending to 1 as $\epsilon \to 0$, independent of the choice of $\mathcal{R}^\epsilon$,
    \begin{equation}\label{eq:lem4.4}
    \mathsf a_\epsilon^{-1} \widehat D_h^\epsilon ( \partial B_{1/2}(0) , \partial B_1(0)) \leq \mathsf a_\epsilon^{-1} \widehat D_h^\epsilon ( B_{ \epsilon^{1-\zeta}}(z) , B_{  \epsilon^{1-\zeta}}(w) ; B_{ \epsilon^{1-\zeta}}(U) ) 
    \leq C \left(\max_{r \in \mathcal{R}^\epsilon } \frac{r \mathsf a_{\epsilon/r}}{\frk c_r \mathsf a_\epsilon} \right) D_h(z,w ; U) .
    \end{equation}
    By Lemmas~\ref{lem:tightness-whole-plane} and~\ref{lem:equifield-sec4}, we know that as $\epsilon \to 0$, $\mathsf a_\epsilon^{-1} \widehat D_h^\epsilon $ is tight with respect to the local uniform topology and any subsequential limit induces the Euclidean topology. In particular, this implies that as $\epsilon \to 0$, $\mathsf a_\epsilon^{-1} \widehat D_h^\epsilon ( \partial B_{1/2}(0) , \partial B_1(0)) $ is tight and bounded away from 0 in probability. Therefore, there exists $c>0$ such that for at least $\tfrac{9}{10} |\mathcal{N}^\epsilon|$ radii in $\mathcal{N}^\epsilon$, we have
    \[
    \frac{r \mathsf a_{\epsilon/r}}{\frk c_r \mathsf a_\epsilon} \geq c.
    \]
    Otherwise, we could choose $\mathcal{R}^\epsilon$ so that $\max_{r \in \mathcal{R}^\epsilon} \frac{r \mathsf a_{\epsilon/r}}{\frk c_r \mathsf a_\epsilon} < c$, yielding a contradiction with~\eqref{eq:lem4.4}. Similar we can use~\eqref{eq:lem4.7-compare-lower} to show that for at least $\tfrac{9}{10} |\mathcal{N}^\epsilon|$ radii in $\mathcal{N}^\epsilon$, we have
    \[
    \frac{r \mathsf a_{\epsilon/r}}{\frk c_r \mathsf a_\epsilon} \leq c^{-1},
    \] 
    where we decrease $c$ if necessary. This concludes the proof.
\end{proof}

    We now finish the proof of Proposition~\ref{prop:compare-lqg-lfpp}.
\begin{proof}[Proof of Proposition~\ref{prop:compare-lqg-lfpp}]
    By Lemma~\ref{lem:good-ratios}, we can pick $\mathcal{R}^\epsilon \subset \mathcal{N}^\epsilon$ with $|\mathcal{R}^\epsilon|\geq \tfrac{1}{10} |\mathcal{N}^\epsilon|$ such that~\eqref{eq:good-ratios} holds. Applying Lemma~\ref{lem:compare-subset} with this choice of $\mathcal{R}^\epsilon$ completes the proof.
\end{proof}

    We now proceed with the proof of Theorem~\ref{thm:sharp-c}.
\begin{proof}[Proof of Theorem~\ref{thm:sharp-c}]
    \textbf{Step 1.} Let $\zeta = 1/2$ and fix any $\mathfrak K$ satisfying~\eqref{eq:asmp-kernel}. We first show that there exists $C>1$, depending only on $\mathfrak K$ and the law of $D_h$, such that for all $R\geq 1$, the following holds when $\epsilon$ is sufficiently small (which may depend on $R$):
    \begin{equation}\label{eq:thm1.6-1}
    C^{-1} \leq \frac{r \mathsf a_{\epsilon/r}}{\mathfrak c_r \mathsf a_\epsilon } \leq C, \qquad \forall 10 \epsilon^{1 - \zeta} \leq r \leq R.
    \end{equation}
    This essentially follows from Proposition~\ref{prop:compare-lqg-lfpp}. Indeed, by~\eqref{eq:prop4.2-1}, with probability tending to 1 as $\epsilon \to 0$, for all $z \in \partial B_{r/2}(0)$ and $w \in \partial B_{3r}(0)$, and with $U = B_{4R}(0)$, we have
    \[
    \mathsf a_\epsilon^{-1} D_h^\epsilon (\mbox{across } A_{r, 2r } (0)) \leq \mathsf a_\epsilon^{-1} D_h^\epsilon ( B_{ \epsilon^{1-\zeta}}(z) , B_{ \epsilon^{1-\zeta}}(w) ; B_{ \epsilon^{1-\zeta}}(U) ) 
    \leq C\, D_h(z,w ; U).
    \]
    Taking the infimum over $z,w$ on the right-hand side gives
    \[
    \mathsf a_\epsilon^{-1} D_h^\epsilon ( \mbox{across } A_{r, 2r}(0)) 
    \leq C \, D_h(\mbox{across } A_{r/2, 3r } (0)).
    \]
    By Axiom~\ref{axiom-tight}, the family $\{\mathfrak c_r^{-1} e^{-\xi h_r(0)} D_h(\mbox{across } A_{r/2, 3r } (0))\}_{0 < r <\infty}$ is tight. It follows from~\eqref{eq:lem4.3-tight} (see also Lemma~\ref{lem:tight-scale}) that the laws of $ \{ \mathsf a_\epsilon^{-1} D_h^\epsilon(\mbox{across } A_{r, 2r}(0))/(\frac{r \mathsf a_{\epsilon/r}}{\mathsf a_\epsilon} e^{\xi h_r(0)} ) \}_{0 < \epsilon < r,\  r>0}$ are tight and any subsequential limit is positive. Therefore, there exists $C>1$ (independent of $R$) such that for all sufficiently small $\epsilon$,
    \[
    \mathfrak c_r \geq C^{-1} \frac{r \mathsf a_{\epsilon/r}}{\mathsf a_\epsilon},\qquad \forall 10 \epsilon^{1 - \zeta} \leq r \leq R.
    \]
    The other inequality in~\eqref{eq:thm1.6-1} follows from a similar argument by applying~\eqref{eq:prop4.2-2} from Proposition~\ref{prop:compare-lqg-lfpp}. 
    
    \textbf{Step 2.} It follows from~\eqref{eq:thm1.6-1} that for any $r,s>0$, as $\epsilon$ is sufficiently small,
    \begin{equation*}
    \mathfrak c_r \mathfrak c_s \asymp \frac{r\mathsf a_{\epsilon/r}}{\mathsf a_\epsilon } \times \frac{s \mathsf a_{\epsilon/(sr)}}{\mathsf a_{\epsilon/r}} = \frac{rs\mathsf a_{\epsilon/(sr)}}{\mathsf a_\epsilon } \asymp \mathfrak c_{sr}, 
    \end{equation*}
    where the implicit constants in $\asymp$ depend only on the law of $D_h$. Combined with a subadditive inequality (see, e.g., Lemma 3.8 in~\cite{dg-constant}), this yields that there exists $\alpha \in \mathbb{R}$ such that $\mathfrak c_{2^{-n}} \asymp 2^{-\alpha n}$ for $n \in \mathbb{N}$. Using Axiom~\ref{axiom-tight} and $\mathfrak c_r \asymp \tfrac{\mathfrak c_1}{\mathfrak c_{1/r}}$, we further obtain $\mathfrak c_{r} \asymp r^\alpha$ for all $r>0$. By Proposition 1.1 in~\cite{dgz-exponential-metric}, we have $\mathsf a_\epsilon = \epsilon^{1 - \xi Q + o(1)}$ as $\epsilon \to 0$. Combining this with~\eqref{eq:thm1.6-1}, we deduce that $\alpha = \xi Q$. 
\end{proof}

Now we show that any two weak $\gamma$-exponential metrics are up-to-constants equivalent.

\begin{prop}\label{prop:up-to-constant-eq}
    Fix $\gamma \in (0,\sqrt{2d})$. Let $D_h$ and $\widetilde D_h$ be two weak $\gamma$-exponential metrics. Then there exists a deterministic constant $C>1$ (depending on the laws of $D_h$ and $\widetilde D_h$) such that almost surely
    \[
    C^{-1} D_h(z,w) \leq \widetilde D_h(z,w) \leq C D_h(z,w), \qquad \forall z,w \in \mathbb{R}^d.
    \]
\end{prop}

\begin{proof}
    We claim that for each $\zeta \in (0,1)$, there exists a constant $C>0$ (depending only on $\zeta$ and the laws of $D_h$ and $\widetilde D_h$) such that for any bounded connected open set $U$, with probability tending to 1 as $\epsilon \to 0$,
    \begin{align*}
    &\widetilde D_h ( B_{\epsilon^{1-\zeta}}(z) , B_{\epsilon^{1-\zeta}}(w) ; B_{\epsilon^{1-\zeta}}(U) ) 
    \leq C\, D_h(z,w ; U)  ,\quad \forall z,w\in U,\\
    &D_h ( B_{\epsilon^{1-\zeta}}(z) , B_{\epsilon^{1-\zeta}}(w) ; B_{ \epsilon^{1-\zeta}}(U) ) 
    \leq C \, \widetilde D_h(z,w ; U) ,\quad \forall z,w \in U . 
    \end{align*}
    In fact this follows verbatim from the proof of Lemma~\ref{lem:compare-subset}, with $\mathsf a_\epsilon^{-1} \widehat D_h^\epsilon$ replaced by $\widetilde D_h$. Note that the scaling constants for $D_h$ and $\widetilde D_h$ are the same.

    Let $\zeta = 1/2$. Fix $z,w \in \mathbb{R}^d$. Sending $\epsilon \to 0$ in the first inequality gives that almost surely
    \begin{equation}\label{eq:prop4.6-1}
    \widetilde D_h ( z , w) \leq C \,D_h(z,w ; U).
    \end{equation}
    It follows from Axiom~\ref{axiom-tight} that for any $p \in (0,1)$ and $A>0$, there exists $R>1$ (depending only on $p,A$, and the law of $D_h$) such that for all $r>0$,
    \[
    \mathbb{P}\left[\sup_{u,v \in B_r(0)} D_{h}(u,v) \leq \frac{1}{A} D_h(\partial B_r(0), \partial B_{Rr}(0))\right] \geq p.
    \]
    Therefore, the $D_h$-geodesic between $z$ and $w$ is almost surely contained in some bounded region. Hence letting $U$ increase to $\mathbb{R}^d$ in~\eqref{eq:prop4.6-1} yields $\widetilde D_h ( z , w) \leq C D_h(z,w)$. By considering rational pairs $(z,w)$ and using that $D_h$ and $\widetilde D_h$ are continuous metrics, we get that almost surely $\widetilde D_h ( z , w ) \leq C D_h(z,w)$ for all $z,w \in \mathbb{R}^d$. The reverse inequality is obtained by interchanging the roles of $D_h$ and $\widetilde D_h$.
\end{proof}

\subsection{Superpolynomial concentration bound for set-to-set distances}\label{subsec:superpolynomial}

In this section, we prove superpolynomial concentration bounds for the set-to-set distance (Proposition~\ref{prop:superpolynomial-cross}) and the around distance (Proposition~\ref{prop:superpolynomial-around}). The proof strategy is similar to that of \cite[Proposition 3.1]{lqg-metric-estimates}.

\begin{prop}\label{prop:superpolynomial-cross}
    Let $U \subset \mathbb{R}^d$ be connected and open. Let $K_1, K_2 $ be two disjoint connected compact subsets of $U$ that are not singletons. For each $\mathbbm{r}>0$, with superpolynomially high probability as $A \to \infty$, at a rate uniform in $\mathbbm{r}$, we have
    $$
    A^{-1} \mathfrak c_{\mathbbm{r}} e^{\xi h_{\mathbbm{r}}(0)} \leq D_h(\mathbbm{r} K_1, \mathbbm{r} K_2; \mathbbm{r} U) \leq A \mathfrak c_{\mathbbm{r}} e^{\xi h_{\mathbbm{r}}(0)}.
    $$
\end{prop}

The proof of Proposition~\ref{prop:superpolynomial-cross} is another application of the shell independence lemma (Lemma~\ref{lem:shell-independence}) and Axiom~\ref{axiom-tight}. Using them, with high probability, we can cover $\mathbbm{r} U$ by small Euclidean shells $\{A_{r/2,2r}(w)\}$ such that the $D_h$-across and $D_h$-around distances are both up-to-constants equivalent to $r^{\xi Q} e^{\xi h_r(w)}$. This allows us to control the $D_h$-distances globally in terms of $r^{\xi Q}$ and $e^{\xi h_r(w)}$. Next, we compare these quantities with $\mathfrak c_{\mathbbm r}$ and  $e^{\xi h_{\mathbbm r}(0)}$ (Lemma~\ref{lem:4.5}). By choosing the scales appropriately as functions of $A$, we obtain the proposition.

For $C>1$, $z \in \mathbb{R}^d$, and $r>0$, let $E_r(z; C)$ denote the event that
\begin{equation}\label{eq:def-sec4.2-ErC}
\sup_{u,v \in \partial B_r(z)} D_h(u,v; A_{r/2,2r}(z)) \leq C r^{\xi Q} e^{\xi h_r(z)} \quad \mbox{and} \quad D_h(\partial B_{r/2}(z), \partial B_r(z) ) \geq C^{-1} r^{\xi Q} e^{\xi h_r(z)}.
\end{equation}
By Axioms~\ref{axiom-weyl} and~\ref{axiom-local}, the event $E_r(z; C)$ is measurable with respect to $(h - h_r(z))|_{A_{r/2, 2r}(z)}$ and thus fits into the assumptions of Lemma~\ref{lem:shell-independence}. Moreover, by Axioms~\ref{axiom-translation} and~\ref{axiom-tight}, $E_r(z; C)$ occurs with probability arbitrarily close to 1 as $C \to \infty$, at a rate uniform in $z$ and $r$. Applying Lemma~\ref{lem:shell-independence} yields the following lemma (see, e.g.,~\eqref{eq:prop3.13-event} for a similar application).

\begin{lemma}\label{lem:sec4.2-lem4.8}
    For each $M > 0$, there exists $C  = C(M) > 1$ such that with probability at least $1-O(\epsilon^M)$ as $\epsilon \to 0$, at a rate uniform in $\mathbbm{r}$, the following holds. For each $w \in \tfrac{\epsilon^2}{10d} \mathbbm r \mathbb{Z}^d \cap B_{\epsilon^{-1} \mathbbm r}(0)$, there exists $r\in [\epsilon^2 \mathbbm r , \epsilon \mathbbm r] \cap \{2^{-k} \mathbbm r : k \in \mathbb{N} \}$ such that $E_r(w;C)$ occurs.  
\end{lemma}

Next, we upper-bound $|h_r(w) - h_{\mathbbm r}(0)|$ for those $r$ and $w$ satisfying the above lemma.

\begin{lemma}\label{lem:4.5}
    There exists a universal constant $\lambda>0$ (depending only on $d$) such that the following holds. For each $M > \lambda$, with probability at least $1 - O(\epsilon^{M/\lambda})$ as $\epsilon \to 0$, at a rate uniform in $\mathbbm{r}$, we have for all $w \in \tfrac{\epsilon^2}{10d} \mathbbm r \mathbb{Z}^d \cap B_{\epsilon^{-1} \mathbbm r}(0)$ and $r \in [\epsilon^2 \mathbbm r, \epsilon \mathbbm r] \cap \{2^{-k} \mathbbm{r} : k \in \mathbb{N} \} $,
    \begin{equation}\label{eq:lem4.5}
    |h_r(w) - h_{\mathbbm r}(0)| \leq \sqrt{M} \log \epsilon^{-1}.
    \end{equation}
\end{lemma}

\begin{proof} 
    By the scaling invariance of $h$ (Lemma~\ref{lem:scale-invariance}), we may assume $\mathbbm{r} = 1$. Note that $h_r(w) - h_{\mathbbm r}(0)$ is centered Gaussian with variance $O(\log \epsilon^{-1})$. Therefore, for some universal constant $C>0$ and any $M>1$,
    \[
    \mathbb{P}\left[|h_r(w) - h_{\mathbbm r}(0)| > \sqrt{M} \log \epsilon^{-1}\right] = O( \epsilon^{M/C}).
    \]
    There are at most $O(\epsilon^{-3d - 10})$ choices of $(w, r)$. Taking a union bound over all such pairs and choosing $\lambda>0$ sufficiently large then yields the desired bound of order $O(\epsilon^{M/\lambda})$.
\end{proof}

We are now ready to prove Proposition~\ref{prop:superpolynomial-cross}.

\begin{proof}[Proof of Proposition~\ref{prop:superpolynomial-cross}]
    Let $\lambda$ be as in Lemma~\ref{lem:4.5} and fix any $M>\lambda$. Let $C = C(M)>1$ be the constant from Lemma~\ref{lem:sec4.2-lem4.8}. Suppose that the events in Lemmas~\ref{lem:sec4.2-lem4.8} and~\ref{lem:4.5} both occur. We first lower-bound $D_h(\mathbbm{r} K_1, \mathbbm{r} K_2; \mathbbm{r} U)$. 
    
    For any path $P$ connecting $\mathbbm r K_1$ and $\mathbbm r K_2$ in $\mathbbm r U$, let $w$ be a point in $\tfrac{\epsilon^2}{10d} \mathbbm{r} \mathbb{Z}^d$ which is closest to this path. Let $r \in [\epsilon^2 \mathbbm r, \epsilon \mathbbm r] \cap \{2^{-k} \mathbbm{r} : k \in \mathbb{N} \}$ be such that $E_r(w; C)$ occurs. By construction, $P$ crosses $A_{r/2,r}(w)$. Therefore, by the event $E_r(w; C)$ from~\eqref{eq:def-sec4.2-ErC},
    \[
    {\rm len}(P; D_h) \geq D_h(\partial B_{r/2}(w), \partial B_r(w)) \geq C^{-1} \mathfrak c_r e^{\xi h_r(w)}.
    \]
    By Theorem~\ref{thm:sharp-c}, we may assume $\mathfrak c_t = t^{\xi Q}$ for $t >0$. Combined this with $r \geq \epsilon^2 \mathbbm r$ and~\eqref{eq:lem4.5}, we get
    \[
    {\rm len}(P; D_h) \geq C^{-1} \mathfrak \epsilon^{2 \xi Q} \mathfrak c_{\mathbbm r} e^{\xi (h_{\mathbbm{r}}(0) - \sqrt{M} \log \epsilon^{-1}) } = C^{-1}  \epsilon^{2 \xi Q + \xi \sqrt{M} } \mathfrak c_{\mathbbm r} e^{\xi h_{\mathbbm{r}}(0) }.
    \]
    Since $P$ is arbitrary, we obtain
    \begin{equation}\label{eq:prop4.7-lower}
    D_h(\mathbbm{r} K_1, \mathbbm{r} K_2; \mathbbm{r} U) \geq C^{-1}  \epsilon^{2 \xi Q + \xi \sqrt{M} } \mathfrak c_{\mathbbm r} e^{\xi h_{\mathbbm{r}}(0) }.
    \end{equation}
    Choosing $\epsilon$ such that $C^{-1}\epsilon^{2 \xi Q + \xi \sqrt{M} } = A^{-1}$, we see that the probability that the events in Lemmas~\ref{lem:sec4.2-lem4.8} and~\ref{lem:4.5} both occur is at least $1 - O_\epsilon(\epsilon^M) - O_\epsilon(\epsilon^{M/\lambda})$ as $\epsilon \to 0$, which equals
    \begin{equation}\label{eq:prop4.7-probab}
    1 -O_A\Big((C A^{-1})^{\frac{M}{2 \xi Q + \xi \sqrt{M}}}\Big) -  O_A\Big((C A^{-1})^{\frac{M/\lambda}{2 \xi Q + \xi \sqrt{M}}}\Big) \quad \mbox{as $A \to \infty$}.
    \end{equation}
    Since $M$ can be arbitrary large, this holds with superpolynomially high probability as $A \to \infty$.

    To upper-bound $D_h(\mathbbm{r} K_1, \mathbbm{r} K_2; \mathbbm r U)$, we can find a lattice path $(w_1,\ldots, w_L)$ on $\tfrac{\epsilon^2}{10d} \mathbbm{r} \mathbb{Z}^d$ connecting $\mathbbm r K_1$ and $\mathbbm r K_2$ inside $\mathbbm r U$, with $L \leq C' \epsilon^{-2d}$ for some $C' = C'(U) > 0$. For each $1 \leq i \leq L$, by Lemma~\ref{lem:sec4.2-lem4.8}, there exists $r_i \in [\epsilon^2 \mathbbm r, \epsilon \mathbbm r] \cap \{2^{-k} \mathbbm{r} :  k \in \mathbb{N}\}$ such that $E_{r_i}(w_i; C)$ occurs. As in the proof of Proposition~\ref{prop:bi-lipschitz}, we can construct a path $\widetilde P$ connecting $\mathbbm r K_1$ and $\mathbbm r K_2$ in $\cup_{1 \leq i \leq L} B_{2r_i}(w_i)$ such that
    \[
    {\rm len}(\widetilde P; D_h) \leq \sum_{i=1}^L \sup_{u,v \in \partial B_{r_i}(w_i)} D_h(u,v; A_{r_i/2,2r_i}(w_i)).
    \]
    Using the event $E_{r_i}(w_i; C)$ from~\eqref{eq:def-sec4.2-ErC}, Theorem~\ref{thm:sharp-c} and~\eqref{eq:lem4.4-1}, we get
    \[
     {\rm len}(\widetilde P; D_h) \leq \sum_{i=1}^L C r_i^{\xi Q} e^{\xi h_{r_i}(w_i)} \leq \sum_{i=1}^L C \epsilon^{2\xi Q} {\mathbbm{r}}^{\xi Q} e^{\xi (h_{\mathbbm r}(0) + \sqrt{M} \log \epsilon^{-1})}.
    \]
    Combined with $L \leq C' \epsilon^{-2d}$, we get
    \begin{equation}\label{eq:prop4.7-upper}
    D_h(\mathbbm{r} K_1, \mathbbm{r} K_2; \mathbbm{r} U) \leq C C' \epsilon^{-2d + 2 \xi Q - \xi \sqrt{M}} \mathbbm{r}^{\xi Q} e^{\xi h_{\mathbbm r}(0)}.
    \end{equation}
    Choosing $\epsilon$ such that $C C' \epsilon^{-2d + 2 \xi Q - \xi \sqrt{M}} = A$ yields the desired upper bound. As in~\eqref{eq:prop4.7-probab}, since $M$ can be taken arbitrarily large, this bound holds with superpolynomially high probability as $A \to \infty$.
\end{proof}

Recall from~\eqref{eq:def-around} the definition of $D_h$-distance around a shell. For $z \in \mathbb{R}^d$ and $r_2>r_1>0$,
\[
    D_h({\rm around} \mbox{ } A_{r_1,r_2}(z)) := \sup_{P_1, P_2} D_h(P_1, P_2; A_{r_1,r_2}(z)),
\]
where the supremum is taken over all continuous paths connecting the inner and outer boundaries of $A_{r_1,r_2}(z)$.

\begin{prop}\label{prop:superpolynomial-around}
    Let $r_2 > r_1>0$. For each $\mathbbm{r}>0$ and $z \in \mathbb{R}^d$, with superpolynomially high probability as $A \to \infty$, at a rate uniform in $\mathbbm{r}$ and $z$, we have 
    $$
    A^{-1} \mathfrak c_{\mathbbm{r}} e^{\xi h_{\mathbbm{r}}(z)} \leq D_h({\rm around} \mbox{ } A_{r_1 \mathbbm r,r_2 \mathbbm r}(z)) \leq A \mathfrak c_{\mathbbm{r}} e^{\xi h_{\mathbbm{r}}(z)}.
    $$
\end{prop}

\begin{proof}[Proof of Proposition~\ref{prop:superpolynomial-around}]
    By Axiom~\ref{axiom-translation}, we may assume $z = 0$. The lower bound follows by considering two fixed disjoint paths $P_1$ and $P_2$ crossing $A_{r_1,r_2}(0)$ and observing that~\eqref{eq:prop4.7-lower} from the proof of Proposition~\ref{prop:superpolynomial-cross} still holds for $(U, K_1,K_2) = (A_{r_1,r_2}(0), P_1,P_2)$ and the same $\mathbbm r$. The upper bound follows by noting that~\eqref{eq:prop4.7-upper} in fact holds simultaneously for $(K_1,K_2) = (P_1,P_2)$, where $P_1$ and $P_2$ are any paths crossing $A_{r_1,r_2}(0)$, with $U = A_{r_1,r_2}(0)$ and the same $\mathbbm r$.
\end{proof}

\subsection{Diameter moment bounds}\label{subsec:diameter}

In this section we prove the diameter moment bound in Theorem~\ref{thm:moments}. Let \[p_0 = \frac{2d \mathsf d_\gamma}{\gamma^2} \qquad \mbox{and} \qquad S= (0,1)^d.\]

\begin{prop}\label{compactbound}
Let $U \subset \R^d$ be open, and let $K \subset U$ be a compact connected set with more than one point. Then for all $p \in (-\infty,p_0)$ there exists a constant $c_p = c_p(U,K) < \infty$ such that for all $\mathbbm{r}>0$,
\[
\mathbb{E}\left[\left(\mathfrak c_{\mathbbm{r}}^{-1}e^{-\xi h_{\mathbbm{r}}(0)}\sup_{z,w \in \mathbbm{r}K} D_h(z,w;\mathbbm{r}U)\right)^p\right] \leq c_p.
\]
\end{prop}

We will deduce Proposition~\ref{compactbound} from the following special case.

\begin{prop}\label{squarebound}
For all $p \in (-\infty,p_0)$ there exists a constant $c_p>0$ such that for all $\mathbbm{r}>0,$
\[
\mathbb{E}\left[\left(\mathfrak c_{\mathbbm{r}}^{-1}e^{-\xi h_{\mathbbm{r}}(0)}\sup_{z,w \in \mathbbm{r}S} D_h(z,w;\mathbbm{r}S)\right)^p\right] \leq c_p.
\]
\end{prop}

We will need the following lemma on spherical average fluctuations.

\begin{lemma}\label{circavgdiff} Let $q>d$, $R>0$, and $\mathbbm{r}>0$. Then, as $\epsilon \to 0$, with probability $1-O_\e(\e^{\tfrac{q^2}{2}-d})$ uniformly in $\mathbbm r$, we have
\[
\sup\left\{|h_{\e\mathbbm{r}}(w)-h_{\mathbbm{r}}(0)|:w\in B_{R\mathbbm{r}}(0)\cap \tfrac{\e\mathbbm{r}}{4}\mathbb{Z}^d\right\} \leq q \log \e^{-1}.
\]
\end{lemma}
\begin{proof}
Let $s \in (0,q).$ Since $t \mapsto h_{e^{-t}\e\mathbbm{r}}(w)-h_{\e\mathbbm{r}}(w)$ has a Gaussian tail, there exist constants $c_1,\;c_2,\;p_1,\;p_2$ such that
\[
\mathbb{P}\left(|h_{\e \mathbbm{r}}(w)-h_{\mathbbm{r}}(0)| \leq (q-s)(\log \e^{-1})\right) \geq 1-O_\e\left(\e^{\tfrac{(q-s)^2}{2}}\right).
\]
We obtain
\[
\mathbb{P}\left(|h_{\mathbbm{r}}(w)-h_{\mathbbm{r}}(0)| \leq q\log \e^{-1}\right) \geq 1-O_\e(\e^{\tfrac{q^2}{2}}).
\]
Now taking a union bound over $O_\e(\e^{-d})$ values of $w \in B_{R\mathbbm{r}}(0)\cap \left(\frac{\e\mathbbm{r}}{4}\mathbb{Z}^d\right)$ we conclude.
\end{proof}

For the proof of Proposition \ref{squarebound} we will need the following lemma.
\begin{lemma}\label{circavgdiff2}
Let $R>0$ and $q>\sqrt{2d}$. For each $C>1$ and $\mathbbm r>0$, it holds with probability $1 - C^{-q - \sqrt{q^2-2d} + o_C(1)}$ as $C \to \infty$, at a rate uniform in $\mathbbm{r}$, that
\[
\sup\left\{|h_{2^{-n}\mathbbm{r}}(w)-h_{\mathbbm{r}}(0)|:w\in B_{R\mathbbm{r}}\cap \left(2^{-n-1}\mathbbm{r}\mathbb{Z}^d\right)\right\} \leq \log(C2^{qn}), \qquad \forall n \in \mathbb{N}.
\]
\end{lemma}
\begin{proof}
Let $E_r^n$ be the event defined by
\[
E_r^n:= \left\{\sup\left\{|h_{2^{-n}\mathbbm{r}}(w)-h_\mathbbm{r}(0)|:w \in B_{R\mathbbm{r}}(0)\cap \left(2^{-n-1}\mathbb{Z}^d\right)\right\} \leq \log(2C^{qn})\right\}.
\]
Let $\zeta>0,$ and let $\alpha_0, \ldots , \alpha_N$ be a partition of $[\zeta,\zeta^{-1}],$ $\zeta = \alpha_0 < \cdots < \alpha_N = \zeta^{-1}$ such that $\alpha_k-\alpha_{k-1} \leq \zeta$ for all $1 \leq k \leq N.$ Now using Lemma \ref{circavgdiff} with $\e=2^{-n},$ $\nu=0,$ $q+\frac{1}{\alpha_k}$ instead of $q,$ we obtain that for any $n \geq 0$ such that $2^n \in [C^{\alpha_{k-1}},C^{\alpha_k}],$ $ 1\leq k \leq N,$
\begin{align*}
\mathbb{P}\left((E_{\mathbbm{r}}^n)^c\right) & \leq \mathbb{P}\left(\sup \left\{|h_{2^{-n}\mathbbm{r}}(w)-h_{\mathbbm{r}}(0)|:w \in B_{R\mathbbm{r}}(0)\cap (2^{-n-1}\mathbbm{r}\mathbb{Z}^d)\right\}> \left(q+\frac{1}{\alpha_k}\right)\log(2^n)\right)\\
&\leq 2^{-\tfrac{n}{2}(q+\alpha_k^{-1})2-d} \leq C^{-\alpha_{k-1}\left(\tfrac{(q+\alpha_k^{-1})^2}{2}-d\right)} \leq C^{d\alpha_k-\frac{(q\alpha_k+1)^2}{2\alpha_k}+o_\zeta(1)}.
\end{align*}
Using a union bound over $n$ such that $2^n \in \left[C^{\alpha_{k-1}},C^{\alpha_k}\right]$ we obtain
\[
\mathbb{P}\left(E_{\mathbbm{r}}^n, \; \forall n \geq 0, C^{\alpha_{k-1}} \leq 2^n \leq C^{\alpha_k}\right) \geq 1-C^{d\alpha_k - \frac{(q\alpha_k+1)2}{2\alpha_k}+o_\zeta(1)}.
\]
Again using Lemma \ref{circavgdiff} with $\e=2^{-n},$ $\nu=0,$ $q+\zeta$ instead of $q,$ we obtain
\[
\mathbb{P}((E_{\mathbbm{r}}^n)^c) \leq 2^{-n(\tfrac{(q+\zeta)^2}{2}-d)}
\]
and thus adding over all $n$ we obtain
\[
\mathbb{P}(E_{\mathbbm{r}}^n, \; \forall n \in \mathbb{N} : 2^N \geq C^{\frac{1}{\zeta}}) \geq 1-C^{-\frac{(q+\zeta)^2-2d}{2\zeta} + o_C(1)}.
\]
Finally, if $n \geq 0,$ $2^n \leq C^\zeta,$ then using a union bound we obtain
\[
\mathbb{P}((E_{\mathbbm{r}}^n)^c) \leq C^{d\zeta-\frac{(q\zeta+1)^2}{2\zeta}+o_C(1)}
\]
and so
\[
\mathbb{P}(E_{\mathbbm{r}}^n, \; \forall n \in \mathbb{N}: 2^n \geq C^\zeta) \geq 1-C^{d\zeta - \frac{(q\zeta+1)^2}{2\zeta}+o_C(1)}.
\]
After maximizing the quantity
\[
d\alpha-\frac{(q\alpha+1)^2}{2\alpha}
\]
with respect to $\alpha,$ we conclude
\[
\mathbb{P}(E_{\mathbbm{r}}^n \; \forall n \geq 0) \geq 1 - C^{-q - \sqrt{q^2-2d}+o_\zeta(1) + o_C(1)}.
\]
Taking the limit as $\zeta \to 0$ we conclude.
\end{proof}

We now Proposition \ref{squarebound}.

\begin{proof}[Proof of Proposition \ref{squarebound}]
For $p<0,$ the bound follows directly from Proposition \ref{prop:superpolynomial-cross}, so we may restrict to $p>0$. 

Let $q \in (\sqrt{2d},Q).$ Then by Lemma \ref{circavgdiff2} we have with probability $1 - C^{-q - \sqrt{q^2-2d} + o_C(1)}$ that
\[
\sup\left\{|h_{2^{-n}\mathbbm{r}}(w) - h_{\mathbbm{r}}(0)|:w \in \mathbbm{r}(0,1)^d \cap (2^{-n-1}\mathbbm{r}\mathbb{Z}^d)\right\} \leq \log(C 2^{qn}) \; \forall n \geq 0.
\]
By Propositions \ref{prop:superpolynomial-around} and \ref{prop:superpolynomial-cross} and Theorem~\ref{thm:sharp-c}, we have that
\[
A^{-1} \mathbbm{r}^{\xi Q} e^{\xi h_{\mathbbm{r}}(0)} \leq D_h({\rm around} \mbox{ } B(z,r_1 \mathbbm r) \setminus \overline{B(z,r_2 \mathbbm r)}) \leq A \mathbbm{r}^{\xi Q} e^{\xi h_{\mathbbm{r}}(0)}.
\]
Since the metric $D_h$ is continuous, if $z \in \mathbbm{r}(0,1)^d$ and we let $S_n(z)$ be the the cube of sidelength $2^{-n}\mathbbm{r}$ containing $z,$ then $\mathrm{Diam}_{D_h}(S_n(z))$ tends to zero as $n\to \infty.$ Hence
\[
\sup_{w\in S_{N_C}(z)} D_h(z,w;\mathbbm{r}(0,1)^d) \leq C^\xi \mathbbm{r}^{\xi Q} e^{\xi h_{\mathbbm{r}}(0)}\sum_{n=N_C}^\infty 2^{-(Q-q-\zeta)\xi n+o_n(n)} \leq O_C(C^\xi) \mathbbm{r}^{\xi Q} e^{\xi h_{\mathbbm{r}}(0)}
\]
with superpolynomially high probability in $C.$ Since this holds for every $z \in \mathbbm{r}(0,1)^d$ we obtain that with probability at least $1-C^{-q-\sqrt{q^2-2d}+o_C(1)}$ for each $n\geq N_C:= \lfloor \log_2 C^\zeta\rfloor,$ each cube with corners in $2^{-n}\mathbbm{r}\mathbb{Z}$ has $D_h(\cdot,\cdot;\mathbbm{r}(0,1)^d)$ diameter at most $O_C(C^\xi) \mathbbm{r}^{\xi Q} e^{\xi h_{\mathbbm{r}}(0)}.$ Let $\mathcal{R}_{\mathbbm{r}}^n$ denote the set of parallelepipeds of sidelengths $2^{-n}\mathbbm{r}\times 2^{-n-1}\mathbbm{r} \times \cdots \times 2^{-n-1}\mathbbm{r}.$ Now note that we have that for any parallelepiped $R \in \mathcal{R}_{\mathbbm{r}}^n$ there is a path $P_R$ between the smaller faces of $R$ such that
\[
D_{\mathrm{across}}(P_R) \leq C^\xi 2^{-(Q-q-\zeta)\xi n+o_n(n)}\mathbbm{r}^{\xi Q} e^{\xi h_{\mathbbm{r}}(0)}
\]
where we used the fact that by definition of $D_{\mathrm{across}},$ if $R,R'$ are parallelipipeds with sidelengths $2^{-n}\mathbbm{r}\times 2^{-n-1}\mathbbm{r} \times \cdots \times 2^{-n-1}\mathbbm{r},$ and $P,P'$ are paths connecting the square sides of $R$ and $R'$ respectively, then $D_h(P_1,P_2) \lesssim D_{\mathrm{across}}(R)+D_{\mathrm{across}}(R').$ Note that any $2^{-n}\mathbbm{r}$ sidelength cube whose vertices have coordinates who are integer multiples of ${2^{-n}\mathbbm{r}}$ contain exactly $d2^{d-1}$ parallelepipeds in $\mathcal{R}_{\mathbbm{r}}^n.$

By continuity of the $D_h$ metric, note that $\mathrm{Diam}_{D_h}(S_n(z)) \to 0$ as $n\to \infty.$ Therefore
\[
\sup_{w \in S_{N_C}(z)} D_h(z,w;\mathbbm{r}(0,1)^d) \leq C^\xi \mathbbm{r}^{\xi Q}e^{\xi h_{\mathbbm{r}}(0)} \sum_{n=N_C}^\infty 2^{-(Q-q-\zeta)\xi n +o_n(n)} \leq O_C(C^\xi)\mathbbm{r}^{\xi Q}e^{\xi h_{\mathbbm{r}}(0)}.
\]
Since this holds for any $z \in \mathbbm{r}(0,1)^d,$ with probability at least $1-C^{-1-\sqrt{q^2-2d}+o_C(1)},$ for each $n\geq N_C,$ each cube of sidelength $2^{-n}\mathbbm{r}$ has $D_h(\cdot,\cdot;\mathbbm{r}(0,1)^d)$-diameter at most $O_C(C^{\xi+\zeta})\mathbbm{r}^{\xi Q}e^{\xi h_{\mathbbm{r}}(0)}.$

Since $2^{N_C} \leq C^\zeta,$ we see that
\[
\mathbb{P}\left((\mathbbm{r}^{-1}e^{-\xi h_{\mathbbm{r}}(0)})^{\xi Q} \sup_{z,w \in \mathbbm{r}(0,1)^d}D_h(z,w;\mathbbm{r}(0,1)^d)>\tilde{C}\right)\leq \tilde{C}^{-\xi^{-1}(q+\sqrt{q^2-2d})+o_{\tilde{C}}(1)}.
\]
Sending $q \to Q$ we conclude. \qedhere

\end{proof}

\begin{proof}[Proof of Proposition \ref{compactbound} using Proposition \ref{squarebound}.]
The case $p<0$ again follows from Proposition \ref{prop:superpolynomial-cross}. Let us assume that $0 \leq p < p_0.$ Then we can cover $K$ by finitely many translations of dilations of $S,$ so that $S_1, \ldots , S_n \subset U.$ Let $u_k, \rho_k$ be such that $S_k = u_k + \rho_kS$ and $\rho_k>0.$ Then by Proposition \ref{squarebound} together with translation invariance for $D_h,$ there is a constant $\tilde{C}_p$ such that
\[
\mathbb{E}\left(\left(\mathfrak c_{\mathbbm{r}\rho_k}^{-1} e^{-\xi h_{\mathbbm{r}\rho_k}(\mathbbm{r}u_k)} \sup_{z,w\in \mathbbm{r}S_k} D_h (z,w;\mathbbm{r}S_k)\right)^p\right) \lesssim \tilde{C}_p.
\]
 Now using the fact that $h_{\mathbbm{r}\rho_k}(\mathbbm{r}u_k) - h_{\mathbbm{r}}(0)$ is Gaussian with constant order variance together with Theorem \ref{thm:sharp-c}, we conclude. 
\end{proof}

\subsection{Moment bounds for point-to-set and point-to-point distances}\label{subsec:moment-other}

In this section we prove items~\ref{moment-point-point} and~\ref{moment-point-set} in Theorem \ref{thm:moments}. Recall that $h$ is normalized so that $h_1(0)=0$. Items~\ref{moment-point-point} and~\ref{moment-point-set} will follow from the estimates below. In this section we let $\mathfrak{c}_r:=r^{\xi Q}.$
\begin{prop}\label{thing1}
Let $\alpha \in \R$ and set $h^\alpha:=h-\alpha \log|\cdot|$. If $\alpha<Q$, then there exists a deterministic function $\psi:[0,\infty)\to[0,\infty)$ (depending on $\alpha$ and the law of $D_h$) such that $\psi$ is bounded near $0$ and $\psi(t)/t \to 0$ as $t\to \infty$, with the following property. For each $\mathbbm{r}>0$, with superpolynomially high probability as $C\to \infty$, at a rate uniform in $\mathbbm{r}$, we have
\begin{equation}\label{prev}
C^{-1} \frac{\mathfrak c_\mathbbm{r}}{\mathbbm{r}^{\alpha \xi}} \int_0^\infty e^{\xi h_{\mathbbm{r}e^{-t}}(0)-\xi(Q-\alpha)t-\psi(t)} dt \leq D_{h^\alpha} (0,\p B_{\mathbbm{r}}(0)) \leq C \frac{\mathfrak c_\mathbbm{r}}{\mathbbm{r}^{\alpha \xi}} \int_0^\infty e^{\xi h_{\mathbbm{r}e^{-t}}(0)-\xi(Q-\alpha)t+\psi(t)} dt
\end{equation}
and the $D_{h^\alpha}$-distance around $A_{\mathbbm{r}/e, \mathbbm{r}}(0)$ is at most the right-hand side of \eqref{prev}. 

If $\alpha>Q,$ then almost surely $D_{h^\alpha} (0,z) = \infty$ for every $z \in \R^d\setminus\{0\}.$
\end{prop}
\begin{prop}\label{thing2}
Let $\alpha,\beta \in \R$, let $z,w \in \mathbb{R}^d$ be distinct, and set $h^{\alpha, \beta}:=h-\alpha \log|\cdot -z| -\beta \log |\cdot - w|$. If $\alpha,\beta<Q$, then there exists a deterministic function $\psi:[0,\infty)\to[0,\infty)$ (depending only on $\alpha, \beta$, and the law of $D_h$) such that $\psi$ is bounded near $0$ and $\psi(t)/t \to 0$ as $t\to \infty$, with the following property. For each $\mathbbm{r}>0,$ with superpolynomially high probability as $C\to \infty$, at a rate uniform in deterministic $z,w, \mathbbm r$, we have
\begin{equation}
D_{h^{\alpha,\beta}} (z,w) \geq C^{-1} \frac{\mathfrak c_\mathbbm{r}}{\mathbbm{r}^{\alpha \xi}} \int_0^\infty \left(e^{\xi h_{\mathbbm{r}e^{-t}}(0)-\xi(Q-\alpha)t-\psi(t)}+e^{\xi h_{\mathbbm{r}e^{-t}}(0)-\xi(Q-\beta)t-\psi(t)}\right) dt
\end{equation}
and
\begin{equation}
D_{h^{\alpha,\beta}} (z,w;B_{8\mathbbm{r}}(z)) \leq C \frac{\mathfrak c_\mathbbm{r}}{\mathbbm{r}^{\alpha \xi}} \int_0^\infty \left(e^{\xi h_{\mathbbm{r}e^{-t}}(0)-\xi(Q-\alpha)t-\psi(t)}+e^{\xi h_{\mathbbm{r}e^{-t}}(0)-\xi(Q-\beta)t+\psi(t)}\right) dt
\end{equation}
If either $\alpha>Q$ or $\beta>Q,$ then $D_{h^\alpha} (z,w) = \infty$ a.s..
\end{prop}

Note that as a direct consequence of Propositions \ref{thing1} and \ref{thing2} we obtain the following.
\begin{prop}\label{thing11}
Let $\alpha \in \R$ and $h^\alpha = h-\alpha \log |\cdot|.$ If $\alpha < Q,$ then for any $p<\frac{d \mathsf d_\gamma}{\gamma}(Q-\alpha)$ there is a constant $C_p$ such that for any $\mathbbm{r}>0$
\[
\mathbb{E}\left(\left(\mathbbm{r}^{-\xi Q} \mathbbm{r}^{\alpha\xi}e^{-\xi h_{\mathbbm{r}}(0)}D_{h^\alpha}(0,\p B_{\mathbbm{r}}(0))\right)^p\right) \leq C_p.
\]
If $\alpha>Q,$ then a.s. $D_{h^\alpha}(0,z) = \infty$ for all $z \in \p B_{\mathbbm{r}}(0).$
\end{prop}

\begin{prop}\label{thing22}
Let $\alpha,\beta\in\R,$ and let $h^{\alpha,\beta}:= h -\alpha\log |\cdot - z| -\beta \log |\cdot - w|.$ Set $\mathbbm{r} = \tfrac{|z-w|}{2}.$ If $\alpha,\beta < Q,$ then for all $p < \tfrac{d \mathsf d_\gamma}{\gamma} (Q-\max\{\alpha,\beta\}),$ there exists a constant $C_p>0$ such that for every choice of $z,w$ as above, we have
\[
\mathbb{E}\left(\left(\mathbbm{r}^{-\xi Q}\mathbbm{r}^{\alpha \xi} e^{-\xi h_{\mathbbm{r}}(z)}D_{h^\alpha}(z,w;B_{8\mathbbm{r}}(z))\right)^p\right) \leq C_p.
\]
If $\max\{\alpha,\beta\}>Q$ then a.s. $D_{h^{\alpha,\beta}}(z,w)=\infty.$
\end{prop}
\begin{proof}[Proof of Proposition \ref{thing11}]
By Proposition \ref{thing1}, with superpolynomially high probability as $C\to\infty$ we have that
\[
\mathbbm{r}^{-\xi Q} \mathbbm{r}^{\alpha\xi}e^{-\xi h_{\mathbbm{r}}(0)}D_{h^\alpha}(0,\p B_{\mathbbm{r}}(0)) \lesssim \int_0^\infty e^{\xi (h_{\mathbbm{r}e^{-t}}(0)-h_{\mathbbm{r}}(0))-\xi(Q-\alpha)t+\psi(t)} dt.
\]
Using Proposition \ref{expbound} to bound the LHS we conclude.
\end{proof}

\begin{proof}[Proof of Proposition \ref{thing22}]
In this case we obtain
\[
\mathbbm{r}^{-\xi Q}\mathbbm{r}^{\alpha \xi} e^{-\xi h_{\mathbbm{r}}(z)}D_{h^\alpha}(z,w;B_{8\mathbbm{r}}(z)) \lesssim \int_0^\infty \left(e^{\xi (h_{\mathbbm{r}e^{-t}}(0)-h_{\mathbbm{r}}(0))-\xi(Q-\alpha)t-\psi(t)}+e^{\xi (h_{\mathbbm{r}e^{-t}}(0)-h_{\mathbbm{r}}(0))-\xi(Q-\beta)t+\psi(t)}\right) dt.
\]
Splitting the integral and using Proposition \ref{expbound} we conclude.
\end{proof}

Items~\ref{moment-point-set} and~\ref{moment-point-point} in Theorem \ref{thm:moments} follow from Propositions \ref{thing1} and \ref{thing2} in the same way that Theorems 1.10 and 1.11 in \cite{lqg-metric-estimates} follow from Propositions 3.14 and 3.15 therein.

We now focus on the proofs of Propositions~\ref{thing1} and~\ref{thing2}.

\begin{proof}[Proof of Proposition \ref{thing1}]
Let $\zeta \in (0,1).$ By Propositions \ref{prop:superpolynomial-around}, \ref{prop:superpolynomial-cross} and Axiom~\ref{axiom-weyl}, together with a union bound over all $[0,C^{\frac{1}{\zeta}}]\cap \mathbb{Z},$ we see that with polynomially high probability  as $C\to 0,$
\begin{equation}\label{cond1}
D_{h^\alpha}(\p B_{\mathbbm{r}e^{-k-1}}(0),\p B_{\mathbbm{r}e^{-k}}(0)) \geq C^{-1} \mathfrak c_{\mathbbm{r} e^{-k}} \mathbbm{r}^{-\xi \alpha} \exp(\xi h_{\mathbbm{r}e^{-k}}(0) + \xi \alpha k)
\end{equation}
\begin{equation}\label{cond2}
\mbox{There is a path from }\p B_{\mathbbm{r}e^{-k-2}(0)} \mbox{ to } \p B_{\mathbbm{r}e^{-k}(0)} \mbox{ with } D_{h^\alpha} \mbox{-length at most }C\mathfrak c_{\mathbbm{r}e^{-k}r^{-\xi \alpha}}\exp(\xi h_{\mathbbm{r}e^{-k}}(0)+\xi \alpha k)
\end{equation}
\begin{equation}\label{cond3}
\mbox{The $D_{h^\alpha}$ distance around }B_{\mathbbm{r}e^{-k}}(0) \setminus \bar{B}_{\mathbbm{r}e^{-k-1}}(0) \mbox{ has $D_{h^\alpha}$-length at most }C \mathfrak c_{\mathbbm{r}e^{-k}}\mathbbm{r}^{-\xi\alpha}\exp(\xi h_{\mathbbm{r}e^{-k}}(0)+\xi \alpha k).
\end{equation}
Moreover, using Propositions \ref{prop:superpolynomial-around} and \ref{prop:superpolynomial-cross} with $\xi^\zeta$ and using another union bound over all $k,$ we see that with superpolynomially high probability as $C\to \infty,$ \eqref{cond1}, \eqref{cond2}, \eqref{cond3} hold along with
\begin{equation}\label{cond4}
\mbox{The $D_{h^\alpha}$-distance around } B_{\mathbbm{r}e^{-k}(0)}\setminus  B_{\mathbbm{r}e^{-k-2}(0)} \mbox{ is at most } k^\zeta \mathfrak c_{\mathbbm{r}}r^{}-\xi \alpha\exp(\xi h_{\mathbbm{r}e^{-k}}(0)+\xi \alpha k) \mbox{ for }k\leq C^{\frac{1}{\zeta}}
\end{equation}
and
\begin{equation}\label{cond5}
\mbox{The $D_{h^\alpha}$-distance around } B_{\mathbbm{r}e^{-k}(0)}\setminus  B_{\mathbbm{r}e^{-k-2}(0)} \mbox{ is at most } k^\zeta \mathfrak c_{\mathbbm{r}}r^{}-\xi \alpha\exp(\xi h_{\mathbbm{r}e^{-k}}(0)+\xi \alpha k) \mbox{ for }k\geq C^{\frac{1}{\zeta}}.
\end{equation}
We will assume from now on that \eqref{cond1}, \eqref{cond2}, \eqref{cond3} hold for $k \leq C^{\frac{1}{\zeta}},$ and \eqref{cond4} and \eqref{cond5} hold for $k\geq C^{\frac{1}{\zeta}},$ since their intersection is an event with superpolynomially high probability.

These conditions imply that with superpolynomially high probability, we have
\begin{eqnarray}\label{rightside}
C^{-1} \frac{\mathfrak c_\mathbbm{r}}{\mathbbm{r}^{\alpha\xi}} \sum_{k=0}^{\lfloor C^{\frac{1}{\zeta}}\rfloor} e^{\xi h_{\mathbbm{r}e^{-k}}(0)-\xi(Q-\alpha)k-\phi(k)}\leq D_{h^\alpha}(0,\p B_\mathbbm{r}(0))\nonumber\\
\leq C \frac{\mathfrak c_\mathbbm{r}}{\mathbbm{r}^{\alpha\xi}} \sum_{k=0}^{\lfloor C^{\frac{1}{\zeta}}\rfloor} e^{\xi h_{\mathbbm{r}e^{-k}}(0)-\xi(Q-\alpha)k-\phi(k)} + \frac{\mathfrak c_\mathbbm{r}}{\mathbbm{r}^{\alpha\xi}} \sum_{k=\lfloor C^{\frac{1}{\zeta}}\rfloor+1}^\infty k^\zeta e^{\xi h_{\mathbbm{r}e^{-k}}(0)-\xi(Q-\alpha)k-\phi(k)}.
\end{eqnarray}
Moreover by \eqref{cond2} and \eqref{cond3} we have that the $D_{h^\alpha}$-distance around $B_{\mathbbm{r}}(0)\setminus B_{\frac{\mathbbm{r}}{e}}(0)$ is at most the right hand side of \eqref{rightside}.

Now using the superpolynomial tail estimates on $h_{\mathbbm{r}e^{-t}}(0)-h_{\mathbbm{r}}(0)$ together with a union bound we see that with superpolynomially high probability as $C\to \infty,$ we have
\begin{equation}\label{cond6}
\sup_{t \in [k,k+1]}|h_{\mathbbm{r}e^{-t}}(0)-h_{\mathbbm{r}e^{-k}}(0)| \leq \frac{1}{\xi} \log C
\end{equation}
for all integers $k \leq C^{\frac{1}{\zeta}}.$ Letting $\psi(t)=\phi(\lfloor t\rfloor),$ we see that if \eqref{cond6} holds, then for all integers $k \leq C^{\frac{1}{\zeta}}$
\[
e^{\xi h_{\mathbbm{r}e^{-k}}(0)-\xi(Q-\alpha)k-\phi(k)} \geq C^{-1} \int_k^{k+1} e^{\xi h_{\mathbbm{r}e^{-t}}(0)-\xi(Q-\alpha)t -\psi(t)} dt,
\]
\[
e^{\xi h_{\mathbbm{r}e^{-k}}(0)-\xi(Q-\alpha)k+\phi(k)} \leq C \int_k^{k+1} e^{\xi h_{\mathbbm{r}e^{-t}}(0)-\xi(Q-\alpha)t +\psi(t)} dt,
\]
and hence after summing over integers $k\leq C^{\frac{1}{\zeta}}$ we obtain
\begin{equation}\label{eq1}
\sum_{k=0}^{\lfloor C^{\frac{1}{\zeta}}\rfloor}e^{\xi h_{\mathbbm{r}e^{-k}}(0)-(Q-\alpha)k-\phi(k)} \geq C^{-1} \int_0^{\lfloor C^{\frac{1}{\zeta}}\rfloor+1} e^{\xi h_{\mathbbm{r}e^{-t}}(0)-\xi(Q-\alpha)t-\psi(t)} dt,
\end{equation}
\begin{equation}\label{eq2}
\sum_{k=0}^{\lfloor C^{\frac{1}{\zeta}}\rfloor}e^{\xi h_{\mathbbm{r}e^{-k}}(0)-(Q-\alpha)k+\phi(k)} \leq C \int_0^{\lfloor C^{\frac{1}{\zeta}}\rfloor+1} e^{\xi h_{\mathbbm{r}e^{-t}}(0)-\xi(Q-\alpha)t+\psi(t)} dt,
\end{equation}

Since $t \mapsto h_{\mathbbm{r}e^{-t}(0)}-h_{\mathbbm{r}}(0)$ has a superpolynomial tail, if $q$ is chosen small enough such that $\xi q (Q-\alpha)\geq \frac{\xi^2 q^2}{2},$ then
\begin{eqnarray*}
\mathbb{E}\left(\left(\int_{\lfloor C^{\frac{1}{\zeta}}\rfloor}^\infty e^{\xi h_{\mathbbm{r}e^{-t}}(0)-\xi(Q-\alpha)t +\psi(t)} dt\right)^q\right)& \lesssim& e^{q h_{\mathbbm{r}}(0)} \int_{\lfloor C^{\frac{1}{\zeta}}\rfloor}^\infty \exp\left(-\left(\xi q(Q-\alpha -\frac{\xi^2 q^2}{2})\right)t +o_t(t)\right) dt\\
&\lesssim& e^{q h_{\mathbbm{r}}(0)} \exp\left(-\frac{1}{2}\left(\xi q (Q-\alpha) -\frac{\xi^2 q^2}{2}\right) C^{\frac{1}{\zeta}}\right).
\end{eqnarray*}
By Chebyshev we obtain that
\[
\mathbb{P}\left(\int_{\lfloor C^{\frac{1}{\zeta}} \rfloor}^\infty e^{\xi h_{\mathbbm{r}e^{-t}}(0)-\xi(Q-\alpha)t + \psi(t)}dt>e^{\xi h_{\mathbbm{r}}(0)-C^{\frac{1}{2\zeta}}}\right) 
\]
decays faster than any (negative) power of $C.$ On the other hand, by a Gaussian tail bound, we have that
\[
\mathbb{P}\left(\int_0^{\lfloor C^{\frac{1}{\zeta}} \rfloor} e^{\xi h_{\mathbbm{r}e^{-t}}(0)-\xi(Q-\alpha)t + \psi(t)}dt<e^{\xi h_{\mathbbm{r}}(0)-C^{\frac{1}{2\zeta}}}\right) 
\]
Decays faster than any negative power of $C,$ thus with superpolynomially high probability,
\[
\int_0^\infty e^{\xi h_{\mathbbm{r}e^{-t}}(0)-\xi(Q-\alpha)t + \psi(t)}dt>e^{\xi h_{\mathbbm{r}}(0)-C^{\frac{1}{2\zeta}}}  \leq 2 \int_0^{\lfloor C^{\frac{1}{\zeta}} \rfloor} e^{\xi h_{\mathbbm{r}e^{-t}}(0)-\xi(Q-\alpha)t + \psi(t)}dt
\]
Similarly, with superpolynomially high probability as $C\to\infty$ we have that
\[
\sum_{k=\lfloor C^{\frac{1}{\zeta}} \rfloor+1}^\infty k^\zeta e^{\xi h_{\mathbbm{r}e^{-k}}(0)-\xi(Q-\alpha)k-\phi(k)} \leq \int_0^\infty e^{\xi h_{\mathbbm{r}e^{-t}}(0)-\xi(Q-\alpha)t-\phi(t)} dt.
\]
Combining these last two statements with \eqref{eq1} and \eqref{eq2} we obtain that with superpolynomially high probability as $C\to \infty,$ \eqref{prev} and the statement right after hold $2C^2$ instead of $C.$ Since this is true with superpolynomially high probability as $C\to \infty,$ this completes the $\alpha<Q$ case.

If $\alpha>Q,$ a.s. $e^{\xi h_{\mathbbm{r}e^{-k}}(0)-\xi(Q-\alpha)k-\phi(k)} \geq e^{\beta k}$ for any $0< \beta< Q-\alpha$ for large enough $k.$ By Borel-Cantelli, this implies that $D_{h^\alpha} (0,\p B_{\mathbbm{r}}(0)) \geq C^{-1} e^{\beta \lfloor C^{\frac{1}{\zeta}}\rfloor} \to \infty$ as $C\to \infty.$ Since this holds for all rational $\mathbbm{r},$ this completes the proof.
\end{proof}

\begin{proof}[Proof of Proposition \ref{thing2}]
Note that by the LQG coordinate change, the conclusions of Proposition \ref{thing1} still hold if we instead $z$ in place of $0$ with a uniform in $z$ and $\mathbbm{r}$ rate. Since any path between $z$ and $w$ must contain a path from $z$ to $\p B_{|z-w|/2}(z)$ and a path from $w$ to $\p B_{|z-w|/2}(w),$ by the lower bound in Lemma \ref{thing1} we obtain the first assertion in Proposition \ref{thing2}. We also obtain the assertion in the case $\alpha>Q$ or $\beta>Q.$

Now, suppose $\alpha<Q.$ Using Proposition \ref{thing1} for $8\mathbbm{r}$ instead of $\mathbbm{r}$ we obtain that there is path from $z$ to $\p B_{8 \mathbbm{r}}(z)$ of length at most $\int_{-\log 8}^\infty e^{\xi h_{\mathbbm{r}e^{-t}}(z)-\xi(Q-\alpha)t + \psi(t)}$ and the distance around $B_{8\mathbbm{r}}(z)\setminus B_{8\mathbbm{r}/e}(z)$ is also at most $\int_{-\log 8}^\infty e^{\xi h_{\mathbbm{r}e^{-t}}(z)-\xi(Q-\alpha)t + \psi(t)},$ with identical statements for $w.$ Since $w \in B_{8\mathbbm{r}}(z),$ this implies the Proposition's statement with 0 replaced by $-\log 8.$ To obtain the full result, we first use the fact that $t\mapsto h_{\mathbbm{r}e^{-t}}(z)-h_{\mathbbm{r}}(z)$ has an exponential tail to see that
\[
\sup_{t \in [-\log 8,0]} h_{\mathbbm{r}e^{-t}}(z)\leq \inf_{[0,\log 2]} h_{\mathbbm{r}e^{-t}}(z)+\log C
\]
with superpolynomial probability as $C\to\infty.$ Hence with superpolynomial probability as $C\to\infty$ we have
\begin{eqnarray*}
\int_{-\log 8}^\infty e^{\xi h_{\mathbbm{r}e^{-t}}(z)-\xi(Q-\alpha)t + \psi(t)} \leq \int_0^\infty e^{\xi h_{\mathbbm{r}e^{-t}}(z)-\xi(Q-\alpha)t + \psi(t)} + C^\zeta \int_0^{\log 2} e^{\xi h_{\mathbbm{r}e^{-t}}(z)-\xi(Q-\alpha)t + \psi(t)},
\end{eqnarray*}
with an identical estimate for $w.$ After combining with the statement of Proposition \ref{thing2} with $0$ replaced by $-\log 8$ we conclude.
\end{proof}

Finally we prove Theorem~\ref{thm:moments}.

\begin{proof}[Proof of Theorem~\ref{thm:moments}]
    Item~\ref{moment-diam} follows from Proposition~\ref{compactbound}. Item~\ref{moment-set-set} follows from Proposition~\ref{prop:superpolynomial-cross}. Items~\ref{moment-point-set} and~\ref{moment-point-point} are proved in Section~\ref{subsec:moment-other}.
\end{proof}

\subsection{H\"older Continuity}\label{subsec:holder}

In this section, we prove Theorem~\ref{thm:holder-continuous}. We will need the following important lemma.

\begin{prop}\label{holderprop}
Let $K \subset \mathbb{R}^d$ be a compact set and let $\chi \in (0, \xi (Q-\sqrt{2d}))$ and $\chi' > \xi (Q+\sqrt{2d}).$ For each $\mathbbm{r}>0,$ we have with polynomially probability that as $\e \to 0$ at a rate uniform in $\mathbbm{r}$ that
\[
\left|\frac{u-v}{\mathbbm{r}}\right|^{\chi'} \leq \mathfrak c_{\mathbbm{r}}^{-1} e^{-\xi h_{\mathbbm{r}}(0)} D_h(u,v)\leq \left|\frac{u-v}{\mathbbm{r}}\right|^{\chi}
\]
for all $u,v \in \mathbbm{r}K$ with $|u-v|\leq \e\mathbbm{r}.$
\end{prop}

For this, we will need the following lemma.
\begin{lemma}\label{lemmathing}
For all $s \in (0,\xi Q),$ each $\mathbbm{r}>0$ and each $z \in \mathbbm{r}K,$
\[
\mathbb{P}\left(\sup_{u,v \in B_{\e\mathbbm{r}}(z)}D_h(u,v;B_{2\e\mathbbm{r}}(z))\leq \e^s \mathfrak c_{\mathbbm{r}}e^{\xi h_{\mathbbm{r}}(0)}\right) \geq 1 - \e^{\frac{(\xi Q-s)^2}{2\xi^2} +o_\e(1)}
\]
as $\e \to 0.$
\end{lemma}
\begin{proof}
Note that $h_{2 \e \mathbbm{r}}(z)-h_{\mathbbm{r}}(z)$ is a centered Gaussian with variance $\log \e^{-1} - \log 2$ and is independent from $\left(h-h_{2\e\mathbbm{r}}(z)\right)|_{B_{2\e\mathbbm{r}}(z)}.$ Then by axioms II, III, $h_{2\e\mathbbm{r}}(z) - h_{\mathbbm{r}}(z)$ is independent from
\[
D_{h-h_{2\e\mathbbm{r}}(z)(u,v; B_{2\mathbbm{r}})(z)} = e^{-\xi h_{2\e\mathbbm{r}}(z)} D_h(u,v;B_{2\e\mathbbm{r}}(z)).
\]
Now using Theorem~\ref{thm:sharp-c} and Proposition \ref{compactbound}, we have
\begin{eqnarray*}
&&\mathbb{E}\left(\left(\mathfrak c_\mathbbm{r}^{-1}e^{-\xi h_{\mathbbm{r}}(0)}\sup_{u,v \in B_{\e\mathbbm{r}}} D_h(u,v;B_{2\e\mathbbm{r}}(z))\right)^p\right)\\ && = \left(\frac{c_{\e\mathbbm{r}}}{\mathfrak c_{\mathbbm{r}}}\right)^p \mathbb{E}\left(e^{\xi_p (h_{\e\mathbbm{r}}(z)-h_{\mathbbm{r}}(z))}\right) \mathbb{E}\left(\left(\mathfrak c_\mathbbm{r}^{-1}e^{-\xi h_{\mathbbm{r}}(0)}\sup_{u,v \in B_{\e\mathbbm{r}}} D_h(u,v;B_{2\e\mathbbm{r}}(z))\right)^p\right)\\
&&\leq \e^{\xi Q_p - \frac{\xi^2p^2}{2}+o_\e(1)}.
\end{eqnarray*}
Combining this with Chebyshev we see that
\[
\mathbb{P}\left(\sup_{u,v \in B_{\e\mathbbm{r}}} D_h(u,v;B_{2\e\mathbbm{r}}(z)) > \e^s \mathfrak c_{\mathbbm{r}} e^{\xi h_{\mathbbm{r}(z)}}\right) \leq \e^{\e^{p \xi Q}-\frac{p^2\xi^2}{2}-ps +o_\e(1)}.
\]
Taking $p = \frac{(\xi Q-s)}{\xi^2}$ we conclude.
\end{proof}
\begin{lemma}
For each $\chi \in (0, \xi (Q-\sqrt{2d}))$ and each $\mathbbm{r}>0,$ it holds with polynomially high probability as $\e \to 0$ at a rate which is uniform in $\mathbbm{r}$ that
\[
\mathfrak c_{\mathbbm{r}}^{-1} e^{-\xi h_{\mathbbm{r}}(0)} D_h(u,v;B_{2|u-v|}(u))\leq \left|\frac{u-v}{\mathbbm{r}}\right|^\chi
\]
for all $u,v \in \mathbbm{r}K$ with $|u-v| \leq \e\mathbbm{r}.$ Moreover, it also holds with polynomially high probability as $\e\to 0$ at a rate uniform in $\mathbbm{r}$ that for each $K \in \mathbb{N}_0$ and each cube of side length $2^{-K}\e\mathbbm{r},$ with corners in $2^{-K}\e \mathbbm{r} \mathbb{Z}^d$ which intersects $\mathbbm{r}K,$ we have
\[
\mathfrak c_{\mathbbm{r}}^{-1} e^{-\xi h_{\mathbbm{r}}(0)}\sup_{u,v \in S} D_h(u,v ; S) \leq (2^{-k}\e)^{\chi}.
\]
\end{lemma}
\begin{proof}
The bound \eqref{eq1} follows from Lemma \ref{lemmathing} applied with $s=\chi$ and with $2^{-K}\e$ for $K \in \mathbb{N}_0$ instead of $\e,$ and a union bound over all $z \in B_{\e\mathbbm{r}} \cap (2^{-K-2}\e\mathbbm{r}\mathbb{Z}^d)$ and then over all $K \in \mathbb{N}_0.$
\end{proof}
\begin{lemma}
Let $s>\xi Q,$ $\mathbbm{r}>0$ and $z \in \mathbbm{r}K.$ Then
\[
\mathbb{P}(D_h(B_{\e\mathbbm{r}}(z), \p B_{2\e\mathbbm{r}}(z))) \geq 1-\e^{\frac{(s-\xi Q)^2}{2\xi^2}+o_\e(1)}
\]
as $\e \to 0$ uniformly over $\mathbbm{r}$ and $z \in \mathbbm{r}K.$
\end{lemma}
\begin{proof}
The proof is similar to that of Lemma \ref{lemmathing}. By Proposition \ref{prop:superpolynomial-around}, $c_{\e\mathbbm{r}}^{-1} e^{-\xi h_{\e\mathbbm{r}}(z)} D_h(B_{\e\mathbbm{r}}(z) , \p B_{2\e\mathbbm{r}}(z))$ has finite moments of all negative orders which are bounded above uniformly over all $z \in \mathbb{C}$ and $\mathbbm{r}>0.$ By the same computations as in Lemma \ref{lemmathing}, we have
\[
\mathbb{E}\left(\mathfrak c_{\mathbbm{r}}^{-1}e^{-\xi h_{\mathbbm{r}}(z)} D_h(B_{\e\mathbbm{r}}(z) , \p B_{2\e\mathbbm{r}}(z))^{-p}\right) = \e^{-\xi Qp - \frac{\xi^2 p^2}{2}+o_\e(1)}
\]
Uniformly over all $z \in \mathbb{C}$ and $\mathbbm{r}>0.$ The rest of the proof proceeds similarly using Chebyshez’s inequality.
\end{proof}
\begin{lemma}\label{other}
For each $\chi’>\xi (Q+\sqrt{2d})$ and each $\mathbbm{r}>0,$ it holds with polynomially high probability as $\e\to 0$ at a rate which is uniform in $\mathbbm{r}>0$ that
\[
\mathfrak c_{\mathbbm{r}}^{-1} e^{-\xi h_{\mathbbm{r}}(0)} D_h(u,v) \geq \left|\frac{u-v}{\mathbbm{r}}\right|^{\chi’}
\]
for all $|u-v|\leq \e.$
\end{lemma}
\begin{proof}
This follows from Lemma \ref{lemmathing} applied with $s = \chi’$ and with $2^{-K}\e$ for $K \in \mathbb{N}_0$ instead of $\e$ together with a union bound over $z \in B_{\e\mathbbm{r}}(K)\cap \left(2^{-K-2}\e\mathbbm{r}\mathbb{Z}^d\right)$ and then over all $K \in \mathbb{N}_0$ and then over all $K \in \mathbb{N}_0.$
\end{proof}
Proposition \ref{holderprop} now follows from Lemma \ref{lemmathing} and Lemma \ref{other}.

Finally the last step is to show that the H\"older exponent is optimal.

\begin{lemma}\label{optimal}
Let $V \in \R^d$ be an open set. Then almost surely, the identity map on $V$ equipped with the Euclidean metric to $(V,D_h|_V)$ is not h\"older for any exponent greater than $\xi (Q-\sqrt{2d}).$ Moreover the inverse of this map is not H\"older continuous for any exponent greater than $\xi^{-1} (Q+\sqrt{2d})^{-1}.$ 
\end{lemma}
\begin{proof}
By Axiom \ref{axiom-weyl}, we can assume without loss of generality that $h_1(0)=0.$ We can additionally assume that $V$ is bounded with a smooth boundary.

Let $h^V$ be the zero boundary part of $h$ chosen so that $h-h^V$ is a random function such that $\Delta^{\frac{d}{2}}(h-h^V) = 0.$

Let $\alpha \in (-\sqrt{2d},\sqrt{2d}).$  Let $\mathbbm{z}$ be sampled uniformly from $\mu_h^\alpha$ normalized to be a probability measure. Let $\tilde{\mathbb{P}}$ be the law of $(h,\mathbbm{z})$ weighted by the total mass $\mu_{h^V}^\alpha(V).$ Then a sample of the $\tilde{\mathbbm{P}}$ measure can be obtained by first sampling $\tilde{h}$ from the unweighted marginal law of $h,$ then independently sampling $\mathbbm{z}$ uniformly with respect to the Lebesgue measure and setting $h = \tilde{h}-\alpha \log |\cdot - \mathbbm{z}|+g_\mathbbm{z}$ where $g_\mathbbm{z}$ is a deterministic continuous function.

By Proposition \ref{thing1}, together with the fact that $g_\mathbbm{z}$ is bounded in a neighborhood of $\mathbbm{z}$ and Borel-Cantelli we have that
\[
D_h(\mathbbm{z},\p B_r(\mathbbm{z})) =r^{o_r(1)} \frac{c_r}{r^{\alpha\xi}} \int_0^\infty e^{\xi \tilde{h}_{re^{-t}}(\mathbbm{z})-\xi(Q-\alpha) + o_t(t)} dt
\]
where the $o_t(t)$ is deterministic and tends to $0$ as $t \to 0.$ Using Proposition \ref{expbound}, we see that
\[
\int_0^\infty e^{\xi \tilde{h}_{re^{-t}}(\mathbbm{z})-\xi(Q-\alpha) t + o_t(t)} dt = r^{o_r(1)} e^{\xi \tilde{h}_r(\mathbbm{z})} = r^{o_r(1)}.
\]
Since $\alpha$ can be chosen to be arbitrarily close to $ \sqrt{2d}$ this implies that the identity map from the Euclidean metric to the LQG metric is not H\"older continuous, and neither is its inverse. This completes the proof.
\end{proof}

We now prove Theorem \ref{thm:holder-continuous}.

\begin{proof}[Proof of Theorem \ref{thm:holder-continuous}.]
The proof follows by combining Proposition \ref{holderprop} and Lemma \ref{optimal}.
\end{proof}

\subsection{KPZ relation for the weak exponential metric: proof of Theorem~\ref{thm:kpz}}\label{sec:KPZ}

In this section, we prove the KPZ relation for the weak exponential metric (Theorem~\ref{thm:kpz}). The proof is similar to Sections 2 and 3 of~\cite{GP-kpz}. 

Recall that $h$ is a whole-space LGF with additive constant chosen so that $h_1(0) = 0$. For $\alpha \in [-\sqrt{2d}, \sqrt{2d}]$, define the thick-point set
\begin{equation}\label{eq:thick-point}
\mathcal{T}^\alpha_h := \Big{\{} z \in \mathbb{R}^d: \lim_{\epsilon \to 0} \frac{h_\epsilon(x)}{\log \epsilon^{-1}} = \alpha \Big{\}}.
\end{equation}
By \cite[Theorem 4.1]{GHM-KPZ}\footnote{Theorem 4.1 in~\cite{GHM-KPZ} is stated for the two-dimensional GFF, but the same proof extends to the LGF in $d \geq 3$.}, we know that if $X$ is a deterministic Borel set or a random Borel set independent of $h$, then the Euclidean Hausdorff dimension of $X \cap \mathcal{T}^\alpha_h$ is given by
\begin{equation}\label{eq:dim-thick}
{\rm dim}^0_{\mathcal H}(X \cap \mathcal{T}^\alpha_h) = \max \Big{\{} {\rm dim}^0_{\mathcal H}(X) - \frac{\alpha^2}{2}, 0 \Big{\}}.
\end{equation}
Moreover, when ${\rm dim}^0_{\mathcal H}(X) < \frac{\alpha^2}{2}$, we have $X \cap \mathcal{T}^\alpha_h = \emptyset$ a.s.

To prove Theorem~\ref{thm:kpz}, we first prove the following proposition.

\begin{prop}\label{prop:kpz-1}    
Let $\gamma \in (0,\sqrt{2d})$ and let $D_h$ be a weak $\gamma$-exponential metric. If $X$ is a deterministic Borel set or a random Borel set independent of $h$ and $\alpha \in [-\sqrt{2d}, \sqrt{2d}]$, then almost surely
    $$
    {\rm dim}^\gamma_{\mathcal H}(X \cap \mathcal{T}^\alpha_h) =  \frac{1}{\xi(Q - \alpha)} {\rm dim}^0_{\mathcal H}(X \cap \mathcal{T}^\alpha_h). 
    $$
\end{prop}

We have the following corollary of Theorem~\ref{thm:kpz} and Proposition~\ref{prop:kpz-1}. 

\begin{cor}\label{cor:hausdorff-2}
    Let $X$ be a deterministic Borel set or a random Borel set independent of $h$, and let $\alpha = Q - \sqrt{Q^2 - 2 \dim_{\mathcal H}^0 X}$. Then ${\rm dim}^\gamma_{\mathcal H}(X) = {\rm dim}^\gamma_{\mathcal H}(X \cap \mathcal{T}^\alpha_h)$. Moreover, we have ${\rm dim}^\gamma_{\mathcal H}(\mathbb{R}^d) = {\rm dim}^\gamma_{\mathcal H}(\mathcal{T}^\gamma_h) = \mathsf d_\gamma$. 
\end{cor}

\begin{proof}[Proof of Corollary~\ref{cor:hausdorff-2} assuming Theorem~\ref{thm:kpz} and Proposition~\ref{prop:kpz-1}]
    Taking $\alpha =  Q - \sqrt{Q^2 - 2 \dim_{\mathcal H}^0 X}$ in Proposition~\ref{prop:kpz-1} and using~\eqref{eq:dim-thick} gives 
    \[
    {\rm dim}^\gamma_{\mathcal H}(X \cap \mathcal{T}^\alpha_h) = \frac{1}{\xi}\bigl(Q - \sqrt{Q^2 - 2 {\rm dim}_{\mathcal H}^0 X}\bigr),
    \]
    which equals ${\rm dim}^\gamma_{\mathcal H}(X)$ by Theorem~\ref{thm:kpz}.
\end{proof}

We will need the following estimate for $D_h$.

\begin{lemma}\label{lem:sec4.6-Dh}
    Let $R \geq 1$, $\alpha \in [-\sqrt{2d},\sqrt{2d}]$, $\delta >0$, and $p \in (0,\frac{2 d \mathsf d_\gamma}{\gamma^2})$. Then, for all $z \in B_R(0)$,
    \begin{align*}
    \mathbb{E}\Big{[}\mathbbm{1}_{\{|h_r(z) - \alpha \log r^{-1}| \leq \delta \log r^{-1}\}}\Bigl(\sup_{u,v \in B_r(z)} D_h( u, v)\Bigr)^p\Big{]} &\leq r^{\xi p (Q - \alpha - \delta)+ \frac{1}{2}(|\alpha| - \delta)_+^2 + o_r(1)}, \\
    \mathbb{E}\Big{[}\Bigl(\sup_{u,v \in B_r(z)} D_h( u, v)\Bigr)^p\Big{]} &\leq r^{\xi Qp - \frac{1}{2}\xi^2p^2 + o_r(1)} \qquad \mbox{as $r \to 0$},
    \end{align*}
    where the $o_r(1)$ terms are uniform in $z$.
\end{lemma}
\begin{proof}
    We assume $r<1$. The constants $C$ in this proof may change from line to line but are independent of $z$ and $r$. By Lemma~\ref{lem:markov-whole}, we have \begin{equation}\label{eq:lem4.27-decompose}
        h|_{B_{2r}(z)} = \mathfrak{h}+ \mathring{h},
    \end{equation}where $\mathfrak{h}$ is a random $s$-harmonic function on $B_{2r}(z)$ determined by $h|_{\mathbb{R}^d \setminus B_{2r}(z)}$, and $\mathring{h}$ is a LGF on $B_{2r}(z)$ (minus its average over $\partial B_1(0)$), which is independent of $h|_{\mathbb{R}^d \setminus B_{2r}(z)}$. 
    
    Define
    \[
    O := \sup_{u,v \in B_{3r/2}(z)}|\mathfrak{h}(u) - \mathfrak{h}(v)|.
    \]
    Since for $u,v \in B_{2r}(z)$, $\mathfrak{h}(u) - \mathfrak{h}(v)$ is a smooth centered Gaussian process in $(u,v)$, the Borel-TIS inequality~\cite{itsjustborell,sudakov} implies that there exists a constant $A>0$ such that 
    \begin{equation}\label{eq:lem4.27-os}
    \mathbb{P}[|O| \geq t] \leq Ae^{-t^2/A}, \qquad t>0.
    \end{equation}
    By the scaling invariance of $h$ (Lemma~\ref{lem:scale-invariance}), we may choose $A$ independent of $z$ and $r$. 

    \medskip
    
    \noindent\textbf{Step 1.} We first show that
    \begin{equation}\label{eq:lem4.27-internal}
       \mathbb{E}\Big{[}\Bigl(\sup_{u,v \in B_r(z)} D_{\mathring h}( u, v; B_{3r/2}(z))\Bigr)^p\Big{]} \leq r^{\xi Qp + o_r(1)} \qquad \mbox{as $r \to 0$}, 
    \end{equation}
    with $o_r(1)$ uniform in $z$. 
    
    By Proposition~\ref{compactbound} and Theorem~\ref{thm:sharp-c} (recall that we may take $\mathfrak c_r = r^{\xi Q}$), we know that for every $s \in (0, \frac{2d \mathsf d_\gamma}{\gamma^2})$ there exists $C_s>0$ such that
    \begin{equation}\label{eq:lem4.27-2}
        \mathbb{E} \Big{[} \Bigl(  e^{-\xi h_{r}(z)} \sup_{u,v \in B_r(z)} D_h\left( u, v ; B_{3r/2}(z)   \right)   \Bigr)^s \Big{]} \leq C_s \, r^{\xi Q s},  \qquad \forall r \in (0,1), \mbox{ } z \in \mathbb{R}^d.
    \end{equation}
    
    On the other hand, by~\eqref{eq:lem4.27-decompose} and Axiom~\ref{axiom-weyl}, 
    \begin{align*}
        \sup_{u,v \in B_r(z)} D_{\mathring h}\left( u, v ; B_{3r/2}(z)   \right) &\leq e^{- \xi \inf_{w \in B_{3r/2}(z)} \mathfrak h(w)}\sup_{u,v \in B_r(z)} D_h\left( u, v ; B_{3r/2}(z)   \right)\\
        &\leq e^{ - \xi \mathfrak h_r(z) + \xi O} \sup_{u,v \in B_r(z)} D_h\left( u, v ; B_{3r/2}(z)   \right)\\
        &= e^{\xi (\mathring{h}_r(z) + O)} \Bigl( e^{-\xi h_r(z) }\sup_{u,v \in B_r(z)} D_h\left( u, v ; B_{3r/2}(z)  \right) \Bigr).
    \end{align*} 
    Fix $q>1$ close enough to $1$ so that $pq < \frac{2d \mathsf d_\gamma}{\gamma^2}$. By H\"older's inequality,
    \begin{align*}
        &\quad \mathbb{E}\Big{[}\Bigl(\sup_{u,v \in B_r(z)} D_{\mathring h}( u, v; B_{3r/2}(z))\Bigr)^p\Big{]} \leq \mathbb{E} \Big[ e^{\xi p (\mathring{h}_r(z) + O)} \Bigl(e^{ - \xi h_r(z)} \sup_{u,v \in B_r(z)} D_h\left( u, v ; B_{3r/2}(z)   \right) \Bigr)^p \Big]\\
        &\leq \mathbb{E}\left[e^{\xi p (\mathring{h}_r(z) + O) q/(q-1)}\right]^{1-1/q}\mathbb{E} \Big[ \Bigl(e^{ - \xi h_r(z)} \sup_{u,v \in B_r(z)} D_h\left( u, v ;B_{3r/2}(z)   \right) \Bigr)^{pq} \Big]^{1/q} \leq C \, r^{\xi Qp},
    \end{align*}
    where in the last inequality we used $\mathbb{E}[\mathring{h}_r(z)^2] = O(1)$,~\eqref{eq:lem4.27-os}, and~\eqref{eq:lem4.27-2} applied with $s = pq$.
    
    \medskip
    
    \noindent\textbf{Step 2.} We now prove the second inequality in the lemma. By~\eqref{eq:lem4.27-decompose} and Axiom~\ref{axiom-weyl},
    \begin{equation}\label{eq:lem4.25-est-metric}
    \begin{aligned}
        \sup_{u,v \in B_r(z)} D_h( u, v; B_{3r/2}(z)) &\leq e^{\xi \sup_{w \in B_{3r/2}(z)} \mathfrak h(w)} \sup_{u,v \in B_r(z)} D_{\mathring h}( u, v; B_{3r/2}(z))  \\
        &\leq e^{\xi (\mathfrak h_r(z) + O)} \sup_{u,v \in B_r(z)} D_{\mathring h}( u, v; B_{3r/2}(z)).
    \end{aligned}
    \end{equation}
    Since $\mathfrak h$ and $\mathring h$ are independent and $D_{\mathring h}( \cdot, \cdot ; B_{3r/2}(z))$ is determined by $\mathring h$ by Axiom~\ref{axiom-local}, we have
    \[
    \mathbb{E}\Big{[}\Bigl(\sup_{u,v \in B_r(z)} D_h( u, v; B_{3r/2}(z))\Bigr)^p\Big{]} \leq \mathbb{E}\Big{[}e^{\xi p (\mathfrak h_r(z) + O)} \Big] \; \mathbb{E}\Big{[}\Bigl(\sup_{u,v \in B_r(z)} D_{\mathring h}( u, v; B_{3r/2}(z))\Bigr)^p\Big{]}.
    \]
    For any $q>1$, H\"older's inequality gives
    \[
    \mathbb{E}\Big{[}e^{\xi p (\mathfrak h_r(z) + O)} \Big] \leq \mathbb{E}\Big[e^{\xi pq\mathfrak h_r(z) }\Big]^{1/q} \; \mathbb{E}\Big[e^{\xi pO q/(q-1) }\Big]^{1 - 1/q} \leq Cr^{-\frac{1}{2} \xi^2p^2q}.
    \]
    Combining this with~\eqref{eq:lem4.27-internal}, we obtain
    \[
    \mathbb{E}\Big{[}\Bigl(\sup_{u,v \in B_r(z)} D_h( u, v; B_{3r/2}(z))\Bigr)^p\Big{]} \leq C r^{\xi Q p - \frac{1}{2} \xi^2p^2q}.
    \]
    Since $q$ can be taken arbitrarily close to $1$, we let $q \to 1$ at a sufficiently slow rate as $r \to 0$. Using that $\sup_{u,v \in B_r(z)} D_h( u, v; B_{3r/2}(z)) \geq \sup_{u,v \in B_r(z)} D_h( u, v)$, we obtain the second inequality.

    \medskip

    \noindent\textbf{Step 3.} We finally prove the first inequality in the lemma. We first consider replacing $h_r(z)$ with $\mathfrak h_r(z)$. By~\eqref{eq:lem4.25-est-metric}, we have
    \begin{align*}
        &\quad \mathbb{E}\Big{[}\mathbbm{1}_{\{|\mathfrak h_r(z) - \alpha \log r^{-1}| \leq \delta \log r^{-1}\}}\Bigl(\sup_{u,v \in B_r(z)} D_h( u, v; B_{3r/2}(z))\Bigr)^p\Big{]} \\
        &\leq \mathbb{E}\Big{[}\mathbbm{1}_{\{|\mathfrak h_r(z) - \alpha \log r^{-1}| \leq \delta \log r^{-1}\}} e^{\xi p (\mathfrak h_r(z) + O)} \Bigl(\sup_{u,v \in B_r(z)} D_{\mathring h}( u, v; B_{3r/2}(z))\Bigr)^p\Big{]}\\
        &= \mathbb{E}\Big{[}\mathbbm{1}_{\{|\mathfrak h_r(z) - \alpha \log r^{-1}| \leq \delta \log r^{-1}\}} e^{\xi p (\mathfrak h_r(z) + O)} \Big] \; \mathbb{E} \Big[ \Bigl(\sup_{u,v \in B_r(z)} D_{\mathring h}( u, v; B_{3r/2}(z))\Bigr)^p\Big{]},
    \end{align*}
    where the last equality is by the independence of $\mathfrak h$ and $\mathring h$.  For any $q>1$, H\"older's inequality gives
    \begin{align*}
        &\quad \mathbb{E}\Big{[}\mathbbm{1}_{\{|\mathfrak h_r(z) - \alpha \log r^{-1}| \leq \delta \log r^{-1}\}} e^{\xi p (\mathfrak h_r(z) + O)} \Big] \leq e^{\xi p(\alpha + \delta) \log r^{-1}} \mathbb{E}\Big{[}\mathbbm{1}_{\{|\mathfrak h_r(z) - \alpha \log r^{-1}| \leq \delta \log r^{-1}\}} e^{\xi p O} \Big]\\
        &\leq e^{\xi p(\alpha + \delta) \log r^{-1}} \mathbb{E}\Big{[}\mathbbm{1}_{\{|\mathfrak h_r(z) - \alpha \log r^{-1}| \leq \delta \log r^{-1}\}}\Big]^{1/q} \mathbb{E}\Big[e^{\xi p O q/(q-1)} \Big]^{1-1/q}\leq r^{-\xi p (\alpha + \delta) + (|\alpha|-\delta)_+^2/(2q) +o_r(1)}.
    \end{align*}
    Sending $q \to 1$ at a sufficiently slow rate as $r \to 0$ and combining with~\eqref{eq:lem4.27-internal} gives
    \[
    \mathbb{E}\Big{[}\mathbbm{1}_{\{|\mathfrak h_r(z) - \alpha \log r^{-1}| \leq \delta \log r^{-1}\}}\Bigl(\sup_{u,v \in B_r(z)} D_h( u, v; B_{3r/2}(z))\Bigr)^p\Big{]}  \leq r^{\xi p (Q - \alpha - \delta) + \frac{1}{2}(|\alpha|-\delta)_+^2 +o_r(1)}.
    \]
    Finally, by Hölder's inequality and $\mathbb{E}[\mathring h_r(z)^2] = O(1)$, for any fixed $\epsilon>0$, there exists $c>0$ such that
    \[
    \mathbb{E}\Big{[}\mathbbm{1}_{\{|\mathring h_r(z)| \geq \epsilon \log r^{-1} \}}\Bigl(\sup_{u,v \in B_r(z)} D_h( u, v; B_{3r/2}(z))\Bigr)^p\Big{]}  \leq r^{c \log r^{-1}}.
    \]
    Combining these two preceding estimates and using  $\sup_{u,v \in B_r(z)} D_h( u, v; B_{3r/2}(z)) \geq \sup_{u,v \in B_r(z)} D_h( u, v)$ gives the first inequality.
\end{proof}

We now prove Proposition~\ref{prop:kpz-1}.
\begin{proof}[Proof of Proposition~\ref{prop:kpz-1}]
Fix $\alpha \in [-\sqrt{2d}, \sqrt{2d}]$. Since $X$ is independent of $h$, by conditioning on $X$ we may assume that $X$ is deterministic. By the countable stability of Hausdorff dimension, we may assume that $X$ is bounded. Suppose that $X \subset B_M(0)$ for some $M>1$. We also assume that ${\rm dim}^0_{\mathcal H}(X) \geq \frac{\alpha^2}{2}$; otherwise $X \cap \mathcal{T}^\alpha_h = \emptyset$ a.s.\ by~\eqref{eq:dim-thick}, and both sides are zero.

\medskip
\noindent\textbf{Step 1.} We first show that 
\begin{equation}\label{eq:prop4.25-upper}
{\rm dim}^\gamma_{\mathcal H}(X \cap \mathcal{T}^\alpha_h) \leq \frac{1}{\xi(Q - \alpha)} {\rm dim}^0_{\mathcal H}(X \cap \mathcal{T}^\alpha_h).
\end{equation}

Fix any
\[
x > {\rm dim}^0_{\mathcal H}(X) \qquad \mbox{and} \qquad \delta>0,
\]
which will eventually tend to ${\rm dim}^0_{\mathcal H}(X)$ and $0$, respectively. For $u>0$ and integers $n\geq 1$, define 
\begin{equation}\label{eq:def-Thaun}
\mathcal{T}_h^{\alpha; u, n} = \Big{\{}z \in \mathbb{R}^d: \frac{h_\epsilon(z)}{\log \epsilon^{-1}} \in [\alpha - u, \alpha + u] \,\mbox{ for all }0 < \epsilon \leq 2^{-n} \Big{\}}.
\end{equation}
Recall $\mathcal{T}_h^{\alpha}$ from~\eqref{eq:thick-point}. Then
\begin{equation}\label{eq:sec5.1.1-inclusion}
\mathcal{T}_h^{\alpha} \subset \bigcup_{n \geq 1}  \mathcal{T}_h^{\alpha; u, n} \quad \mbox{and} \quad \mathcal{T}_h^{\alpha} = \bigcap_{u>0} \bigcup_{n \geq 1}  \mathcal{T}_h^{\alpha; u, n}.
\end{equation}

We now give an upper bound for ${\rm dim}^\gamma_{\mathcal H}(X \cap \mathcal{T}_h^{\alpha; \delta, n})$ for any integer $n \geq 1$. Observe that there exists a random $\overline{r} > 0$ (depending on $h$) such that the following holds for all $\epsilon \in (0, \overline{r}]$ and $u, v \in B_{2M}(0)$ with $|u-v| \leq \epsilon$:
\begin{equation}\label{eq:sec5.1.1-2}
|h_\epsilon(u) - h_\epsilon(v)| \leq \delta \log \epsilon^{-1}.
\end{equation}
(Note that $h_\epsilon(u) - h_\epsilon(v)$ is a centered Gaussian with variance $O(1)$ for such $u,v$, so \eqref{eq:sec5.1.1-2} follows from a union bound over a discrete net and the Borel-Cantelli lemma.) 

Since $x > {\rm dim}^0_{\mathcal H}(X)$, for any integer $k \geq 1$, we can choose a covering of $X$ by balls $\{B_{r_i}(z_i) \}_{i \geq 1}$ such that 
\begin{equation}\label{eq:sec5.1.1-hausdorff}
\sum_{i \geq 1} r_i^x \leq 2^{-k} \quad \mbox{and} \quad \sup_{i \geq 1} r_i \leq \max \{ \overline{r}, 2^{-n}\}.
\end{equation}
Let \[I := \{i \geq 1: B(z_i, r_i) \cap \mathcal{T}_h^{\alpha; \delta, n}  \neq \emptyset \}.\] Then $X \cap \mathcal{T}_h^{\alpha; \delta, n} \subset \cup_{i \in I} B(z_i, r_i)$. For any $i \in I$, there exists $w \in B(z_i,r_i) \cap \mathcal{T}_h^{\alpha; \delta, n}$. By~\eqref{eq:sec5.1.1-2}, we have $|h_{r_i}(z_i) - h_{r_i}(w)| \leq \delta \log r_i^{-1}$. By the definition of $\mathcal{T}_h^{\alpha; \delta, n}$ in~\eqref{eq:def-Thaun} and $r_i \leq 2^{-n}$, we have $\frac{h_{r_i}(w_i)}{\log r_i^{-1}} \in [\alpha - \delta, \alpha + \delta]$. Therefore,
$$
|h_{r_i}(z_i) - \alpha \log r_i^{-1}| \leq 2 \delta \log r_i^{-1}.
$$

Hence, for $p \in (0, \frac{2d \mathsf d_\gamma}{\gamma^2})$, Lemma~\ref{lem:sec4.6-Dh} (applied with $2 \delta$ in place of $\delta$ and $R = 2M$) gives
\begin{align*}
    \quad \mathbb{E} \Big{[}\sum_{i \in I} \Bigl(\sup_{u,v \in B_{r_i}(z_i)} D_h(u,v) \Bigr)^p \Big{]}& \leq \sum_{i \geq 1} \mathbb{E} \Big{[}\mathbbm{1}_{\{|h_{r_i}(z_i) - \alpha \log r_i^{-1}| \leq 2 \delta \log r_i^{-1}\}}  \Bigl(\sup_{u,v \in B_{r_i}(z_i)} D_h(u,v) \Bigr)^p \Big{]} \\
    &\leq \sum_{i \geq 1} r_i^{\xi p (Q - \alpha - 2\delta)+ \frac{1}{2}(|\alpha| - 2\delta)_+^2 + o_n(1)},
\end{align*}
where $o_n(1) \to 0$ as $n \to \infty$ and is uniform in $i$.

Choose $p$ such that \[\xi p (Q - \alpha - 2\delta)+ \frac{1}{2}(|\alpha| - 2\delta)_+^2 = x + \delta^2.\] 
Because $x$ and $\delta$ are chosen close to $\dim^0_{\mathcal H}(X)$ and $0$, respectively, and using that $d \geq \dim^0_{\mathcal H}(X) \geq \frac{\alpha^2}{2}$, $\alpha \in [-\sqrt{2d},\sqrt{2d}]$, one checks that such $p$ belongs to $(0,\frac{2d \mathsf d_\gamma}{\gamma^2})$. Combining the preceding inequality with~\eqref{eq:sec5.1.1-hausdorff}, we get for sufficiently large $n$,
\[
\mathbb{E} \Bigl[\sum_{i \in I} \Bigl(\sup_{u,v \in B_{r_i}(z_i)} D_h(u,v) \Bigr)^p \Bigr] 
\leq \sum_{i \geq 1} r_i^{x + \delta^2 + o_n(1)} \leq \sum_{i \geq 1} r_i^x \leq 2^{-k}.
\]
By Markov's inequality, with probability at least $1 - 2^{-k/2}$, the sum in brackets is at most $2^{-k/2}$. By the Borel-Cantelli lemma, this event occurs a.s.\ for all sufficiently large integer $k$. It follows from the definition of $D_h$-Hausdorff dimension that for sufficiently large $n$,
$$
{\rm dim}^\gamma_{\mathcal H}(X \cap \mathcal{T}_h^{\alpha; \delta, n}) \leq  p = \frac{x - \frac{1}{2}(|\alpha| - 2\delta)_+^2 + \delta^2}{\xi ( Q - \alpha - 2\delta)}.
$$
Letting $x$ to ${\rm dim}^0_{\mathcal H}(X)$ and using that $\mathcal{T}_h^{\alpha; \delta, n}$ is increasing in $n$ gives 
$$
{\rm dim}^\gamma_{\mathcal H}\bigl(X \cap \bigcup_{n \geq 1} \mathcal{T}_h^{\alpha; \delta, n}\bigr) \leq  \frac{{\rm dim}^0_{\mathcal H}(X) - \frac{1}{2}(|\alpha| - 2\delta)_+^2 + \delta^2}{\xi ( Q - \alpha - 2\delta)}.
$$
Combining this with~\eqref{eq:sec5.1.1-inclusion} and then letting $\delta \to 0$, and using~\eqref{eq:dim-thick}, yields~\eqref{eq:prop4.25-upper}.

\medskip

\noindent\textbf{Step 2.} We now show that 
\begin{equation}\label{eq:prop4.25-lower}
{\rm dim}^\gamma_{\mathcal H}(X \cap \mathcal{T}^\alpha_h) \geq \frac{1}{\xi(Q - \alpha)} {\rm dim}^0_{\mathcal H}(X \cap \mathcal{T}^\alpha_h).
\end{equation}

Fix any $\delta>0$, which will eventually tend to 0. The key estimate we will use about $D_h$ is Proposition~\ref{prop:superpolynomial-cross}: there exists a random integer $m_* \geq 1$ (depending on $h$) such that
\begin{equation}\label{lem-annulus-dist-uniform}
D_h(\partial B_{2^{-m-1}}(z), \partial B_{2^{-m}}(z)) \geq 2^{-(\xi Q +\delta)m} e^{\xi h_{2^{-m}}(z)}
\end{equation}
for all integer $m \geq m_*$ and all $z \in \frac{1}{10d} 2^{-m} \mathbb{Z}^d \cap B_{2M}(0)$. In fact, by Proposition~\ref{prop:superpolynomial-cross}, for fixed $m$ and $z$, the event in~\eqref{lem-annulus-dist-uniform} occurs with superpolynomially high probability in $2^{-m}$ as $m \to \infty$, which implies \eqref{lem-annulus-dist-uniform} by the Borel-Cantelli lemma.

Fix an integer $n \geq 1$. We claim that \begin{equation}\label{eq:sec5.1.1-claim}
\begin{aligned}
&\mbox{for any $D_h$-metric ball $\widetilde B_r(z) := \{y \in \mathbb{R}^d: D_h(y,z) < r\}$ that intersects $X \cap \mathcal{T}_h^{\alpha} \cap \mathcal{T}^{\alpha; \delta, n}_h$},\\
&\mbox{we can cover it by a Euclidean ball with radius $r^{\tfrac{1}{\xi(Q - \alpha) + (1+ 2\xi)\delta} +o_r(1)}$ as $r \to 0$,}
\end{aligned}
\end{equation}
where the $o_r(1)$ term is uniform in $z$.

We may assume $r$ is sufficiently small. Let $m$ be a large integer to be chosen depending on $r$. Let $w \in \widetilde B_r(z) \cap X \cap \mathcal{T}_h^{\alpha} \cap \mathcal{T}^{\alpha; \delta, n}_h$. Choose $w_m \in \frac{1}{10d} 2^{-m} \mathbb{Z}^d$ closest to $w$. Then $w \in B_{2^{-m-1}}(w_m)$. By the definition of $\mathcal{T}_h^{\alpha; \delta, n}$ in~\eqref{eq:def-Thaun}, we have $h_{2^{-m}}(w)/\log 2^m \in [\alpha - \delta, \alpha + \delta]$ whenever $m \geq n$. Moreover, \eqref{eq:sec5.1.1-2} implies that $|h_{2^{-m}}(w) - h_{2^{-m}}(w_m)| \leq \delta \log 2^m$. Combining these gives
\[
|h_{2^{-m}}(w_m) - \alpha \log 2^m | \leq 2 \delta \log 2^m.
\]
Together with~\eqref{lem-annulus-dist-uniform}, this yields
\[
D_h(\partial B_{2^{-m-1}}(w_m), \partial B_{2^{-m}}(w_m)) \geq 2^{-(\xi Q +\delta)m} e^{\xi h_{2^{-m}}(w_m)} \geq 2^{-m[\xi(Q-\alpha) + (1 + 2 \xi)\delta]}.
\]
Therefore, if $2^{-m[\xi(Q-\alpha) + (1 + 2 \xi)\delta]} \geq 2r$, then $\widetilde B_r(z)$ is contained in $B_{2^{-m}}(w_m)$. This proves~\eqref{eq:sec5.1.1-claim}.

It follows from \eqref{eq:sec5.1.1-claim} and the definitions of Euclidean and $D_h$-Hausdorff dimensions that 
\[
\dim^\gamma_{\mathcal H}(X \cap \mathcal{T}_h^{\alpha} \cap \mathcal{T}_h^{\alpha; \delta, n}) \geq 
\frac{1}{\xi(Q - \alpha) + (1+ 2\xi)\delta} \dim^0_{\mathcal H}(X \cap \mathcal{T}_h^{\alpha} \cap \mathcal{T}_h^{\alpha; \delta, n}), \qquad \forall n\geq 1.
\]
Therefore,
\[
\dim^\gamma_{\mathcal H}(X \cap \mathcal{T}_h^{\alpha}) \geq 
\frac{1}{\xi(Q - \alpha) + (1+ 2\xi)\delta} \dim^0_{\mathcal H}\bigl(X \cap \mathcal{T}_h^{\alpha} \cap \bigcup_{n \geq 1}\mathcal{T}_h^{\alpha; \delta, n}\bigr).
\]
Combining this with~\eqref{eq:sec5.1.1-inclusion} and letting $\delta \to 0$ gives~\eqref{eq:prop4.25-upper}.
\end{proof}

We now prove Theorem~\ref{thm:kpz}.

\begin{proof}[Proof of Proof of Theorem~\ref{thm:kpz}]

It suffices to prove that
\[
{\rm dim}_{\mathcal H}^\gamma X = \frac{1}{\xi}\bigl(Q - \sqrt{Q^2 - 2 {\rm dim}_{\mathcal H}^0 X}\bigr).
\]
For the lower bound, note that ${\rm dim}^\gamma_{\mathcal H}(X) \geq {\rm dim}^\gamma_{\mathcal H}(X \cap \mathcal{T}^\alpha_h)$ for every $\alpha$. Taking $\alpha = Q - \sqrt{Q^2 - 2 \dim_{\mathcal H}^0 X}$ in Proposition~\ref{prop:kpz-1} and using~\eqref{eq:dim-thick} gives the desired lower bound.

We now prove the upper bound. We assume that $X$ is deterministic and bounded. Fix $x > {\rm dim}^0_{\mathcal H}(X)$, which will eventually tend to ${\rm dim}^0_{\mathcal H}(X)$. By the definition of Euclidean Hausdorff dimension, for each integer $k \geq 1$ we can choose a deterministic covering $\{B_{r_i}(z_i)\}_{i \geq 1}$ of $X$ such that \[
\sum_{i \geq 1} r_i^x <2^{-k} \qquad  \mbox{and} \qquad \sup_{i \geq 1} r_i \leq 2^{-k}.\] By Lemma~\ref{lem:sec4.6-Dh}, for any $p \in (0, \frac{2d \mathsf d_\gamma}{\gamma^2})$,
\[
\mathbb{E} \Bigl[ \sum_{i \geq 1} \Bigl(\sup_{u,v \in B_{r_i}(z_i)} D_h(u,v) \Bigr)^p \Bigr]
\leq \sum_{i \geq 1} r_i^{\xi Q p - \tfrac{1}{2}\xi^2 p^2 + o_k(1)},
\]
where the $o_k(1)$ term is uniform in $i$, using that $r_i \leq 2^{-k}$.

Suppose that $\xi Q p - \tfrac{1}{2}\xi^2 p^2 > x$. Then, for $k$ large enough, the right-hand side is bounded by $2^{-k}$. By Markov's inequality and the Borel-Cantelli lemma, this implies that the sum in brackets converges to 0 a.s.\ as $k \to \infty$. By the definition of $D_h$-Hausdorff dimension, we conclude that $\dim^\gamma_{\mathcal H}(X) \leq p$. Since this holds for every $p$ with $\xi Q p - \tfrac{1}{2}\xi^2 p^2 > x$, we obtain $\dim^\gamma_{\mathcal H}(X) \leq \frac{1}{\xi}(Q - \sqrt{Q^2 - 2 x})$. Letting $x \to {\rm dim}^0_{\mathcal H}(X)$ yields the desired upper bound.
\end{proof}

\appendix

\section{Appendix}

In this section we prove a general estimate for integrals of exponentials of Brownian-like Gaussian processes with a continuity estimate condition.

We will consider Gaussian processes $X_t$ such that
\begin{equation}\label{gc1}
\mathbb{E}(X_t) =0,\; \mathbb{E}(X_t) = t+o_t(1),
\end{equation}
\begin{equation}\label{gc2}
\mathbb{P}(\sup_{0<s<1} X_s-X_0 \geq C) \lesssim e^{-c_0 C^2}
\end{equation}
for some absolute constant $C_0.$ Note in particular that this implies that $X_t$ has a Gaussian tail.
\begin{lemma}\label{globalmaxbound}
    Let $X_t$ be a centered Gaussian process such that \eqref{gc1} and \eqref{gc2} hold. Additionally, let $a>0$ be a positive real number, and let $G_t:=X_t-at$ be the drifted process. Then there exist positive constants $c_1,c_2$ such that
    \[
c_1 e^{-(2a+o_y(1)) y} \leq \mathbb{P}( \mathrm{sup}_{t \geq 0} G_t >y) \leq c_2 e^{-(2a+o_y(1)) y}.
    \]
\end{lemma}
\begin{proof}
We will show each inequality independently. First, note that
\[
\mathbb{P}( \mathrm{sup}_{t \geq 0} G_t >y) \geq \mathbb{P}( G_{\frac{y}{a}} >y) = \mathbb{P}( X_{\frac{y}{a}} >2y) \lesssim e^{-\left(\frac{(2y)^2}{\left(\frac{2y}{a}\right)}\right)} \leq e^{-(2a+o(1))y}.
\]
This proves that $\mathbb{P}( \mathrm{sup}_{t \geq 0} G_t >y)$ is bounded below. For the other direction, note that
\begin{eqnarray*}
\mathbb{P}\left(\sup_{0\leq t \leq y^2} G_t>y\right) &\leq& \mathbb{P}(\sup_{t \in[0,y^2]\cap \mathbb{Z}} G_t > y-y^{\frac{2}{3}}) + \mathbb{P}\left(\sup_{t \in [0,y^2]\cap \mathbb{Z}}\sup_{t<s<t+1}G_s-G_t\geq y^{\frac{2}{3}}\right).
\end{eqnarray*}
By \eqref{gc2} we have that
\[
\mathbb{P}\left(\sup_{t \in [0,y^2]\cap \mathbb{Z}}\sup_{t<s<t+1}G_s-G_t\geq y^{\frac{2}{3}}\right) \leq y^2 e^{-c_0y^{\frac{4}{3}}}.
\]
Therefore
\begin{eqnarray*}
\mathbb{P}\left(\sup_{0\leq t \leq y^2} G_t>y\right) &\lesssim& \sum_{k=0}^{\lfloor y^2 \rfloor} \mathbb{P}(X_k > y-y^{\frac{2}{3}+ak}) + y^2 e^{-c_0y^{\frac{4}{3}}}\\
&\lesssim& \sum_{k=0}^{\lfloor y^2 \rfloor} \exp\left(-\frac{(y+ak)^2}{2k}\right)+y^2 e^{-c_0y^{\frac{4}{3}}}\\
&\lesssim & e^{-(2a+o(1))y}.
\end{eqnarray*}
This completes the proof.
\end{proof}

Now we prove the following bound on an exponential integral of processes satisfying \eqref{gc1} and \eqref{gc2}.

\begin{prop}\label{expbound}
Let $X_t$ be a centered Gaussian process such that \eqref{gc1} and \eqref{gc2} hold. Additionally, let $a>0$ be a positive real number, and let $G_t:=X_t-at$ be the drifted process. Then there exist constants $c_1,c_2>0$ such that
\[
c_1 x^{-2a+o(1)} \leq\mathbb{P}\left(\int_0^\infty e^{G_t} dt >x\right) \leq c_2 x^{-2a+o(1)}.
\]
\end{prop}
\begin{proof}
We show each inequality separately. We start by showing the lower bound. Let $y=\log x.$ Then note that if $G_t \geq y$ for $t \in \left[\frac{y}{a},\frac{y}{a}+1\right],$ then $\int_0^\infty e^{G_t} dt >x.$ This implies that
\[
\mathbb{P}\left(\int_0^\infty e^{G_t} dt >x\right) \geq \mathbb{P}\left(G_t \geq y \mbox{ for }t \in \left[\frac{y}{a},\frac{y}{a}+1\right]\right) \geq \mathbb{P} (G_{\frac{y}{a}}\geq y) = \mathbb{P}(X_{\frac{y}{a}}\geq 2y)\geq e^{-2ay}.
\]
This proves the lower bound. For the upper bound, we let $E$ be the event
\[
E = \left\{G_t \leq -\frac{1}{2}at \mbox{ for all } t \geq (\log x)^4\right\}.
\]
Using that
\[
\mathbb{P}\left(G_t >-\frac{1}{2}at\right) = \mathbb{P}\left(X_t\geq \frac{1}{2}at\right) \lesssim \exp\left(-\frac{1}{8}a^2 t\right) \ll x^{-2a},
\]
we see that $1- \mathbb{P}(E)\ll X^{-2a}.$ Now we see that if $E$ holds and
\[
\int_0^\infty e^{G_t} dt \geq x,
\]
then $\sup_{0\leq t \leq (\log x)^4}G_t \geq \log \left(\frac{x}{(\log x)^4}\right) = \log x (1-o(1)).$ Therefore
\[
\mathbb{P}\left(\int_0^\infty e^{G_t} dt >x\right) \lesssim (1-\mathbb{P}(E)) + \mathbb{P}\left(\sup_{0\leq t \leq (\log x)^4}G_t \geq \log x (1-o(1))\right).
\]
Using Lemma \ref{globalmaxbound}, we conclude.
\end{proof}

\bibliographystyle{alpha}
\bibliography{ref}

@article{lqg-metric-estimates,
    AUTHOR = {Dub\'{e}dat, Julien and Falconet, Hugo and Gwynne, Ewain and
              Pfeffer, Joshua and Sun, Xin},
     TITLE = {Weak {LQG} metrics and {L}iouville first passage percolation},
   JOURNAL = {Probab. Theory Related Fields},
  FJOURNAL = {Probability Theory and Related Fields},
    VOLUME = {178},
      YEAR = {2020},
    NUMBER = {1-2},
     PAGES = {369--436},
      ISSN = {0178-8051},
   MRCLASS = {60D05 (60G60)},
  MRNUMBER = {4146541},
       DOI = {10.1007/s00440-020-00979-6},
       URL = {https://doi.org/10.1007/s00440-020-00979-6},
   eprint = {\arxiv{1905.00380}}, 
}

@article {Pfeffer-weak-metric,
    AUTHOR = {Pfeffer, Joshua},
     TITLE = {Weak {L}iouville quantum gravity metrics with matter central
              charge {$c \in (-\infty,25)$}},
   JOURNAL = {Probab. Math. Phys.},
  FJOURNAL = {Probability and Mathematical Physics},
    VOLUME = {5},
      YEAR = {2024},
    NUMBER = {3},
     PAGES = {545--608},
      ISSN = {2690-0998,2690-1005},
   MRCLASS = {60D05 (60G60 83C45)},
  MRNUMBER = {4765494},
       DOI = {10.2140/pmp.2024.5.545},
       URL = {https://doi.org/10.2140/pmp.2024.5.545},
}

@article{local-metrics,
    AUTHOR = {Gwynne, Ewain and Miller, Jason},
     TITLE = {Local metrics of the {G}aussian free field},
   JOURNAL = {Ann. Inst. Fourier (Grenoble)},
  FJOURNAL = {Universit\'{e} de Grenoble. Annales de l'Institut Fourier},
    VOLUME = {70},
      YEAR = {2020},
    NUMBER = {5},
     PAGES = {2049--2075},
      ISSN = {0373-0956},
   MRCLASS = {60D05 (60G15 60G60)},
  MRNUMBER = {4245606},
       URL = {http://aif.cedram.org/item?id=AIF_2020__70_5_2049_0},
   eprint = {\arxiv{1905.00379}},
}

@ARTICLE{dgz-exponential-metric,
       author = {{Ding}, Jian and {Gwynne}, Ewain and {Zhuang}, Zijie},
        title = "{Tightness of exponential metrics for log-correlated Gaussian fields in arbitrary dimension}",
      journal = {arXiv e-prints},
         year = 2023,
        month = oct,
          eid = {arXiv:2310.03996},
        pages = {arXiv:2310.03996},
          doi = {10.48550/arXiv.2310.03996},
archivePrefix = {arXiv},
       eprint = {2310.03996},
 primaryClass = {math.PR},
       adsurl = {https://ui.adsabs.harvard.edu/abs/2023arXiv231003996D}
}

@article {fgf-survey,
    AUTHOR = {Lodhia, Asad and Sheffield, Scott and Sun, Xin and Watson,
              Samuel S.},
     TITLE = {Fractional {G}aussian fields: a survey},
   JOURNAL = {Probab. Surv.},
  FJOURNAL = {Probability Surveys},
    VOLUME = {13},
      YEAR = {2016},
     PAGES = {1--56},
      ISSN = {1549-5787},
   MRCLASS = {60G15 (60G20 60G60)},
  MRNUMBER = {3466837},
MRREVIEWER = {Serge\ Cohen},
       DOI = {10.1214/14-PS243},
       URL = {https://doi.org/10.1214/14-PS243},
}

@article{chg-support,
  title={A support theorem for Liouville quantum gravity},
  author={Hip, Andres A Contreras and Gwynne, Ewain},
  journal={arXiv preprint arXiv:2305.15588},
  year={2023}
}

@article{itsjustborell,
  title={The brunn-minkowski inequality in gauss space},
  author={Borell, Christer},
  journal={Inventiones mathematicae},
  volume={30},
  number={2},
  pages={207--216},
  year={1975},
  publisher={Springer-Verlag Berlin/Heidelberg}
}

@article{sudakov,
  title={Extremal properties of half-spaces for spherically invariant measures},
  author={Sudakov, Vladimir N and Tsirel'son, Boris S},
  journal={Journal of Soviet Mathematics},
  volume={9},
  number={1},
  pages={9--18},
  year={1978},
  publisher={Springer}
}

@book{rfgbook,
  title={Random fields and geometry},
  author={Adler, Robert J and Taylor, Jonathan E},
  year={2009},
  publisher={Springer Science \& Business Media}
}

@article{thick,
  title={Thick points of the Gaussian free field},
  author={Hu, Xiaoyu and Miller, Jason and Peres, Yuval},
  year={2010}
}

@article{otherthick,
  title={Thick points for Gaussian free fields with different cut-offs},
  author={Cipriani, Alessandra and Hazra, Rajat Subhra},
  journal={arXiv preprint arXiv:1407.5840},
  year={2014}
}

@article{lgf-survey,
  title={Log-correlated Gaussian fields: an overview},
  author={Duplantier, Bertrand and Rhodes, R{\'e}mi and Sheffield, Scott and Vargas, Vincent},
  journal={Geometry, Analysis and Probability: In Honor of Jean-Michel Bismut},
  pages={191--216},
  year={2017},
  publisher={Springer}
}

@article {MQ18-geodesic,
    AUTHOR = {Miller, Jason and Qian, Wei},
     TITLE = {The geodesics in {L}iouville quantum gravity are not
              {S}chramm-{L}oewner evolutions},
   JOURNAL = {Probab. Theory Related Fields},
  FJOURNAL = {Probability Theory and Related Fields},
    VOLUME = {177},
      YEAR = {2020},
    NUMBER = {3-4},
     PAGES = {677--709},
      ISSN = {0178-8051},
   MRCLASS = {60D05 (60J67)},
  MRNUMBER = {4126929},
       DOI = {10.1007/s00440-019-00949-7},
       URL = {https://doi.org/10.1007/s00440-019-00949-7},
}

@incollection {DDG-ICM,
    AUTHOR = {Ding, Jian and Dub\'{e}dat, Julien and Gwynne, Ewain},
     TITLE = {Introduction to the {L}iouville quantum gravity metric},
 BOOKTITLE = {I{CM}---{I}nternational {C}ongress of {M}athematicians. {V}ol.
              6. {S}ections 12--14},
     PAGES = {4212--4244},
 PUBLISHER = {EMS Press, Berlin},
      YEAR = {[2023] \copyright 2023},
   MRCLASS = {60D05 (60G60)},
  MRNUMBER = {4680401},
MRREVIEWER = {Yizheng Yuan},
}

@article {sheffield-gff,
    AUTHOR = {Sheffield, Scott},
     TITLE = {Gaussian free fields for mathematicians},
   JOURNAL = {Probab. Theory Related Fields},
  FJOURNAL = {Probability Theory and Related Fields},
    VOLUME = {139},
      YEAR = {2007},
    NUMBER = {3-4},
     PAGES = {521--541},
      ISSN = {0178-8051},
   MRCLASS = {60K35 (60J65 81T10 82B31)},
  MRNUMBER = {2322706},
MRREVIEWER = {Ofer Zeitouni},
       DOI = {10.1007/s00440-006-0050-1},
       URL = {https://doi.org/10.1007/s00440-006-0050-1},
}

@article {JSW-log,
    AUTHOR = {Junnila, Janne and Saksman, Eero and Webb, Christian},
     TITLE = {Imaginary multiplicative chaos: moments, regularity and
              connections to the {I}sing model},
   JOURNAL = {Ann. Appl. Probab.},
  FJOURNAL = {The Annals of Applied Probability},
    VOLUME = {30},
      YEAR = {2020},
    NUMBER = {5},
     PAGES = {2099--2164},
      ISSN = {1050-5164},
   MRCLASS = {60G20 (60G15 82B20)},
  MRNUMBER = {4149524},
       DOI = {10.1214/19-AAP1553},
       URL = {https://doi.org/10.1214/19-AAP1553},
}

@ARTICLE{CZ-bound,
       author = {{Contreras Hip}, Andres A. and {Zhuang}, Zijie},
        title = "{Bounds on the distance exponent for higher-dimensional Liouville first passage percolation}",
      journal = {arXiv e-prints},
         year = 2025,
        month = apr,
          eid = {arXiv:2504.09141},
        pages = {arXiv:2504.09141},
          doi = {10.48550/arXiv.2504.09141},
archivePrefix = {arXiv},
       eprint = {2504.09141},
 primaryClass = {math.PR},
       adsurl = {https://ui.adsabs.harvard.edu/abs/2025arXiv250409141C}
}

@article {GP-kpz,
    AUTHOR = {Gwynne, Ewain and Pfeffer, Joshua},
     TITLE = {K{PZ} formulas for the {L}iouville quantum gravity metric},
   JOURNAL = {Trans. Amer. Math. Soc.},
  FJOURNAL = {Transactions of the American Mathematical Society},
    VOLUME = {375},
      YEAR = {2022},
    NUMBER = {12},
     PAGES = {8297--8324},
      ISSN = {0002-9947},
   MRCLASS = {60D05 (60G60)},
  MRNUMBER = {4504639},
MRREVIEWER = {Ellen Powell},
       DOI = {10.1090/tran/8085},
       URL = {https://doi.org/10.1090/tran/8085},
}

@article{gm-existence,
  title={Existence and uniqueness of the Liouville quantum gravity metric for $\gamma \in$(0, 2)},
  author={Gwynne, Ewain and Miller, Jason},
  journal={Inventiones mathematicae},
  volume={223},
  number={1},
  pages={213--333},
  year={2021},
  publisher={Springer}
}

@article {GHM-KPZ,
    AUTHOR = {Gwynne, Ewain and Holden, Nina and Miller, Jason},
     TITLE = {An almost sure {KPZ} relation for {SLE} and {B}rownian motion},
   JOURNAL = {Ann. Probab.},
  FJOURNAL = {The Annals of Probability},
    VOLUME = {48},
      YEAR = {2020},
    NUMBER = {2},
     PAGES = {527--573},
      ISSN = {0091-1798},
   MRCLASS = {60G60 (60J67)},
  MRNUMBER = {4089487},
MRREVIEWER = {Juhan Aru},
       DOI = {10.1214/19-AOP1385},
       URL = {https://doi.org/10.1214/19-AOP1385},
}

@article {DDDF-tightness,
    AUTHOR = {Ding, Jian and Dub\'{e}dat, Julien and Dunlap, Alexander and
              Falconet, Hugo},
     TITLE = {Tightness of {L}iouville first passage percolation for
              {$\gamma \in (0,2)$}},
   JOURNAL = {Publ. Math. Inst. Hautes \'{E}tudes Sci.},
  FJOURNAL = {Publications Math\'{e}matiques. Institut de Hautes \'{E}tudes
              Scientifiques},
    VOLUME = {132},
      YEAR = {2020},
     PAGES = {353--403},
      ISSN = {0073-8301},
   MRCLASS = {60K35 (60G15 83C45)},
  MRNUMBER = {4179836},
       DOI = {10.1007/s10240-020-00121-1},
       URL = {https://doi.org/10.1007/s10240-020-00121-1},
}

@article{dg-uniqueness,
    AUTHOR = {Ding, Jian and Gwynne, Ewain},
     TITLE = {Uniqueness of the critical and supercritical {L}iouville
              quantum gravity metrics},
   JOURNAL = {Proc. Lond. Math. Soc. (3)},
  FJOURNAL = {Proceedings of the London Mathematical Society. Third Series},
    VOLUME = {126},
      YEAR = {2023},
    NUMBER = {1},
     PAGES = {216--333},
      ISSN = {0024-6115},
   MRCLASS = {60G60 (60D05 83C45)},
  MRNUMBER = {4535021},
       eprint = {\arxiv{2110.00177}}, 
}

@article {dg-constant,
    AUTHOR = {Ding, Jian and Gwynne, Ewain},
     TITLE = {Up-to-constants comparison of {L}iouville first passage
              percolation and {L}iouville quantum gravity},
   JOURNAL = {Sci. China Math.},
  FJOURNAL = {Science China. Mathematics},
    VOLUME = {66},
      YEAR = {2023},
    NUMBER = {5},
     PAGES = {1053--1072},
      ISSN = {1674-7283},
   MRCLASS = {60D05 (60G60)},
  MRNUMBER = {4577490},
MRREVIEWER = {Anatoliy Malyarenko},
       DOI = {10.1007/s11425-021-1983-0},
       URL = {https://doi.org/10.1007/s11425-021-1983-0},
}

@article {efron-stein,
    AUTHOR = {Efron, B. and Stein, C.},
     TITLE = {The jackknife estimate of variance},
   JOURNAL = {Ann. Statist.},
  FJOURNAL = {The Annals of Statistics},
    VOLUME = {9},
      YEAR = {1981},
    NUMBER = {3},
     PAGES = {586--596},
      ISSN = {0090-5364},
     CODEN = {ASTSC7},
   MRCLASS = {62G05},
  MRNUMBER = {615434 (82k:62074)},
MRREVIEWER = {Vishnu Dayal Jha},
       URL =
              {http://links.jstor.org/sici?sici=0090-5364(198105)9:3<586:TJEOV>2.0.CO;2-M&origin=MSN},
}

@article{dg-supercritical-lfpp,
    AUTHOR = {Ding, Jian and Gwynne, Ewain},
     TITLE = {Tightness of supercritical {L}iouville first passage
              percolation},
   JOURNAL = {J. Eur. Math. Soc. (JEMS)},
  FJOURNAL = {Journal of the European Mathematical Society (JEMS)},
    VOLUME = {25},
      YEAR = {2023},
    NUMBER = {10},
     PAGES = {3833--3911},
      ISSN = {1435-9855},
   MRCLASS = {Prelim},
  MRNUMBER = {4634685},
       DOI = {10.4171/jems/1273},
       URL = {https://doi.org/10.4171/jems/1273},
       eprint = {\arxiv{2005.13576}},
}

@incollection {sheffield-icm,
    AUTHOR = {Sheffield, Scott},
     TITLE = {What is a random surface?},
 BOOKTITLE = {I{CM}---{I}nternational {C}ongress of {M}athematicians. {V}ol.
              2. {P}lenary lectures},
     PAGES = {1202--1258},
 PUBLISHER = {EMS Press, Berlin},
      YEAR = {[2023] \copyright 2023},
   MRCLASS = {60G60 (60G57 60J67 60K35)},
  MRNUMBER = {4680280},
MRREVIEWER = {Mark Kelbert},
}

@article{miermont-brownian-map,
AUTHOR = {Miermont, Gr{\'e}gory},
     TITLE = {The {B}rownian map is the scaling limit of uniform random
              plane quadrangulations},
   JOURNAL = {Acta Math.},
  FJOURNAL = {Acta Mathematica},
    VOLUME = {210},
      YEAR = {2013},
    NUMBER = {2},
     PAGES = {319--401},
      ISSN = {0001-5962},
   MRCLASS = {60D05 (05Cxx 52C17 60F05 60G57 60J65)},
  MRNUMBER = {3070569},
       DOI = {10.1007/s11511-013-0096-8},
       URL = {http://dx.doi.org/10.1007/s11511-013-0096-8},
eprint={\arxiv{1104.1606}}
}

@article{legall-uniqueness,
AUTHOR = {{Le Gall}, Jean-Fran{\c{c}}ois},
     TITLE = {Uniqueness and universality of the {B}rownian map},
   JOURNAL = {Ann. Probab.},
  FJOURNAL = {The Annals of Probability},
    VOLUME = {41},
      YEAR = {2013},
    NUMBER = {4},
     PAGES = {2880--2960},
      ISSN = {0091-1798},
   MRCLASS = {60D05 (05C80 60F17)},
  MRNUMBER = {3112934},
       DOI = {10.1214/12-AOP792},
       URL = {http://dx.doi.org/10.1214/12-AOP792},
eprint={\arxiv{1105.4842}}
}

@article{kpz-scaling,
author={Knizhnik, V.G. and Polyakov, A.M. and Zamolodchikov, A.B.},
title={{Fractal structure of 2D-quantum gravity}},
journal={{Modern Phys. Lett A}},
volume= {3},
number={8},
year= {1988},
pages={ 819-826}
}

\end{document}